\documentclass[sn-mathphys,Numbered]{sn-jnl}

\usepackage{graphicx}%
\usepackage{multirow}%
\usepackage{amsmath,amssymb,amsfonts}%
\usepackage{amsthm}%
\usepackage{mathrsfs}%
\usepackage[title]{appendix}%
\usepackage{xcolor}%
\usepackage{textcomp}%
\usepackage{manyfoot}%
\usepackage{booktabs}%
\usepackage{algorithm}%
\usepackage{algorithmicx}%
\usepackage{algpseudocode}%
\usepackage{listings}%
\usepackage{subcaption} 
\allowdisplaybreaks
\numberwithin{equation}{section}

\theoremstyle{thmstyleone}%
\newtheorem{theorem}{Theorem}
\newtheorem{proposition}[theorem]{Proposition}%

\theoremstyle{thmstyletwo}%

\theoremstyle{thmstylethree}%

\definecolor{greenM}{RGB}{0, 166, 99}
\definecolor{orangeM}{RGB}{255, 162, 0}
\definecolor{blueM}{RGB}{0, 129, 227}
\definecolor{violetM}{RGB}{158, 94, 206}
\newcommand{\sixpt}{\fontsize{6pt}{7.2pt}\selectfont}

\begin{document}

\title[Article Title]{Cyber risk modeling using a two-phase Hawkes process with external excitation}


\author[1]{\fnm{Alexandre} \sur{Boumezoued}}\email{alexandre.boumezoued@milliman.com}
\equalcont{These authors contributed equally to this work.}

\author[1,2]{\fnm{Yousra} \sur{Cherkaoui}}\email{yousra.cherkaouiTangi@ensae.fr}
\equalcont{These authors contributed equally to this work.}

\author[2]{\fnm{Caroline} \sur{Hillairet}}\email{caroline.hillairet@ensae.fr}
\equalcont{These authors contributed equally to this work.}

\affil[1]{\orgdiv{R\&D}, \orgname{Milliman}, \orgaddress{\street{14 Avenue de la Grande Armée}, \city{Paris}, \postcode{75017}, \country{France}}}

\affil[2]{\orgdiv{ENSAE Paris}, \orgname{CREST UMR 9194}, \orgaddress{\street{5 Avenue Henry Le Chatelier}, \city{Palaiseau}, \postcode{91120}, \country{France}}}


\abstract{
With the growing digital transformation of the worldwide economy, cyber risk has become a major issue. As 1\% of the world’s GDP (around \$1,000 billion) is allegedly lost to cybercrime every year, IT systems continue to get increasingly interconnected, making them vulnerable to accumulation phenomena that undermine the pooling mechanism of insurance. As highlighted in the literature, Hawkes processes appear to be suitable to capture contagion phenomena and clustering features of cyber events. This paper extends the standard Hawkes modeling of cyber risk frequency by adding external shocks, such as the publication of cyber vulnerabilities that are deemed to increase the likelihood of attacks in the short term. While the standard Hawkes model attributes all clustering phenomena to self-excitation, this paper introduces a model designed to capture external common factors that may explain part of the systemic pattern. This aims to provide a better quantification of contagion effects. We propose a Hawkes model with two kernels, one for the endogenous factor (the contagion from other cyber events) and one for the exogenous component (cyber vulnerability publications). We use parametric exponential specifications for both the internal and exogenous intensity kernels, and we compare different methods to tackle the inference problem based on public datasets containing features of cyber attacks found in the Hackmageddon database and cyber vulnerabilities from the Known Exploited Vulnerability database and the National Vulnerability Dataset. By refining the external excitation database selection, the degree of endogeneity of the model is nearly halved. We illustrate our model with simulations and discuss the impact of taking into account the external factor driven by vulnerabilities. Once an attack has occurred, response measures may be implemented to limit the effects of an attack. These measures include patching vulnerabilities and reducing the attack's contagion. We use an augmented version of the model by adding a second phase modeling a reduction in the contagion pattern from the remediation measures. Based on this model, we explore various scenarios and quantify the effect of mitigation measures of an insurance company that aims to mitigate the effects of a cyber pandemic in its insured portfolio.}




\maketitle

\section{Introduction}\label{intro}

As Information Technology (IT) systems become increasingly interconnected and networks closely interwined, cyber risk has emerged as one of the main threats to most of organizations worldwide. 
IT hygiene measures are undeniably essential but are unfortunately not always sufficient since no system is entirely proof against attacks. All devices have software vulnerabilities that might be exploited. The past few years which have been shaped by the Covid-19 health crisis on one hand and the war in Ukraine on the other, have shown that cyber attacks can cripple strategic sectors and thus cause significant losses. Insurance companies play a crucial role in providing a financial and IT protection to face the consequences of an attack. They are increasingly concerned about the contagious nature of this risk, which challenges the core principle of any insurance business: the pooling mechanism. In this paper, we focus on modeling the frequency of cyber attacks while taking into account this contagion phenomena.

From an insurer's perspective, several challenges in quantifying this risk arise and are inherent to the nature of cyber risk. Due to its global and contagious nature, the first natural question concerns its insurability. One can cite, for instance, Christian Biener et al. in \cite{biener2015insurability} where they argue that the insurability of cyber risk is limited by the lack of data, the difficulty of modeling, and the potential for catastrophic losses. Other work addresses the modeling questions raised by cyber risk. Yannick Bessy-Roland et al. in \cite{Multi-Variate_HawkesBessy} and Sébastien Farkas et al. in \cite{farkas2021cyber} address the modeling questions in a traditional insurance framework, where the modeling of frequency and severity is separated. To address the accumulative and contagious nature of cyber risk, the Poissonian framework is abandoned in \cite{Multi-Variate_HawkesBessy}, and a multivariate Hawkes process is shown to be suitable for modeling the frequency of cyber breaches from the PRC database. In \cite{farkas2021cyber} and using the number of breached records as an indicator of a data breach severity, the authors propose classification method based on regression trees for cyber claims. This criteria could be used as an insurability indicator. Other works study the spread of cyber attacks within a network, {see \cite{awiszus2023modeling} and \cite{fahrenwaldt2018pricing}}. For example in \cite{fahrenwaldt2018pricing}, the authors show that a star-shaped network is more likely to rapidly propagate cyber attacks than, for example, a homogeneous network. Other works study this spread using epidemiological SIR models such as in  \cite{hillairet2021propagation}.

Another study has investigated including this risk in the Solvency Capital Requirement (SCR) calculation. Notably, Martin Eling et al.'s work in \cite{eling2020capital} highlights that current regulatory frameworks underestimate cyber risk and proposes adopting more advanced models for accurate capital estimation. The study also reveals a strong correlation between cyber incident severity and frequency, impacting cyber risk insurance policy design.

Due to the presence of contagion and accumulation phenomena in cyber risk, Hawkes processes are proposed to model the frequency of cyber attacks in this paper. {For frequency modeling in insurance, the Poisson process is commonly used, see for example \cite{zeller2022comprehensive} and \cite{peng2017modeling} where the arrival of cyber attacks is modeled through an inhomogeneous marked Poisson process. Instead, using self-exciting processes such as Hawkes processes allows } the intensity $\lambda_t$ to capture the effect of the past events occurring at times $(T_n)_{n\geq 1}$ with some characteristics (or marks) $(Y_n)_{n \geq 1}$ .  The {marks $(Y_n)_{n \geq 1}$} indicate the level of contagiousness associated with the event. Since not all cyber attacks have the same potential for contagiousness, the higher the contagiousness of the event, the greater its impact on the overall intensity. The effect of the past  $(T_n)_{n\geq 1}$ is captured \color{black} through a positive kernel $\phi$ that shapes how these events contribute to the intensity $\lambda_t$, namely: 

\begin{equation}
\lambda_t = \lambda_0 + {\sum_{T_n<t}  \phi(t-T_n, Y_n) }.
\label{hawkes_standard}
\end{equation}

As new events occur, the intensity increases in the short term, thus leading to a higher influx of events, while for some longer times their impact vanishes, when the kernel $\phi$ is decreasing. This reflects the auto-excitation property in cyber risk along with the memory of past events - both features are well designed to captured the so-called clustering behavior of events. 
These processes were previously used in cyber risk modeling in the literature. Baldwin et al. suggest that the Hawkes framework is suitable for cyber frequency modeling since contagion is relatively well captured, see \cite{baldwin2017contagion}. Bessy-Roland et al. prove that a  multivariate \color{black}Hawkes process is particularly suitable to capture the clustering and autocorrelation of inter-arrival times of public data breaches from different attack types in multiple sectors, using data collected from the Privacy Rights Clearinghouse database, see  \cite{Multi-Variate_HawkesBessy}.

In addition to the chain reaction effect where one attack triggers another, cyber attacks can also be provoked by the exploitation of cyber vulnerabilities. The latter are weaknesses in systems or software that could be identified and used by malicious actors. {For example, the Kaseya attack in July 2021 exploited vulnerabilities in the Kaseya VSA software firm, allowing attackers to distribute ransomware that encrypted data on targeted systems. }In the cyber security field, one can refer to Su Zhang et al. in \cite{zhang2015predicting} for example where they tried to predict the time to next vulnerability for a given software, using the National Vulnerability Database which is a collection of all public vulnerabilities affecting various softwares. In the insurance field, Gabriela Zeller et al. in  \cite{zeller2023accumulation} address this systemic component of cyber risk through common vulnerabilities, the claims arrival is modeled as a Poisson process with a systemic vulnerability component. 

{In this paper and} in order to capture the contagion effect and the external excitation part coming from cyber vulnerabilities, an external excitation kernel denoted $\overline{\phi}$ is introduced in the expression of the intensity of the standard linear Hawkes process. This kernel $\overline{\phi}$ determines how external events arriving at times $(\overline{T}_k)_{k \geq 1}$ impact the intensity process. {Vulnerabilities can be distinguished depending on their severity, such as provided by the National Vulnerabily Database, allowing industries, organizations, and governments to classify the criticality and prioritize the remediation activities. As such, in the theoretical framework, we introduce the characteristics   $(\overline{Y}_k)_{k \geq 1}$ that could be an indicator of the severity levels of vulnerabilities}. This means that in (\ref{hawkes_standard}) we add the effect of the external event:

\begin{equation}
\lambda_t = \lambda_0 + {\sum_{\overline{T}_k<t}  \overline{\phi}(t-\overline{T}_k, \overline{Y}_k)} + {\sum_{T_n<t}  \phi(t-T_n, Y_n) }.
\label{hawkes_extern}
\end{equation}

This Hawkes model with external excitation has already been used in the litterature in other applications. For example, one can cite Angelos Dassios et al. \cite{dassios2011dynamic} for credit risk, Alexandre Boumezoued  \cite{boumezoued2016population} for population dynamics modeling, \cite{rambaldi2015modeling} and  \cite{rambaldi2018detection} in high frequency finance and exchange market and \cite{brachetta2022optimal} in reinsurance optimal strategy. To our knowledge, this Hawkes model with external excitation has not been used to model cyber attacks in the insurance framework. 

We provide hereafter a  numerical example to illustrate the estimation accuracy improvements for the Hawkes process regime when incorporating external events in the model. By taking deterministic characteristics and exponential kernels, the intensity is written  $\lambda_t = \lambda_0 + \sum_{T_i < t} m e^{-\delta (t-T_i)}+\sum_{\overline{T}_k < t}\overline{m}e^{-\delta (t-T_i)}$. Here, $(\overline{T}_k)_{k \geq 1}$ is a Poisson Process with intensity denoted $\rho$ and the endogeneity degree of the system (detailed in \ref{endo_Hawkes}) is captured here by the $L^1$ norm of $\phi$:
$\lVert \phi \rVert = \frac{m}{\delta}$.
In this example, and as illustrated in Table~\ref{hawkes_external_calibration_regime_switch}, we are in the subcritical regime ($\lvert \phi \rVert<1$). Notably, the contribution of external events is significantly higher when compared to internal events, since we took  $\rho$ quite important and $\overline{m}$ larger than $m$. 

A trajectory of the Hawkes process with external excitation is generated within a simulation horizon of $\tau = 3$ days. The calibration is carried out by maximizing the likelihood of the two processes, one using  (\ref{hawkes_extern}) and the other using the model without externalities (\ref{hawkes_standard}). The results are summarized in Table~\ref{hawkes_external_calibration_regime_switch}:

\begin{table}[h]
\caption{Calibration results for a Hawkes process with external excitation and a standard linear Hawkes process using the likelihood method}\label{hawkes_external_calibration_regime_switch}
\begin{tabular}{@{}lccccccc@{}}
\toprule
& $\boldsymbol{\lambda_0}$ & $\boldsymbol{\rho}$ & $\boldsymbol{\overline{m}}$ & $\boldsymbol{m}$ & $\boldsymbol{\delta}$ & $\boldsymbol{\tau}$ & $\boldsymbol{\lVert \phi \rVert}$ \\
\midrule
Simulated parameters & 5 & 40 & 10 & 0.5 & 0.7 & 3 days & 0.71 \\
Calibrated with external excitation & 3.04 & 39.67 & 12.69 & 0.35 & 0.51 & - & 0.69 \\
Calibrated without external excitation & 49.03 & - & - & 4.20 & 4.02 & - & 1.04 \\
\bottomrule
\end{tabular}
\end{table}

\newpage

Disregarding external events leads to an increased value of $\lambda_0$ and places us in the over-critical regime instead of the correct sub-critical regime. The regime shift is indicated by a change in the norm of the internal excitation kernel $\phi$, which increases from 0.71 to 1.04. On the other hand, when calibrating the Hawkes process with external excitation, the estimated norm of the internal excitation kernel is found to be 0.69, which is close to the true value 0.71.

This simulated example highlights the importance of taking into account the contribution of external events in the intensity of the Hawkes process. The estimation of the external excitation contribution is refined by differentiating between contributions from external events and the baseline intensity. Moreover, the appropriate regime is accurately identified by considering external events in the Hawkes process with external excitation. 

As soon as cyber attacks are detected, response measures can be initiated to mitigate the potential spread of the attack within a company's network. It is important to note that these containment measures are not solely undertaken by an insurer. Since we are in an insurance context, and for the sake of simplicity, we refer to the responding entity as the insurer, but it could be another entity, like the victim itself or an external cyber security firm. These response measures are meant to be applied within a fictional insurance portfolio, which is represented here by the Hackmageddon database. These could for example include  implementing response and prevention measures initiated by the insurer within these companies in the context of a cyber pandemic. The model that is used to perform simulations takes into account this reactive phase. At a specific deterministic time denoted as $\ell$, containment measures are implemented. These include cutting off the impact of external events $(\overline{T}_k)_k$ after time $\ell$. This ensures that no new cyber vulnerabilities arise during the reaction phase that could trigger an attack. The other reaction measures involve reducing the baseline intensity and the self-excitation component from past cyber attacks at time $\ell$ using reaction parameters $\alpha_0$ and $\alpha_1$ which represent preventive measures taken and the speed at which vulnerabilities previously exploited are patched. {For instance, this might involve implementing stronger password management practices for compromised systems, isolating infected computers, or restoring affected systems from backups}. Furthermore, given that vulnerabilities are expected to have been addressed during the second phase, attacks $(T_n)_n$ after time $\ell$ are expected to be less critical than before time $\ell$. This is captured in the model through their respectives characteristics (denoted as $(Y^{al}_n)_n$ and $(Y^{bl}_n)_n$), by assuming $\mathbb{E}[Y^{bl}_n] > \mathbb{E}[Y^{al}_n] $.

To summarize, the intensity of the counting process is the following: 

\begin{equation}
\lambda_t = \begin{cases}
\lambda_0 + {\sum_{\overline{T}_k<t}  \overline{\phi}(t-\overline{T}_k, \overline{Y}_k)} + {\sum_{T_n<t}  \phi(t-T_n, Y^{bl}_n) } & \text{if } t < \ell \\ 
\\ 
\alpha_0 \lambda_{0} + \alpha_1 (\lambda_{\ell^{-}} - \lambda_0) & \text{if } t = \ell \\ 
\\
\alpha_0 \lambda_0  + \alpha_1 (\lambda_{\ell^{-}} - \lambda_0) e^{-\delta(t-\ell)}   + {\sum_{\ell<T_n<t}  \phi(t-T_n, Y^{al}_n) }  & \text{if } t > \ell 
\end{cases}
\end{equation}

In a nutshell, the new phase is characterized by:
\begin{enumerate}
    \item Cutting off the impact of external shocks, by setting to zero the external kernel after time $\ell$,
    \item Modulating the baseline intensity $\lambda_0$ (through $\alpha_0$ parameter) and the self-excitation component from past attacks (through  $\alpha_1$ parameter),
    \item Changing the distribution of the characteristics of future attacks (after $\ell$).
\end{enumerate}

This two-phase (2P) Hawkes process has been used to model the Covid-19 pandemic in \cite{dassiostwo}, where the effect of lockdown measures have been studied in multiple countries. To our knowledge, it has not been used in cyber risk modeling. The choice of time $\ell$ is deterministic in our case to make calculations easier to handle. {This could be interpreted as the average time required to find a patch addressing a vulnerability: unlike \cite{dassiostwo}, where both phases are calibrated in the context of a Covid-19 pandemic modeling, we only calibrate the first phase of the process and construct scenarios considering the second phase since reactive measures in cyber risk differ in general from one entity to another. This choice is further detailed in Section~\ref{justif_1P}.}

\textit{Scope of this paper} In this paper, we propose a stochastic model incorporating external events and a reaction phase to model cyber attacks frequency, with the aim of capturing the contagion originating from external events. {We illustrate the importance of including external events. This allows us to reassess the degree of endogeneity, and we demonstrate that it is overestimated in the standard Hawkes model.} We also provide closed-form formulas for the expectation of the cyber attacks process, the proof is based on \cite{boumezoued2016population}. These formulas allow us to test the MSE calibration method used in \cite{dassiostwo} and to make closed-form predictions {for optimal reaction parameters selection} in a cyber pandemic simulation. The model is calibrated on the Hackmageddon database, which includes various types of attacks beyond data breaches. These attacks include hacking incidents that involve the exploitation of identified vulnerabilities that we connect to vulnerability databases. Three configurations of the external databases are used. The first one contains vulnerabilities exploited and extracted from the Hackmageddon database. The second one includes the Known Exploited Vulnerability (KEV) database which contains all vulnerabilities that were exploited in an attack including those exploited in the Hackmageddon database. The third one is the National Vulnerability Database (NVD) which is a more comprehensive list of all cyber vulnerabilities that includes the previous two sources. After obtaining the calibrated parameters, the intensity is decomposed and plotted into internal, external, and baseline components. Additionally, we investigate the optimal choice of reaction parameters, considering the constraints of a fictional insurer with a limited response capacity.

The paper is organized as follows. Section~\ref{section_model} provides a detailed explanation of the Hawkes process with external excitation and the two-phase dynamics. The model is presented along with simulated examples, as well as a reminder of main results in terms of the expected future number of cyber attacks in closed-form. Section~\ref{section_inference} first describes different calibration strategies for the first phase of the model, using either the Mean Square Error or the likelihood. The best approach for our problem is then selected based on simulated examples. In Section~\ref{inference_section}, the datasets used for both vulnerabilities and cyber attacks are introduced and the calibration results are detailed with a specific emphasis on the learning of the "endogeneity feature" of the model, showing that it is significantly over-estimated when considering a Hawkes model without external excitation. Finally, in Section~\ref{section_forecasting}, simulations and closed-form predictions from the model are performed, using the parameters estimated for the first phase and considering different scenarios regarding mitigation measures and the insurer reaction capacity to quantify the impact of reaction measures in the second phase.

\section{Hawkes process with external excitation}
\label{section_model}

A two-phase Hawkes process with external excitation is proposed to model the contagion of cyber attacks. This model incorporates the impact of external vulnerabilities on the attacks process intensity. In this section we present the model and we compute the conditional expectation of the cyber attacks process and the likelihood of the cyber attacks and vulnerabilities process.

\subsection{Model specification: Two-phase Hawkes process with exponential kernel and external excitation} \label{section_two_phase_model_description}

Let us recap the key features and definition of a standard linear Hawkes process $(N_t)_{t \ge 0}$. This point process is self-exciting and as all counting processes, it is characterized by its intensity function $(\lambda_t)_{t \ge 0}$. Recall that the intensity $\lambda_t$ represents the "instantaneous probability" of witnessing a jump at time $t$. The intensity includes a baseline rate $\lambda_0$ of spontaneous events denoted $T_n$ and a self-exciting part represented by the sum of all previous jump times $\sum_{T_n<t}$ $\phi(t - T_n{,Y_n})$, where $T_n$ is the jump time of the event $n^{th}$, {$Y_n$ its characteristic or mark,} and $\phi$ is a {positive excitation kernel} that determines how they contribute over time  to the intensity of the Hawkes process. This process is self-exciting because the arrival of a new event increases the intensity of the Hawkes process, leading to a higher probability of observing the arrival of another event. This clustering property is of particular interest in modeling cyber attacks, as it is observed in the case of contagious cyber attacks.

Cyber vulnerabilities can provide attackers with entry points to launch contagious cyber attacks. This implies that the emergence of new vulnerabilities can increase the probability of being attacked. To reflect this external excitation property, we propose a stochastic baseline intensity that incorporates $\sum_{\overline{T_k}<t}$ $\overline{\phi}(t - \overline{T}_k,{\overline{Y}_k})$, where $\overline{T_k}$ is the arrival time of the $k^{th}$ vulnerability{, $\overline{Y}_k$ its characteristic or mark,}  and $\overline{\phi}$ is the kernel function that reflects how vulnerabilities contribute to the intensity of the cyber attack process. The introduction of a new vulnerability increases the intensity of the Hawkes process through the sum of all previous vulnerability jump times. Since not all vulnerabilities have the same effect on IT systems, some being very critical and easily exploitable while others are less so, we introduce marks $\left\{\overline{Y}_k\right\}_{k=1,2, \ldots}$ that capture the heterogeneity in their impact on the intensity, meaning that the external kernel function $\overline{\phi}$ is a function of both $(t-\overline{T}_k)$, the age of the event $\overline{T}_k$ at time $t$, and $\overline{Y}_k$, the characteristic of this event.      

In the case of a contagious cyber attack, an insurer may face a claims accumulation and would need to {encourage} remedial and preventive measures, such as patching vulnerabilities. {The aim here is to capture claims accumulation in the insurer portfolio and to quantify how mitigation measures could reduce this risk.} The proposed model takes this into account by introducing a second phase starting at time $\ell$. During this second phase, the attacks are less contagious due to the preventive measures taken. To demonstrate their effects, marks for attacks are introduced and denoted $\left\{Y_i^{bl}\right\}_{i=1,2, \ldots}$ for $t$ before $\ell$ and $\left\{Y_i^{al}\right\}_{i=1,2, \ldots}$ when $t$ is above $\ell$. So, the kernel $\phi$ is a function on $(t-T_n)$, which is the age of the event $T_n$ at time $t$, and the contagiousness potential included in $Y_n^{bl}$ or $Y_n^{al}$ for this attack, depending on whether it occurred before or after $\ell$. 
To reflect the impact of protection measures, it is assumed that $m^{bl} = \mathbb{E}[Y_1^{bl}] < m^{al} = \mathbb{E}[Y_1^{al}]$ while also assuming that the marks {before $\ell$ as well as the marks after $\ell$} are independent and identically distributed.

In what follows, we denote $N_t = \sum_{i \ge 1} \mathbf{1}_{T_i \le t}$  the cyber attacks process and $\overline{N}_t = \sum_{k \ge 1} \mathbf{1}_{\overline{T}_k \le t}$ the  vulnerabilities Poisson process with  constant intensity $\rho$. Due to the diversity of computer vulnerabilities affecting various software systems, the assumption of independence of vulnerabilities inter-arrival times within Poisson processes seems to be reasonable. In the expression of the intensity $\lambda_t$, we take exponential kernels for $\phi$ and $\overline{\phi}$, as we assume that the effect of an attack is immediate, meaning that as soon as an attack arrives, it increases the intensity of the process $N_t$. While this assumption simplifies the model, it facilitates later calculations of the expectation. In what follows,  $\phi(a,x) = \overline{\phi}(a,x) = x \exp(- \delta a)$. For example, an event $(\overline{T}_k,\overline{Y}_k)$ arriving increases the intensity by $\overline{Y}_k$ and ages exponentially over time.    

Let $(\Omega, \mathcal{A}, \mathbb F=(\mathcal F_t)_{t \geq 0},\mathbb{P})$ be a filtered probability space satisfying the usual conditions where  
$\mathcal F_t:= \sigma( (T_k, Y_k) 1_{T_k\leq t }, (\overline{T}_k, \overline{Y}_k) 1_{\overline{T}_k\leq t })$ (as usual, one consider the right continuous version, augmented with the negligible sets). The $\mathbb F $-adapted  intensity of $N$  is expressed as:

\begin{equation}
\lambda_t = \begin{cases}
\lambda_0  + \sum_{\overline{T_k} < t}\overline{Y_k} e^{-\delta (t-\overline{T_k})} + \sum_{T_i < t}Y_i^{bl} e^{-\delta (t-T_i)} & \text{if } t < \ell \\ 
\\ 
\alpha_0 \lambda_{0} + \alpha_1 (\lambda_{\ell^{-}} - \lambda_0) & \text{if } t = \ell \\ 
\\
\alpha_0 \lambda_0  + \alpha_1 (\lambda_{\ell^{-}} - \lambda_0) e^{-\delta(t-\ell)}   +  \sum_{\ell<T_i < t}Y_i^{al} e^{-\delta (t-T_i)} & \text{if } t > \ell 
\end{cases}
\end{equation}

where: 
\begin{itemize}
    \item $\delta > 0$ is the constant rate of exponential decay.
    \item Self-excitation part:
    \begin{itemize}
        \item $\{Y_i^{bl}\}_{i=1,2,\ldots}$ is a sequence of iid positive jump sizes before $\ell$ with a cumulative distribution function $G^{bl}$ having support on $(0,\infty)$, and we denote $m^{bl} := \mathbb{E}[Y_1^{bl}]$. \label{m_bl}
        \item $\{Y_i^{al}\}_{i=1,2,\ldots}$ is a sequence of iid positive jump sizes after $\ell$ with a cumulative distribution function $G^{al}$ having support on $(0,\infty)$, and we denote $m^{al} := \mathbb{E}[Y_1^{al}]$.
    \end{itemize}
    \item External excitation part:
    \begin{itemize}
        \item $\{\overline{T_k}\}_{k=1,2,\ldots}$ are the arrival times of the homogeneous Poisson vulnerabilities process $\overline{N_t}$ with intensity $\rho > 0$.
        \item $\{\overline{Y_k}\}_{k=1,2,\ldots}$ is a sequence of iid positive jump sizes with distribution function $\overline{F}$ with support on 
        $(0,\infty)$. $\overline{m} := \mathbb{E}[\overline{Y}_1]$.
    \end{itemize}
    \item Reaction parameters:
    \begin{itemize}
        \item $(\alpha_0, \alpha_1) \in [0,1]^2$ are reaction parameters, with $(\alpha_0, \alpha_1) \ne (0,0)$ to avoid the degenerate case where no new events occur after time $\ell$. $\alpha_0$ corresponds to the impact of preventive measures, while $\alpha_1$ represents the effect of patching measures against vulnerabilities and the mitigation of contagiousness for attacks that occured prior to the initiation of the second phase.
    \end{itemize}
    \item The sequences $\{\overline{Y_k}\}_{k=1,2,\ldots}$, $\{\overline{T_k}\}_{k=1,2,\ldots}$, $\{Y_i^{al}\}_{i=1,2,\ldots}$, and $\{Y_i^{bl}\}_{i=1,2,\ldots}$ are assumed to be independent of each other.
\end{itemize}

\subsection{Expectation of the 2P Hawkes process with external excitation} \label{closed_formulas_expectation}

In Proposition \ref{prop3}, we provide the expression for the expected value of the two-phase Hawkes process with external excitation.

\begin{proposition}
\label{prop3}
    The conditional expectation of  $N_t$ given $\mathcal{F}_s $ for $0<s<t<\ell$ is:
    \begin{equation}
    \label{E_N_before_l}
    \mathbb{E}\left[N_{t} \bigm| \mathcal{F}_s \right] = 
\begin{cases}
N_{s}+\lambda_{s}(t-s)+\frac{1}{2} (\rho \overline{m} + \delta \lambda_0) (t-s)^{2} & \text {if }  \delta=m^{bl}
\\ 
N_{s}+\frac{(\rho \overline{m} + \delta \lambda_0) }{\delta - m^{bl}} (t-s)+\left(\lambda_{s}-\frac{\rho \overline{m} + \delta \lambda_0 }{\delta - m^{bl}}\right) \frac{1-e^{-(\delta - m^{bl})(t-s)}}{\delta - m^{bl}} & \text {if } \delta \neq m^{bl} 
\end{cases} 
    \end{equation}
    The conditional expectation of the process $N_t$ given $\mathcal{F}_s $ for $0<\ell<s<t$  is:
\begin{equation}
\mathbb{E}\left[N_{t} \bigm| \mathcal{F}_s \right] = 
\begin{cases}
N_{s}+\lambda_{s}(t-s)+\frac{1}{2} \alpha_0 \delta \lambda_0 (t-s)^{2} & \text {if }  \delta=m^{bl}
\\ 
N_{s}+\frac{\alpha_0 \delta \lambda_0 }{\delta - m^{bl}} (t-s)+\left(\lambda_{s}-\frac{\alpha_0 \delta \lambda_0 }{\delta - m^{bl}}\right) \frac{1-e^{-(\delta - m^{bl})(t-s)}}{\delta - m^{bl}} & \text {if } \delta \neq m^{bl} 
\end{cases} 
\end{equation}

The conditional expectation of $N_t$ given $\mathcal{F}_s $ for $0<s<\ell<t$ is:

\begin{equation}
\begin{aligned}
\mathbb{E}\left[N_{t} \bigm| \mathcal{F}_s \right] = 
\begin{cases}
\begin{aligned}
  \mathbb{E}\left[N_{\ell} \bigm| \mathcal{F}_s \right] + & \frac{\alpha_0 \delta \lambda_0}{2} (t-\ell) ^2 + \lambda_0 (\alpha_0 - \alpha_1) (t-\ell) \\ + & \alpha_1 \mathbb{E}\left[\lambda_{\ell^{-}}  \bigm| \mathcal{F}_s \right] (t-\ell),  
\end{aligned}
 & \text{if }  \delta = m^{al} 
\\ 

\begin{aligned}
  \mathbb{E}\left[N_{\ell} \bigm| \mathcal{F}_s \right] + &
\frac{\alpha_0 \delta \lambda_0}{\delta - m^{al}} (t-\ell)  \\ + & \left( (\alpha_0 - \alpha_1) \lambda_0 +  \alpha_1 \mathbb{E}\left[\lambda_{\ell^{-}} \bigm| \mathcal{F}_s \right] - \frac{\alpha_0 \delta \lambda_0}{\delta - m^{al}} \right)  \\ & \frac{1}{(\delta - m^{al})}(1 - e^{- (\delta - m^{al})(t- \ell)}),    
\end{aligned}
& \text{if } \delta \ne  m^{al}
\end{cases} 
\end{aligned}
\label{E_N_after_l}
\end{equation}

with  
\begin{equation}
\mathbb{E}\left[\lambda_{\ell^{-}} \bigm| \mathcal{F}_s \right] = 
\begin{cases}
\lambda_s + (\delta \lambda_0  + \rho \overline{m}) (\ell-s) & \text {if }  \delta=m^{bl}
\\ 
\frac{\rho \overline{m} + \delta \lambda_0}{\delta - m^{bl}} + (\lambda_s - \frac{\rho \overline{m} + \delta \lambda_0}{\delta - m^{bl}} )e^{- (\delta - m^{bl})(\ell- s)} & \text {if } \delta \neq m^{bl} 
\end{cases} 
\end{equation}

and

\begin{equation}
\mathbb{E}\left[N_{\ell} \bigm| \mathcal{F}_s \right] = 
\begin{cases}
N_{s}+\lambda_{s}(\ell-s)+\frac{1}{2} (\rho \overline{m} + \delta \lambda_0) (\ell-s)^{2} & \text {if }  \delta=m^{bl}
\\ 
N_{s}+\frac{(\rho \overline{m} + \delta \lambda_0) }{\delta - m^{bl}} (\ell-s)+\left(\lambda_{s}-\frac{\rho \overline{m} + \delta \lambda_0 }{\delta - m^{bl}}\right) \frac{1-e^{-(\delta - m^{bl})(\ell-s)}}{\delta - m^{bl}} & \text {if } \delta \neq m^{bl} 
\end{cases} 
\end{equation}    
\end{proposition}

\begin{proof}
 The proof is detailed in Appendix \ref{annex_a}.
\end{proof}

\subsection{Illustration on a simulated example}

\subsubsection{A simulated trajectory}

We use the thinning algorithm, described in {\cite{ogata1981lewis}}, to simulate the trajectories of the 2P Hawkes process. In this example, the simulation horizon $\tau$ is set to 10 days. The other simulation parameters are detailed in Table~\ref{sim_parameters_expect} below: 

\begin{table}[h]
\caption{Simulation parameters}\label{sim_parameters_expect}
\begin{tabular}{@{}cccccccccc@{}}
\toprule
$\boldsymbol{\lambda_0}$ & $\boldsymbol{\rho}$ & $\boldsymbol{\overline{m}}$ & $\boldsymbol{m^{bl}}$ & $\boldsymbol{\delta}$ & $\ell$ & $\boldsymbol{m^{al}}$ & $\boldsymbol{\alpha_0}$ & $\boldsymbol{\alpha_1}$ & $\boldsymbol{\tau}$ \\
\midrule
0.6 & 0.2 & 0.8 & 0.5 & 1.5 & 3 days & 0.25 & 0.8 & 0.5 & 10 days \\
\botrule
\end{tabular}
\end{table}

\newpage

A trajectory of this process is illustrated in Figure {\ref{fig_all} below. In Figure~\ref{traj_int}, a trajectory of the intensity $\lambda_t$ is plotted over time. There are two different curves plotted on this graph. The green curve represents the intensity of a one-phase Hawkes process, which does not involve any intervention phase. On the other hand, the blue curve represents the intensity of a two-phase Hawkes process, which includes a reaction phase that starts at $\ell = 3$ days. Figure~\ref{traj_count_process} illustrates the counting process that corresponds to the trajectory shown in Figure~\ref{traj_int} and in \ref{illustration_simulated_Hawkes_process_cum}, we display the cumulative number of observed attacks on each day}. 

\begin{figure}[H]%
\centering
\begin{subfigure}{0.3\textwidth}
    \centering
    \includegraphics[scale=0.4]{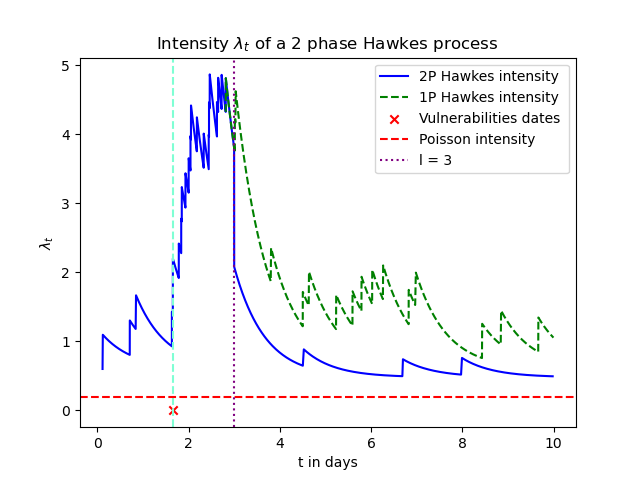}
    \caption{Intensity of a 2P Hawkes process}
    \label{traj_int}
\end{subfigure}
\hfill
\begin{subfigure}{0.45\textwidth}
    \centering
    \includegraphics[scale=0.4]{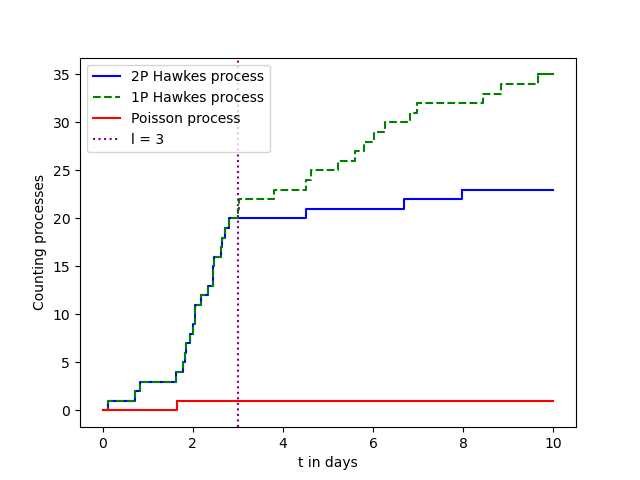}
    \caption{Counting process of a simulated 2P Hawkes process}
    \label{traj_count_process}
\end{subfigure}
\hfill
\begin{subfigure}{\textwidth}
    \centering
    \includegraphics[scale=0.4]{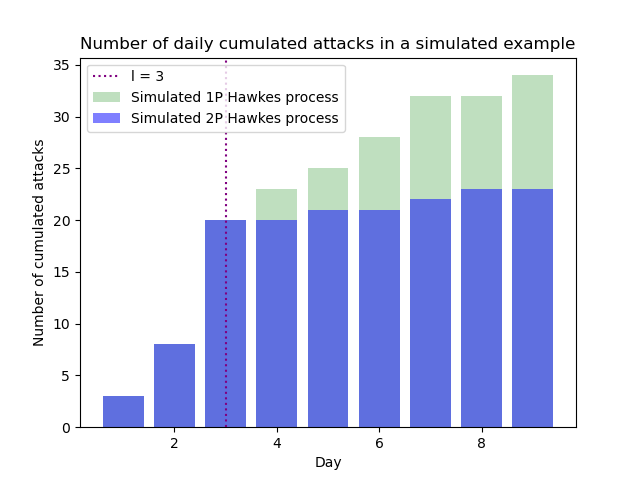}
    \caption{Number of cumulated number of attacks in a simulated example}
    \label{illustration_simulated_Hawkes_process_cum}
\end{subfigure}
\caption{Illustration of a trajectory of a simulated 2P Hawkes process using the thinning algorithm}
\label{fig_all}
\end{figure}

In this trajectory, the arrival of a vulnerability triggers a series of cyber attacks. As these attacks accumulate, the intensity of the process increases, reflecting the self-excited nature of the Hawkes process. This accumulation phase continues until a specific time point $\ell = 3$ days is reached. 

At time $\ell$, the second phase of the process is initiated by an intensity reduction in the intensity. Reduced intensity and smaller $m^{al}$ indicate a decrease in contagiousness. In this trajectory, 3 attacks are observed after the initiation of a second phase. However, if the second phase had not been triggered, 14 attacks would have been observed in this trajectory. This example illustrates the impact of adequate reaction measures implemented in the second phase of the model.

\subsubsection{Analysis of the
expected number of attacks using closed-form formula}

We recall in the table below the parameters of the Hawkes process under study: 

\begin{table}[h]
\begin{tabular}{@{}cccccccccc@{}}
\toprule
$\boldsymbol{\lambda_0}$ & $\boldsymbol{\rho}$ & $\boldsymbol{\overline{m}}$ & $\boldsymbol{m^{bl}}$ & $\boldsymbol{\delta}$ & $\ell$ & $\boldsymbol{m^{al}}$ & $\boldsymbol{\alpha_0}$ & $\boldsymbol{\alpha_1}$ & $\boldsymbol{\tau}$ \\
\midrule
0.6 & 0.2 & 0.8 & 0.5 & 1.5 & 3 days & 0.25 & 0.8 & 0.5 & 10 days \\
\botrule
\end{tabular}
\end{table}

In Figure~\ref{expect_1P_2P} below, the expected number of attacks for both 1P (in green) and 2P (in blue) processes is illustrated  using the closed formulas given in (\ref{E_N_before_l}) and in (\ref{E_N_after_l}) with $s = 0$, unlike the previous graph \ref{illustration_simulated_Hawkes_process_cum} where the thinning algorithm was used. {Using closed-form formulas enables faster computations compared to the thinning algorithm.}

\begin{figure}[H]
\centering
\includegraphics[scale=0.4]{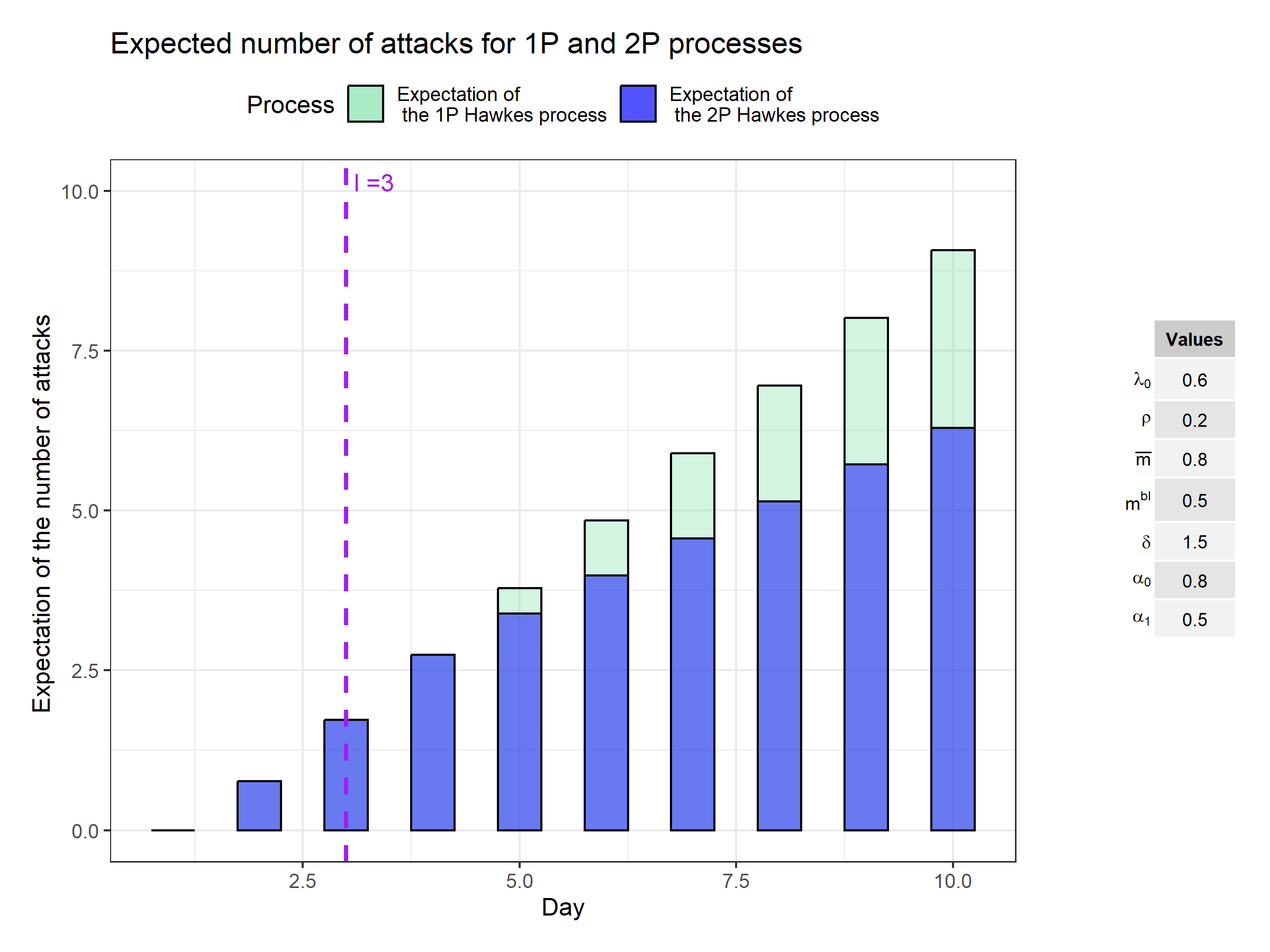}
\caption{Expectation of the 1P and 2P Hawkes process}
\label{expect_1P_2P}
\end{figure}

By introducing a second phase where the number of external events is cut off, the expected number of attacks decreases shortly after the initiation of the second phase of the process. In the following, the sensitivity of the model to the parameters $\alpha_0$ and $\alpha_1$ in terms of the expected number of attacks is illustrated.

\begin{figure}[H]
    \begin{subfigure}{0.45\textwidth}
    \includegraphics[scale=0.3]{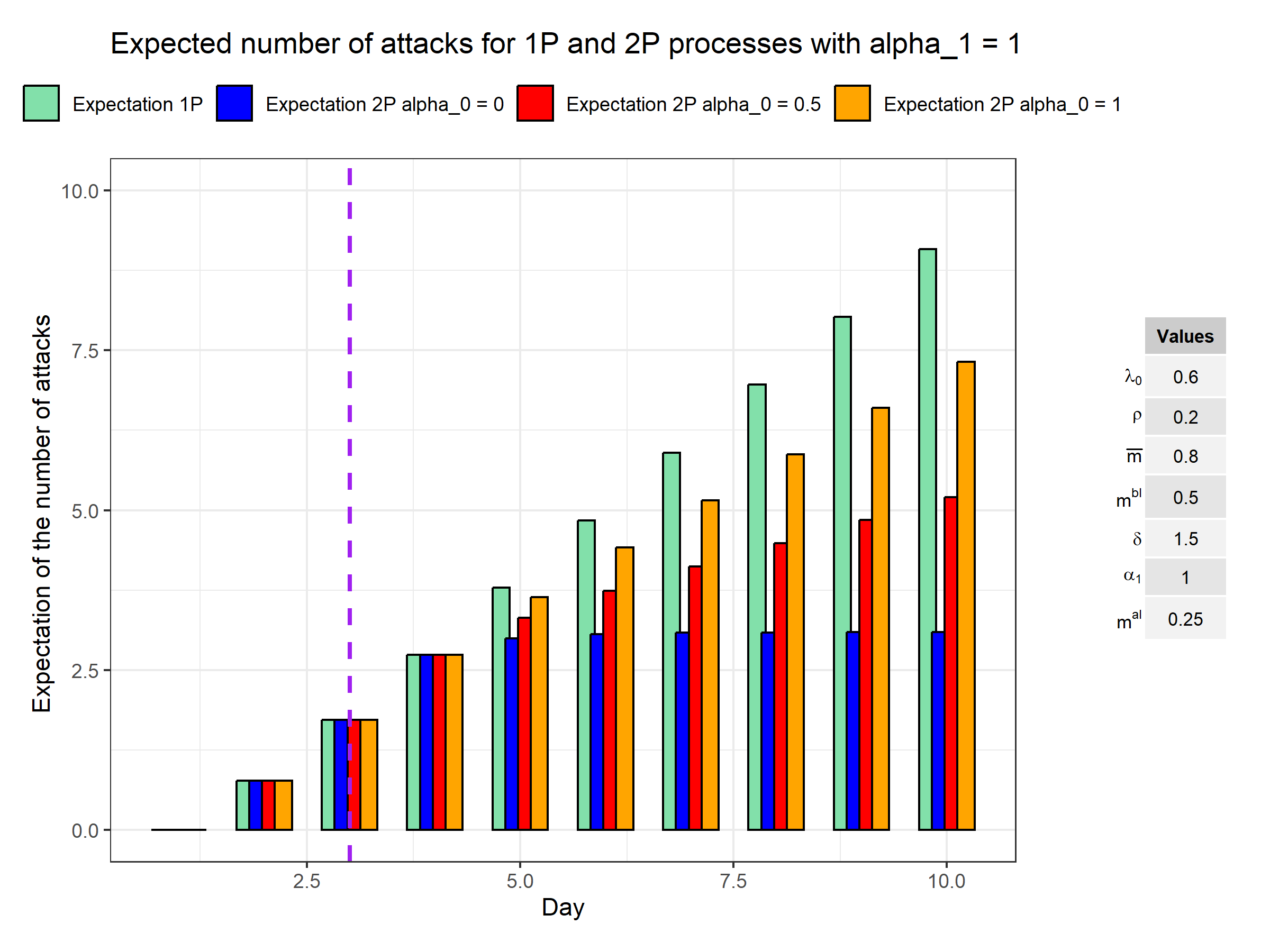}
    \caption{Impact of $\alpha_0$ parameter with $\alpha_1 = 1$}
    \label{impact_alpha_0}
    \end{subfigure}
    \hfill
    \begin{subfigure}{0.45\textwidth}
     \centering
        \includegraphics[scale=0.3]{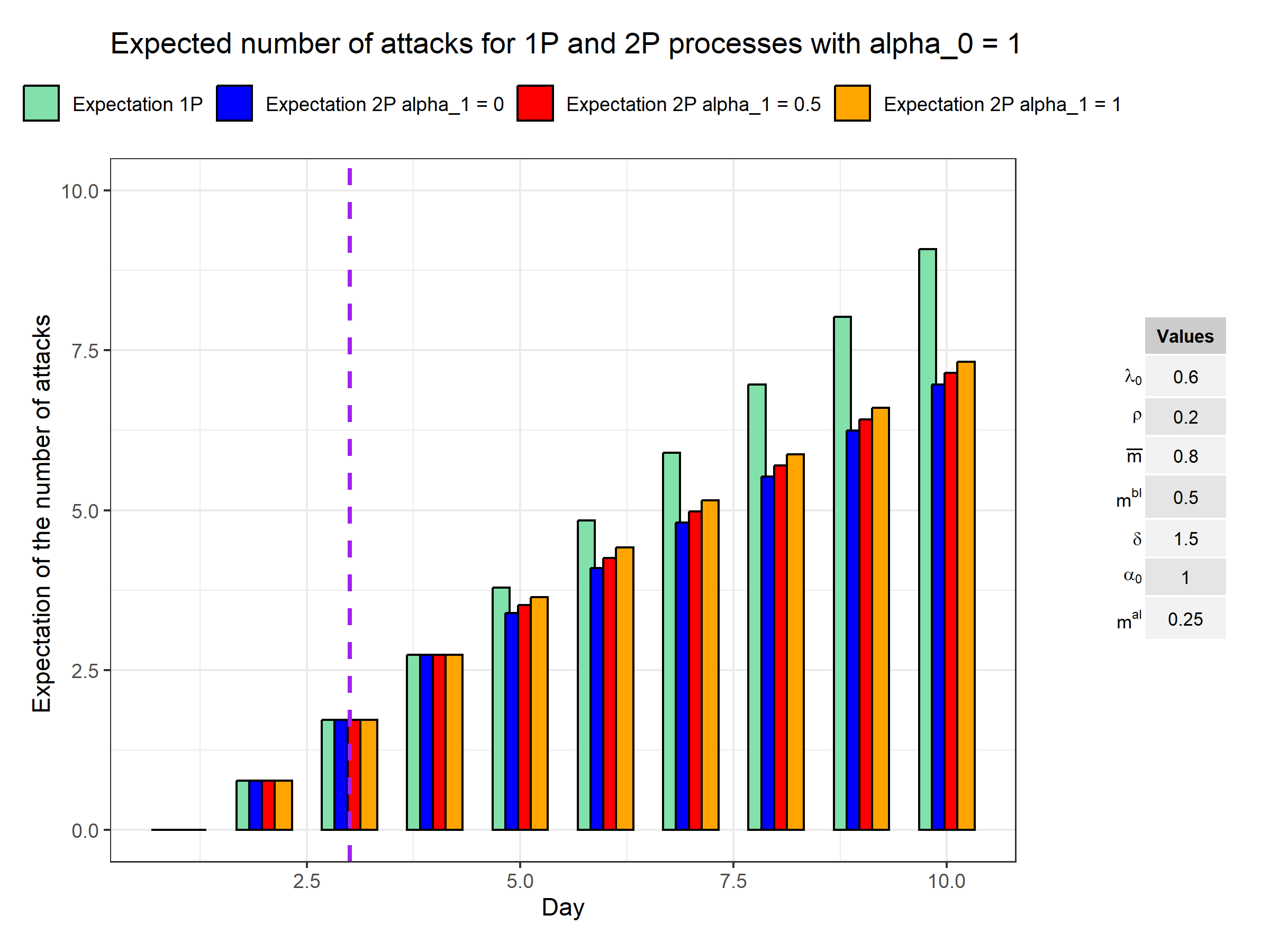}
        \caption{Impact of $\alpha_1$ parameter with $\alpha_0 = 1$}
        \label{impact_alpha_1}
    \end{subfigure}
     \hfill
    \begin{subfigure}{\textwidth}
     \centering
    \includegraphics[scale=0.3]{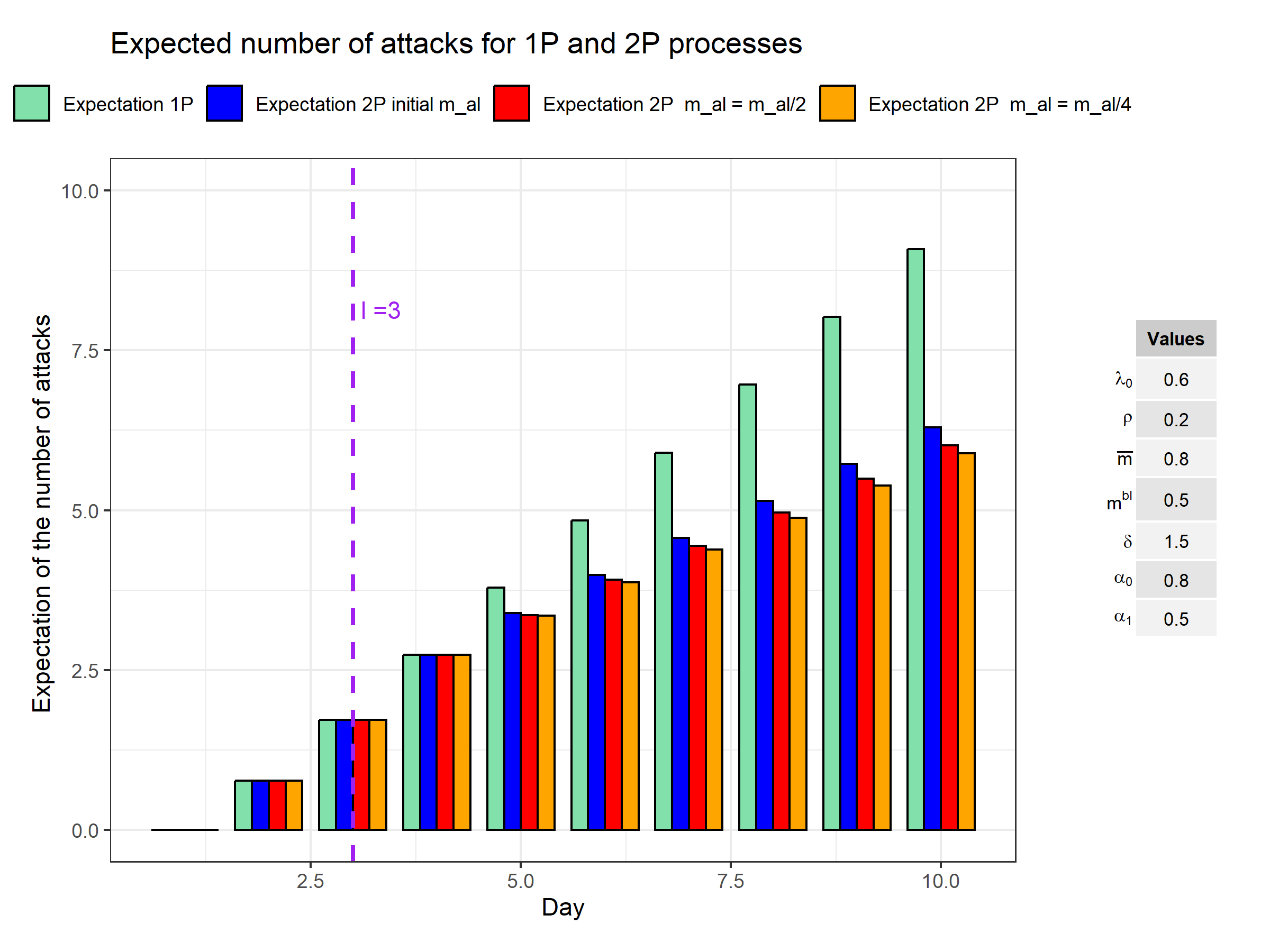}
    \caption{Impact of $m^{al}$ parameter with $\alpha_0 = 0.8$ and $\alpha_1 = 0.5$}
    \end{subfigure}
   \caption{Impact of reaction parameters on the expectation of the 2P Hawkes process }
\end{figure}

In Figure~\ref{impact_alpha_0}, $\alpha_1 = 1 $ (meaning that the contagion effect of the past attacks and vulnerabilities do not decrease) and $\alpha_0$ is varied. Conversely, in Figure~\ref{impact_alpha_1}, $\alpha_0 = 1 $ (meaning that the baseline intensity is not decreased) and $\alpha_1$ is varied. The other parameters remain unchanged.

For this parameters configuration (see Table~\ref{sim_parameters_expect}),
decreasing  $\alpha_0$ has a stronger impact on the number of expected attacks at time $t$ compared to that of decreasing $\alpha_1$: at each time step the three bars corresponding to the variation of $\alpha_1$ are close to one another compared with the three bars corresponding to the variation of $\alpha_0$. {This is likely attributed to the fact that $\alpha_1$ is modulated by the exponential decay. In this configuration, the $\delta$ parameter is set at 1.5, it is quite strong and therefore it can explain this empirical result. We observe in Figure~\ref{fig_impact_alpha_1} that a smaller decay rate of 0.6 leads to a more pronounced impact of $\alpha_1$.}
\begin{figure}[H]
    \begin{subfigure}{0.45\textwidth}
     \centering
    \includegraphics[scale=0.3]{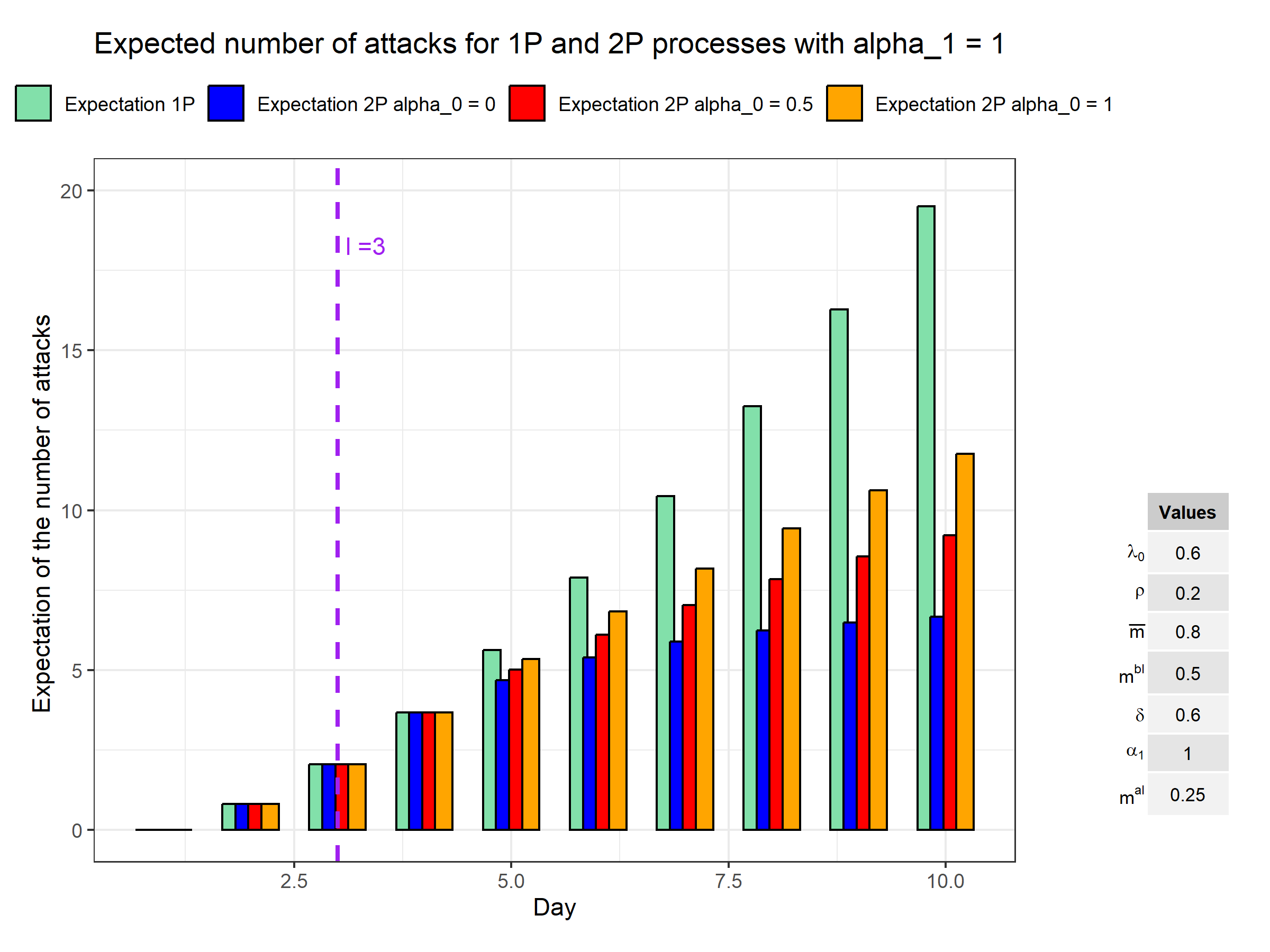}
    \caption{Impact of $\alpha_0$ parameter with $\alpha_1 = 1$}
    \end{subfigure}
    \hfill
    \begin{subfigure}{0.45\textwidth}
     \centering
        \includegraphics[scale=0.3]{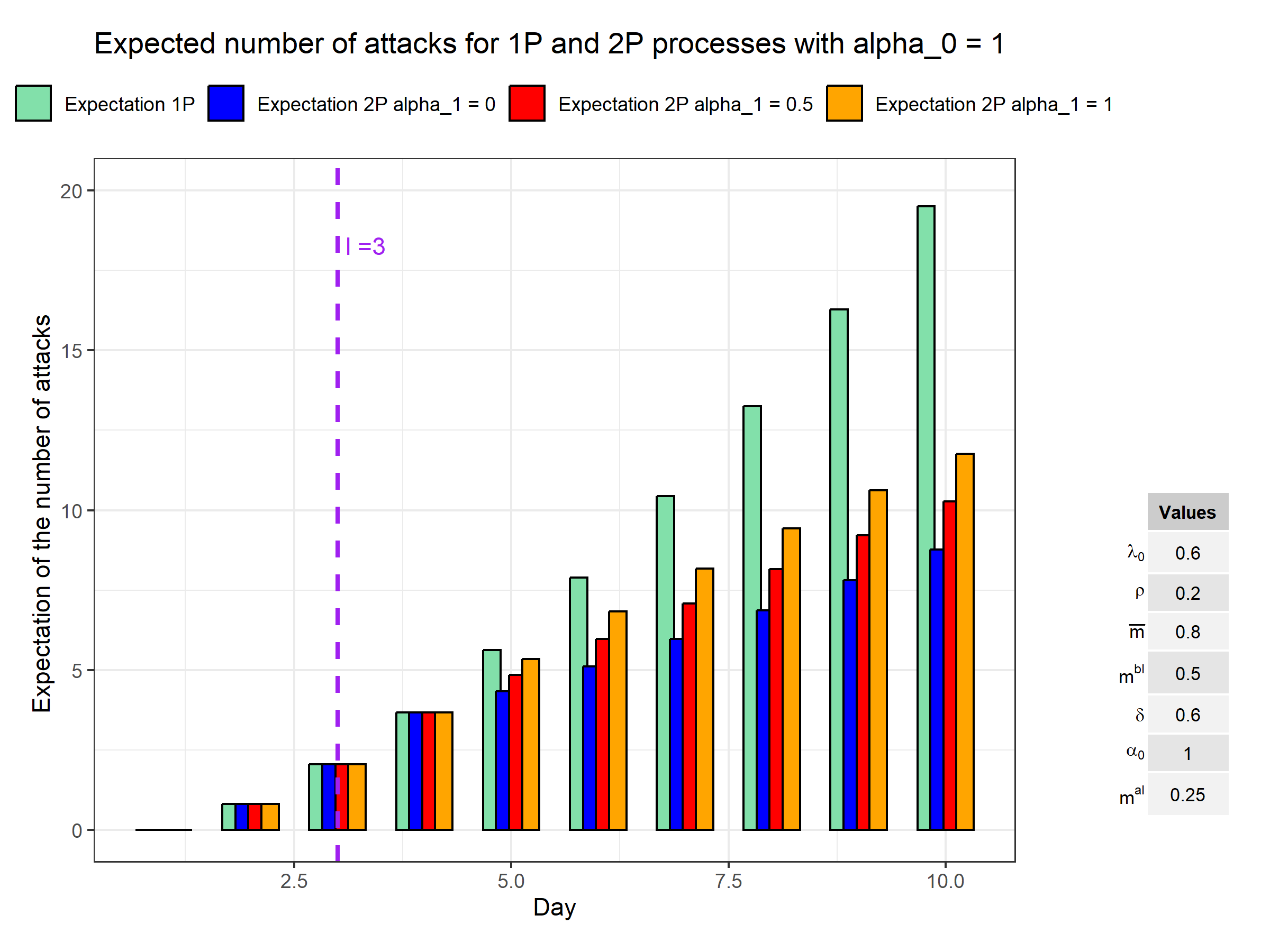}
        \caption{Impact of $\alpha_1$ parameter with $\alpha_0 = 1$}
        \label{fig_impact_alpha_1}
    \end{subfigure}
    \caption{Impact of $\alpha_0$ and $\alpha_1$ parameters for $\delta = 0.6$}
\end{figure}

\vspace*{-0.8cm}
By decreasing the decay parameter $\delta$ and adjusting the $\alpha_1$ parameter, we can effectively control the contagiousness of events that occurred before $\ell$. By selecting a small $\alpha_1$ value, we can better demonstrate the impact of this parameter on the intensity and contagiousness of cyber attacks. This allows us to study the effectiveness of different response and preventive measures tailored to the attack's profile and level of contagiousness.


\section{Choice of parametric inference method of the non-stationary one-phase Hawkes process on simulated data}
\label{section_inference}

\subsection{Calibration strategies}

Based on simulated data, we explore in the following the performance of different methods for the calibration of the Hawkes model with external excitation. We consider a time window $[0,\tau]$ for all observations in the dataset (both vulnerabilities and attacks). We then consider a time $s\in [0,\tau]$ at which the inference starts, i.e. the parameters are estimated to explain at best the process dynamics over the period $[s,\tau]$, conditional on information up to time $s$. We assume that the model is in its first phase {as data about the second phase is not observable, as detailed in Section~\ref{justif_1P}. To simplify notation, what was previously denoted as $m^{bl}$ in Section~\ref{m_bl} within the characterisation of the marks is now simply denoted as $m$, since the second phase is not calibrated. Hereafter, we consider a simplified version of $\lambda_t$  with deterministic marks: $\lambda_t = \lambda_0  + \sum_{\overline{T_k} < t}\overline{m} e^{-\delta (t-\overline{T_k})} + \sum_{T_i < t}m e^{-\delta (t-T_i)} $}, hence the parameters to estimate are the following: 
$$
(\lambda_0,\rho,\overline{m},m,\delta).
$$

The first approach {($\text{MSE}^\text{Int}$ in Table~\ref{comparison_criteria})}, from \cite{dassiostwo}, relies on the Mean Square Error (MSE) to minimize the distance between the observed number of cyber attacks (internal jumps) and the expected number of attacks as predicted by the model. Let $n(t)$ be the observed cumulative number of cyber attacks up to time $t$; then the $\text{MSE}^\text{Int}$ computed on observation intervals of length $\Delta$ is defined as:
\begin{equation*}
\text{MSE}^\text{Int} = \frac{1}{\tau - s }\sum_{k=0}^{\lfloor \frac{\tau - s}{\Delta} \rfloor-1} \left\{\mathbb{E} \left[ N_{s +(k+1)\Delta } - N_{s+k \Delta}  \bigm|   \mathcal{F}_s \right] -  (n( s +(k+1)\Delta ) - n(s+k \Delta))\right\}^2,    
\end{equation*}
where we assume that the start time $s$ and the end time $\tau$ for the inference are taken on the time grid with step $\Delta$.

The advantage of this approach is that it requires few observations, i.e. only the number of cyber attacks at the discrete times $s + k \Delta$, while the conditional expected number of attacks can be computed in closed-form in the model as derived in Section~\ref{closed_formulas_expectation}. In particular, it does not require the knowledge of the exact times of internal jumps, which is valuable here as we know that there may be uncertainty around the precise date of the attack. Also, it does not require information about external jumps, as in \cite{dassiostwo} regarding the Covid-19 use case, where external shocks are features of the model that are not directly interpretable. However, as we observe in the numerical experiment that follows, the $\text{MSE}^\text{Int}$ approach on internal jumps only is not directly identifiable. This is due to a balance between parameters that will provide the same expected values of number of attacks, at least at some points in time. From Equation (\ref{E_N_before_l}), this is particularly true within the external dynamics when the product between the frequency and severity of external shocks $\rho \overline{m}$ appears, as well as for parameters between internal and external shocks so that the ratio $\frac{(\rho \overline{m} + \delta \lambda_0) }{\delta - m}$ is preserved.
{This ratio $\frac{(\rho \overline{m} + \delta \lambda_0) }{\delta - m} = \frac{\rho \lVert \overline{\phi} \rVert + \lambda_0}{1 - \lVert \phi \rVert} $ can also be interpreted as an ergodicity parameter, as detailed in Section~\ref{inference_results_section}.}

An attempt to overcome the identifiability issue is proposed in the {second approach ($\text{MSE}^\text{Ext}$ in Table~\ref{comparison_criteria})} where the MSE is augmented with the distance between predicted and observed number of external shocks. In our approach, the external shocks are directly interpretable since they do represent occurrences related to the disclosure of vulnerabilities. Although there may be uncertainty on the precise date, especially reporting delays, it can be argued that the number of vulnerabilities per time interval is still a useful source of information. Let $\bar{n}(t)$ be the observed cumulative number of vulnerabilities up to time $t$; then the augmented MSE reads:
\begin{align*}
\text{MSE}^\text{Ext} & = \frac{1}{\tau - s }\sum_{k=0}^{\lfloor \frac{\tau - s}{\Delta}\rfloor-1} \left\{\mathbb{E}[N_{s +(k+1)\Delta } - N_{s+k \Delta}  \bigm|   \mathcal{F}_s] -  (n( s +(k+1)\Delta ) - n(s+k \Delta))\right\}^2  \\
&\quad
+ \left\{
\mathbb{E}\left[\overline{N}_{s +(k+1)\Delta} - \overline{N}_{s+k \Delta} \bigm|   \mathcal{F}_s \right]-\left(\overline{n}(s +(k+1)\Delta) - \overline{n}(s+k \Delta ) \right)
\right\}^2,  
\end{align*}
where we note that $\mathbb{E}\left[\overline{N}_{s +(k+1)\Delta} - \overline{N}_{s+k \Delta} \bigm|   \mathcal{F}_s \right]=\rho \Delta$. From there, the virtue of this approach appears since it is expected to better control the external frequency parameter $\rho$ driving the vulnerabilities occurrences. However, one drawback remains, which lies in the fact that the ratio $\frac{\rho \lVert \overline{\phi} \rVert + \lambda_0}{1 - \lVert \phi \rVert}$ is still uncontrolled. For example, while $\rho$ is fixed, there may be compensation between $\bar{m}$, $m$ and $\delta$.

Finally, we consider a third approach {(Likelihood in Table~\ref{comparison_criteria})} to model inference based on maximizing the likelihood, which can be derived in a recursive manner for the point process of interest. Precise times of both attack occurrences and vulnerability dates are used here, as opposed to the methods based on the MSE which only require the number of attacks or vulnerabilities per time interval. 

The likelihood is written:
\begin{equation} 
    \mathcal{L}=\exp \left(-\int_s^\tau \lambda_u \mathrm{~d} u\right) \prod_{n=(N_{s}+1)}^{N_{\tau}}\hspace{-4mm}(\lambda_{t_n}) \hspace{4mm}\rho^{(\overline{N}_{\tau} - \overline{N}_{s}) } \hspace{1mm} \exp \left(-\rho\left(\overline{t}_{\overline{N}_{\tau}}-\overline{t}_{\overline{N}_s}\right)\right)
    \label{likelihood_equ2}
\end{equation}
where $t$ and $\bar{t}$ respectively denote the observed times of attacks and vulnerability discoveries. The likelihood computation is detailed in Appendix A in \cite{rambaldi2015modeling}.
To some extent, the assumptions about data quality for the likelihood approach is stronger than the MSE based methods. We note that Bayesian methods do exist to tackle the inference of Hawkes processes with censored and / or incomplete data, see e.g. \cite{shelton2018hawkes} and \cite{boyd2023inference}, and references therein - we believe that these methods could be used to allow for uncertainty regarding the quality of precise time occurrences in the datasets (for both cyber attacks and vulnerability discoveries); this is left for further research.  

Overall, the advantage of the maximum likelihood method is to capture the full statistical properties of the observed sample, and to solve identifiability issues, as shown whith a practical example in Section~\ref{inference_results_section}.

The simulation parameters are given in Table~\ref{sim_parameters_table}:

\begin{table}[h]
\caption{Simulation parameters}\label{sim_parameters_table}
\begin{tabular}{@{}ccccccc@{}}
\toprule
$\boldsymbol{\lambda_0}$ & $\boldsymbol{\rho}$ & $\boldsymbol{\overline{m}}$ & $\boldsymbol{m}$ & $\boldsymbol{\delta}$ & $\tau$ & s \\
\midrule
0.6 & 0.2 & 0.8 & 0.5 & 1.5 & 1095 days & 548 days \\
\botrule
\end{tabular}
\end{table}

The Nelder-Mead algorithm of the Python package Scipy is used. The execution times correspond to an Intel Weon 4310 CPU server operating at a 2.1 GHZ frequency, equipped with 24 processors and 100 GB of memory.

\subsection{Stationarity and endogeneity of the Hawkes process with external excitation} \label{endo_Hawkes}

By analogy with population theory, the Hawkes process with external excitation, also called Hawkes process with general immigrants, defines the following dynamics, see \cite{boumezoued2016population}:
\begin{itemize}
\item External immigrants arrive with rate $\rho$ over time, and have age 0 ; they give birth to individuals (called internal immigrants) in the Hawkes population with birth rate {$\overline{\phi}(a)=\overline{m} e^{-\delta a}$} while their age $a$ increases with time.
\item In addition from births described above, the Hawkes population is augmented by: (1) additional internal immigrants in the Hawkes population (direct immigration process), with rate $\lambda_0$, and (2) endogenous births, where all individuals with any age $a$ in the Hawkes population give birth to new individuals with birth rate {$\phi(a)=m e^{-\delta a}$}.
\end{itemize}
The latter component (endogenous births) drives the degree of endogeneity of the system. To further describe this component, let us notice that any individual in the Hawkes population will give birth to $\lVert \phi \rVert$ individuals over its entire lifetime. Then, each of these (say, children), gives also birth to $\lVert \phi \rVert$ new individuals in total at the next generation, so the number of grand-children is $\lVert \phi \rVert^2$, etc, so that for each individual in the Hawkes population, a so-called cluster of individuals is created with size below, which is finite under the stationary case $\lVert \phi \rVert < 1$:
$$
\sum_{n=1}^{\infty} \lVert \phi \rVert^n = \frac{\lVert \phi \rVert}{1- \lVert \phi \rVert}.
$$
Then, for a given internal immigrant (either coming from births of external immigrants, or from the direct immigration process with rate $\lambda_0$), the ratio of cluster size versus total average population (immigrant and its cluster), also called branching ratio, is given by:
$$
\frac{\frac{\lVert \phi \rVert}{1- \lVert \phi \rVert}}{1+\frac{\lVert \phi \rVert}{1- \lVert \phi \rVert}} = \lVert \phi \rVert.
$$

As such, $\lVert \phi \rVert$ is a direct measure, under the stationary case, of the degree of endogeneity of the system, which explains the popularity of the branching ratio in the analysis of dynamical systems e.g. in geophysics, finance, neurobiology and social behavior, see e.g. \cite{hardiman2014} and references therein. While the branching ratio has a nice interpretation in the stationary case (i.e., when smaller than 1, and considering the limiting distribution of the Hawkes process), however when the Hawkes process is analyzed in the short term (i.e. when the limiting distribution is not attained), this information can be completed by non-stationary measures. To this aim, we introduce the share of the intensity due to internal excitation, as captured by $\lambda_t^{int}$ and $\lambda_t^{ext}$  defined later in Section~\ref{intensity_decomposition}.

\subsection{{Inference results}} \label{inference_results_section}

{Recall that }the calibration process involves estimating the values of five parameters:
{
$$
(\lambda_0,\rho,\overline{m},m,\delta).
$$
}

Given that the results of the MSE calibration method are less accurate, an attempt has been made to improve them through strategies detailed in Appendix \ref{annex_mse_calibration_strategies}. The step $\Delta$ for the MSE calibration method is set equal to two days. This choice is also detailed in Appendix \ref{annex_step_choice_mse}.

Table~\ref{comparison_criteria}  below provides some of the criteria used to compare the three approaches ($\text{MSE}^\text{Int}$, $\text{MSE}^\text{Ext}$, Likelihood). Among these criteria, $\lVert \phi \rVert$ is relevant since it is indicative of the Hawkes process regime. Another crucial quantity to highlight is $\frac{\rho \lVert \overline{\phi} \rVert + \lambda_0}{1 - \lVert \phi \rVert}$, where $\lVert \overline{\phi} \rVert = \frac{\overline{m}}{\delta}$. This term is significant for two reasons: 

\begin{enumerate}
\item It aligns with the ergodicity result of a standard linear Hawkes, see \cite{stabile2010risk}, where the baseline intensity $\lambda_0$ is deterministic. Here, the exogenous part is stochastic, hence the presence of $\rho \lVert \overline{\phi} \rVert$ in the numerator.
\item It appears within the expectation presented in Proposition \ref{prop3} when applying the Mean Squared Error (MSE) estimation method.
\end{enumerate}  
To summarize, despite attempts to enhance the MSE estimation method's results, it does not provide exact parameter estimates or accurately capture the Hawkes process regime. While the MSE methods accurately estimate the ergodicity ratio  $\frac{\rho \lVert \overline{\phi} \rVert + \lambda_0}{1 - \lVert \phi \rVert}$ and the process's expectation with fewer data requirements, our objective here is to achieve accuracy across all parameters. The likelihood method enables this precision and offers an estimation for the entire distribution, not just the mean. This comprehensive approach is crucial to estimate reserves in cyber risk management. Therefore, the likelihood estimation method is the preferred choice for inference on real data.

\clearpage

\begin{sidewaystable}[h]
\caption{Comparison of calibration methods}\label{comparison_criteria}
\begin{tabular*}{\textheight}{@{\extracolsep\fill}p{2cm}p{3cm}p{3cm}p{3cm}l}
\toprule%
\textbf{Method} & \textbf{Data requirements} & \textbf{Unicity of solution}  & \textbf{Av. Exec. time in (s) on 1000 runs} & $\boldsymbol{\lVert \phi \rVert}$ \\
\midrule
$\text{MSE}^{\text{Int}}$ & Number of attacks in given intervals & Multiple possible solutions &  18.73 s &  Misestimated  \\
$\text{MSE}^{\text{Ext}}$ & Number of attacks and vulnerabilities in given intervals & Multiple possible solutions &  19.05 s &  Misestimated \\
Likelihood & Dates of attacks and vulnerabilities required & One solution &   22.14 s &  Accurate \\
\midrule
\textbf{Method} & \textbf{Convergence} & \textbf{Parameter estimation accuracy} & \textbf{Sample sensitivity} & $\boldsymbol{\frac{\rho \lVert \overline{\phi} \rVert + \lambda_0}{1 - \lVert \phi \rVert}}$\\ 
\midrule
$\text{MSE}^{\text{Int}}$ & Sensitive to the initialisation parameters & No unique solution & Very sensitive &  Accurate    \\
$\text{MSE}^{\text{Ext}}$ & Sensitive to the initialisation parameters & $\rho$ is estimated correctly but not the others & Very sensitive except for$\rho $ &  Accurate     \\
Likelihood & Converges regardless of the initialisation parameters & Parameters close to the true ones & Not sensitive  &  Accurate   \\
\botrule
\end{tabular*}
\end{sidewaystable}

\clearpage

\section{Parametric inference of the non-stationary one-phase Hawkes process on real data} \label{inference_section}

\subsection{Datasets}

To conduct the analysis, it is necessary to have data on the dates of cyber attacks, denoted as $(t_n)_n$, as well as information about vulnerabilities represented by $(\overline{t}_k)_k$. Vulnerabilities are identified using CVE (Common Vulnerabilites and Exposures) identifiers, which can be obtained by searching vulnerability databases, such as the National Vulnerability Database. Having a vulnerability assigned with a CVE identifier allows us to obtain the disclosure dates $(\overline{t}_k)_k$ associated with the vulnerability. The CVE identifier also provides the Common Vulnerability Scoring System (CVSS) scores, which evaluates the severity of each vulnerability using a score that ranges from 0 for no impact to 10 for a severe impact.

\subsubsection{Attacks databases}

Initially, several databases were considered to collect this information, including the PRC database, the VERIS community database, and the Hackmageddon database, among others.

In the PRC database, cyber events are exclusively focused on data breaches. These breaches can occur through various means, such as the loss of physical documents or computers, insider actions like leaking information, or hacking incidents. When it comes to hacking incidents, the descriptions of these events do not specify whether a software vulnerability was exploited or not. Consequently, trying to figure out the specific vulnerabity and the CVE identifier from these descriptions can be quite tricky and complex. Since the analysis aims to study attacks involving IT vulnerabilities, extracting the relevant $(\overline{t}_k)_k$ dates from the PRC database can be quite challenging. 

Although it contains information on multiple types of attacks beyond data breaches, the VERIS Community database has a limited number of attacks with the associated CVE identifiers. Furthermore, the overall number of reported attacks in this database decreases each year and is relatively low. 

As a result, we mainly focus on the Hackmageddon database, {which includes several types of attacks including those associated with a specific CVE identifier related to a vulnerability exploit.} For an overview of the key figures within the database, refer to \cite{hackmageddon}.

\subsubsection{Vulnerability databases}

Regarding the vulnerability dates $(\overline{t}_k)_k$, we have tested three configurations, {each one is included within the other.}

The first configuration assumes that the Hackmageddon database is comprehensive and includes all the necessary information. Consequently, in this initial approach, $(\overline{t}_k)_k$ corresponds to the dates when vulnerabilities associated with a Hackmageddon attack are published in the NVD database. To be concise, when we refer to Hackmageddon vulnerabilities, we are referring to CVE identifiers extracted from the Hackmageddon database that are associated with cyber attacks from Hackmageddon. For these vulnerabilities, we also have the corresponding dates, denoted as $(\overline{t}_k)_k$, sourced from the NVD database.   

However, it is crucial to note that the assumption of comprehensiveness of the Hackmageddon database is not validated in practice, as the vulnerabilities recovered from the Hackmageddon attacks do not necessarily align with those found in the KEV (Known Exploited Vulnerabilities) database which contains all known exploited vulnerabilities, including those that were exploited and reported in the Hackmageddon database. This database is available at \cite{KEV} and is maintained by the Cybersecurity and Infrastructure Security Agency (CISA) in the US. As a result, we have tested a second configuration in which we consider all vulnerabilities from the KEV database.

The main drawback of the first two approaches is that they assume that all vulnerabilities necessarily lead to an attack on the Hackmageddon database, which is not the case in our model. A vulnerability arrival increases the intensity of the Hawkes process but it does not necessarily lead to the arrival of an attack event. To address this limitation, we have adopted a third approach, which involves taking data from the NVD (National Vulnerability Database) vulnerability database to ensure a more comprehensive coverage of vulnerabilities. The latter database is a U.S. government repository of vulnerabilities, see \cite{NVD} for more details. 

\subsubsection{Description of the Hackmageddon database}

Paolo Passeri founded the Hackmageddon database in 2011 with the purpose of gathering and documenting significant cyber attacks based on his expert judgment. This repository is updated regularly, with timelines of cyber attacks being published on a biweekly basis and based on publicly available incidents. {For each attack, the following information is available: 
   \begin{itemize}
    \item \textbf{Date}: the date at which the reported cyber attack occurred.
    \item \textbf{Attack}: the type of technique used to carry out the attack. For example, this could be indicated as ransomware or malware among other multiple attack techniques. In the case of attacks exploiting a computer vulnerability, a CVE identifier is specified in this field.  
    \item \textbf{Attack class}: it is classified into four categories: cyber crime, cyber espionage, cyber war, and hacking.
    \item \textbf{Country}: it is the country where the cyber attack occurred. When multiple countries are affected, it is indicated as "Multiple".
    \item \textbf{Target}: the name of the entity targeted by the attack. It could be indicated as multiple when several organizations are targeted.
    \item \textbf{Target class}: the industry of the victim. When multiple industries are impacted, it is indicated as "Multiple".
    \item \textbf{Author}: the actor or organization behind the cyber attack. This information is not always available.
    \end{itemize}
}

In what follows, descriptive statistics concerning the Hackmageddon database are be provided. It is important to note that these specific statistics are not exploited in this paper's modeling. But the same methodology could be applied to customize the modeling for specific countries or sectors if desired. These statistics include for example the countries found in this database along with the distribution of industries within it.

\begin{figure}[h]
    \begin{subfigure}{0.45\textwidth}
    {\small
    \begin{tabular}{@{}llp{5cm}@{}}
    \toprule
    Country & Percentage of attacks \\ & in Hackmageddon \\
    \midrule
    US & 33.2\% \\
    Multiple countries & 27.7\% \\
    UK & 3.3\% \\
    Italy & 2.6\% \\
    Canada & 1.8\% \\
    Russia & 1.7\% \\
    India & 1.7\% \\
    Other countries & 24.2\% \\
    \botrule
    \end{tabular}
    }
    \caption{Percentage of attacks in Hackmageddon by country}
    \label{tab:hackmageddon-attacks}
    \end{subfigure}
    \hfill
    \begin{subfigure}{0.45\textwidth}
     \centering
     \includegraphics[scale=0.17]{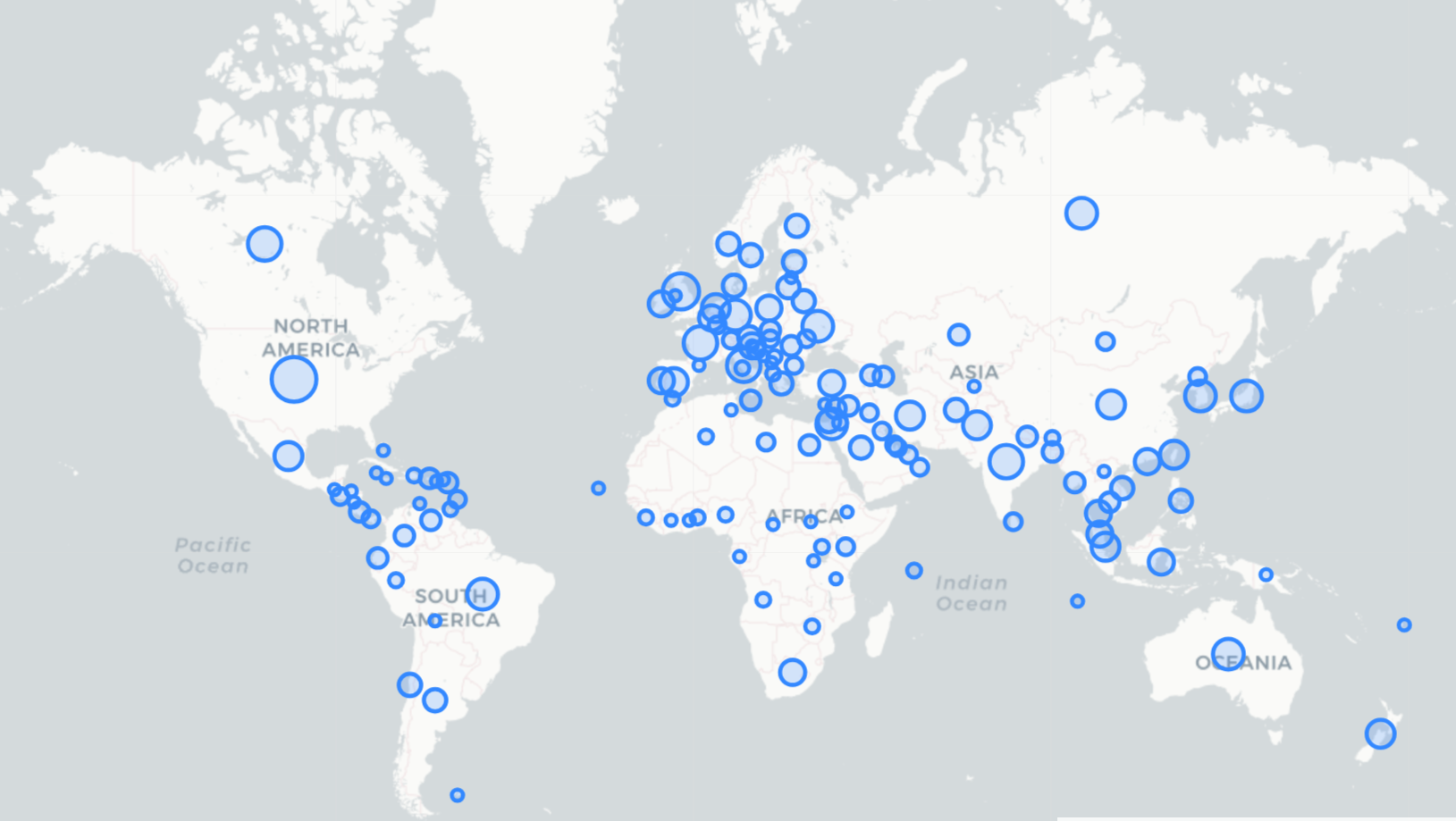}
    \caption{Countries represented in Hackmageddon database}
    \label{map}
    \end{subfigure}
    \caption{Attacked countries and number of attacks per country in Hackmageddon }
\end{figure}

Figure~\ref{map} above shows that the Hackmageddon database includes many regions worldwide, making it more diverse compared to databases like PRC which mainly focuses on the US. However, the United States still holds the highest number of attacks with 33.2 \%, followed by the category "Multiple", which indicates attacks that target several countries simultaneously, highlighting the systemic nature of this risk. Since the remaining countries individually represent less than 1\% of Hackmageddon attacks, a category 'Other' was established to encompass them.

In Table~\ref{tab:percentage-by-industries}, the industries repartition is illustrated: 
\begin{table}[h]
\caption{Percentage of Attacks in Hackmageddon by Industry}\label{tab:percentage-by-industries}
\begin{tabular}{@{}p{5cm}p{6cm}l@{}}
\toprule
\textbf{Industries} & \textbf{Description} & \textbf{Percentage} \\
\midrule
Multiple Industries & 'Multiple' category added for attacks with multiple victims & 19.87\% \\
Individual & Involves individuals and not businesses & 13.88\% \\
Public administration, defense \& social security & Includes responsible government agencies  & 11.70\% \\
Human health and social work activities & Encompasses organizations involved in healthcare and social work. & 10.99\% \\
Financial and insurance activities & Includes financial institutions and insurance companies. & 10.72\% \\
Education & Institutions and organizations related to education and academia. & 6.41\% \\
Professional scientific and technical activities & Involves professional services and technical consulting firms. & 5.28\% \\
Information and communication & Encompasses information technology and communication sectors. & 3.43\% \\
Manufacturing & Involves organizations engaged in manufacturing goods. & 3.14\% \\
Arts entertainment and recreation & Includes sectors related to arts, entertainment, and recreation. & 2.63\% \\
Wholesale and retail trade & Encompasses wholesale and retail trade activities. & 2.79\% \\
Other & Includes attacks on sectors like energy, food services, etc. & 9.17\% \\
\bottomrule
\end{tabular}
\end{table}

The Hackmageddon database covers various sectors. It is important to note that the most represented category is 'Multiple' where multiple sectors are simultaneously affected by a single attack. This highlights the systemic component in cyber risk. Among the other sectors, it is observed that public administration, healthcare, finance, and education sectors appear to experience a higher number of attacks. Two possible explanations can account for this. First, in the case of public sectors, there may be a better reporting a higher number of cyberattacks compared to the private sector due to potentially fewer financial stakes, which could contribute to their elevated representation in the dataset. On the other hand, the prominence of the finance sector could be attributed to its inherent attractiveness to cyber attackers, given its higher financial resources compared with other sectors.

The following graphs show the number of attacks from 2018 to 2022, focusing on two key aspects: the correlation between attacks in consecutive months and the trend of attacks caused by exploiting vulnerabilities.

The plot in Figure~\ref{autocorr_Hackmag} shows a significant autocorrelation in the number of attacks, since the correlation coefficient is 90.63\%. {This aligns with previous findings conducted on the PRC database in \cite{Multi-Variate_HawkesBessy}.}

The second plot highlights the trend of attacks attributed to exploiting vulnerabilities in a system or software. It suggests a significant rise of this type of attacks, specially between 2020 and 2021 and a slight decrease in 2022, that seems to indicate that some preventive measures have been taken.

\begin{figure}[H]
    \begin{subfigure}{0.45\textwidth}
     \centering
   \includegraphics[scale=0.35]{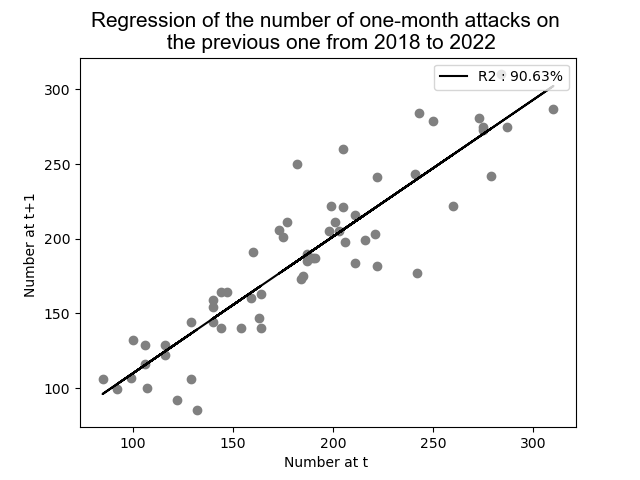}
    \caption{Regression of the number of one-month attacks on the previous one in Hackmageddon from 2018 to 2022}
    \label{autocorr_Hackmag}
    \end{subfigure}
    \hfill
    \begin{subfigure}{0.45\textwidth}
     \centering
        \includegraphics[scale=0.35]{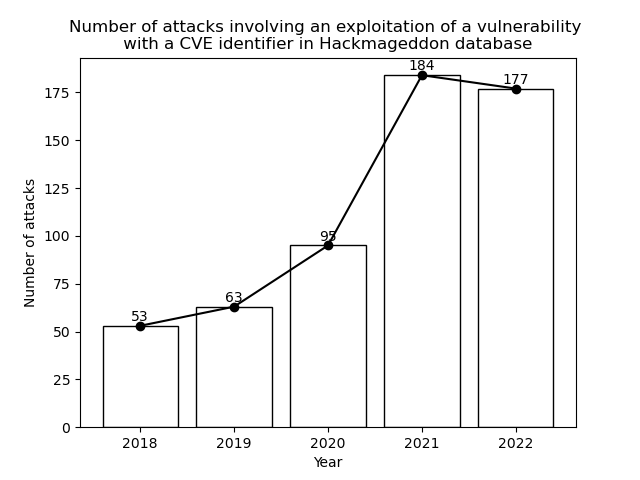}
 \caption{Number of attacks involving a vulnerability with a CVE identifier}
    \end{subfigure}
    \caption{Regression of the number of one-month attacks on the previous one in Hackmageddon and the number of attacks involving a vulnerability with a CVE identifier in the hackmageddon database}
    \label{figure_7}
\end{figure}

\vspace*{-1cm}
It is worth mentioning here that the calibration in the upcoming analysis covers the entire database.  Additionally, based on Figure~\ref{figure_7}, we can see that attacks exploiting vulnerabilities with a CVE identifier are less numerous than for other types of attacks. It should also be noted here that attacks attributed to a CVE identifier significantly increased between 2020 and 2021. This might be a consequence of the Covid-19 crisis: there were overall more attacks from hackers, and on the defenders' side, there was increased vigilance and more regular monitoring of vulnerabilities. As a result, attacks were more often linked to vulnerabilities, leading to an increase in reporting with CVEs.

Our objective here is to quantify the contribution of this external excitation compared to the overall one. Together, the two plots indicate the importance of employing a comprehensive model to understand and predict the number of attacks effectively over a certain period of time. A model that considers both auto-excitation, external excitation factors (such as vulnerability exploitation) and remediation and preventive measures is crucial to gaining deeper insights into a cyber attack propagation mechanism.

\subsubsection{Description of vulnerability databases }

A vulnerability is a weakness in an IT system or software that might be exploited by a hacker to gain unauthorized access or inflict damage. It is similar to an unlocked door or an open window in a house that a potential thief could use to break in. Vulnerabilities arise from errors in programming, configuration mistakes, design flaws or any other risk factor that could expose the system to security risks. For example, imagine a website that allows users to submit comments on blog posts, if the website does not properly validate the comments before displaying them on the page, this could be exploited by an attacker who submits a malicious script as a comment. When other users visit the page and view the comments, the malicious script executes in their web browsers, allowing the attacker to steal their login credentials or spread malware. This type of vulnerability is known as a "Cross-Site Scripting" (XSS) attack. It has been used in various cyber attacks to compromise user accounts, steal sensitive information, and spread harmful content.

All known exploited vulnerabilities can be found in the KEV (Known Exploited Vulnerabilities) database. It is a comprehensive repository of documented vulnerabilities that have been confirmed to be exploited by cyber attackers. This for example allows security teams to prioritize their patching efforts. 

When security researchers or organizations discover a vulnerability, they request a CVE (Common Vulnerabilities and Exposures) ID from the MITRE corporation, which is responsible for managing the CVE system. Once a CVE ID is assigned, the vulnerability and its details, such as its impact and how to fix it, are documented and made publicly available in the NVD database. This consistency helps in sharing information and collaborating on fixing vulnerabilities. The NVD database also assigns Common Vulnerability Scoring System (CVSS) scores that range from 0 (no impact) to 10 (severe impact) to assess the severity and potential impact of each vulnerability.

Figure~\ref{density_NVD_KEV} below depicts the distribution of CVSS scores between 2018 and 2022 in the NVD (National Vulnerability Database) and KEV databases. 
It can be observed that vulnerabilities with a CVSS score below 5 are not known to be exploited, as they are not present in the KEV database. Additionally, vulnerabilities with a CVSS score between 7 and 8 are the most numerous in both the NVD and KEV databases. Moreover, the KEV database contains a significant number of critical vulnerabilities with a score above 9.

\begin{figure}[H]
    \centering
    \begin{subfigure}[b]{0.45\textwidth}
        \centering
        \includegraphics[scale=0.35]{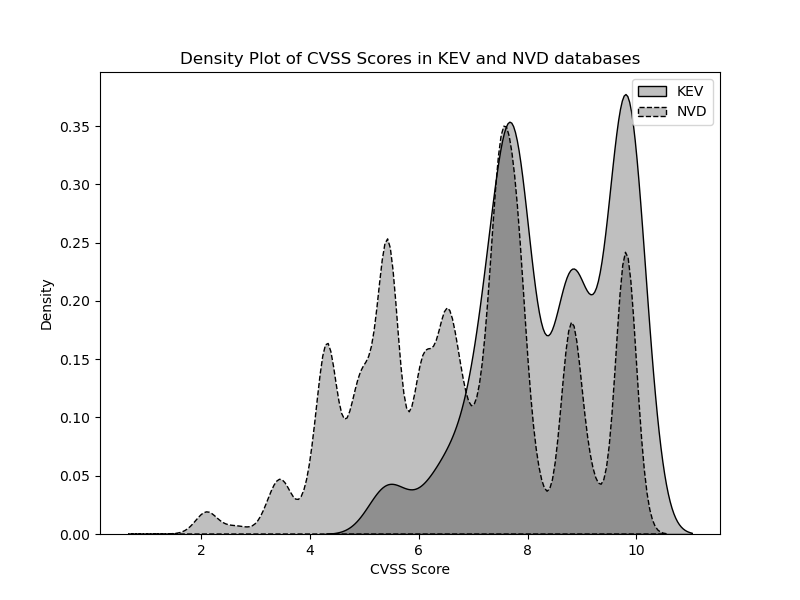}
        \caption{CVSS scores repartition in NVD and KEV databases from 01/01/2018 to 12/31/2022}
        \label{density_NVD_KEV}
    \end{subfigure}
    \hfill
    \begin{subfigure}[b]{0.45\textwidth}
        \centering
        \includegraphics[scale=0.45]{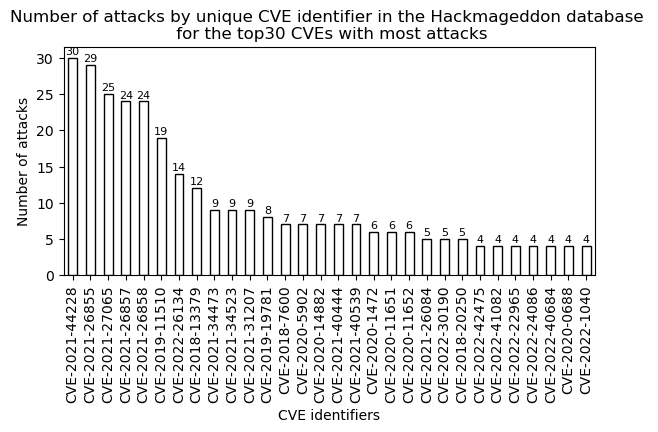}
        \caption{Number of attacks by CVE identifiers in the Hackmageddon database}
        \label{number_attacks_CVE}
    \end{subfigure}
    \caption{CVSS scores distribution in NVD and KEV databases and the CVE-identified attacks in the Hackmageddon database}
\end{figure}

Figure~\ref{number_attacks_CVE} illustrates the number of attacks per CVE identifier reported in the Hackmageddon database. For visualization purposes, we are focusing on the 30 most frequent vulnerabilities out of the 375 unique CVE identifiers in the Hackmageddon database. The most commonly occurring vulnerability is CVE 2021-44228, known as the log4j vulnerability, which affected a widely used Java authentication library of the same name. The other vulnerabilities that follow typically target Microsoft Exchange servers.

Upon analyzing the Hackmageddon database, we found that vulnerabilities associated with cyber attacks have also a score above 5, which aligns with the observation in the KEV database. Therefore, when conducting calibrations based on the extended vulnerability database (here, it is  the NVD database where not all vulnerabilities are necessarily associated with attacks), we focus on vulnerabilities with a score above 5 to ensure that groups of vulnerabilities are coherent for our analysis. 

\subsection{Inference results}
\subsubsection{Calibration period}
The analysis is carried out over the period 2018-2022. Data before 2018 are not directly exploitable because some files, for example, contain images which are hard to exploit. Additionally, given the rapidly evolving nature of cyber risk, having an extended historical record may not be essential. Figure~\ref{calibration_periods} illustrates the different timelines used in the calibration process. The calibration period is in blue (on the year 2021) knowing all attacks and vulnerabilities from the orange period (from 2018 to 2021). The validation is conducted using data from the year 2022 (shown in green). 

\begin{figure}[H]
    \centering
    \includegraphics[scale=0.45]{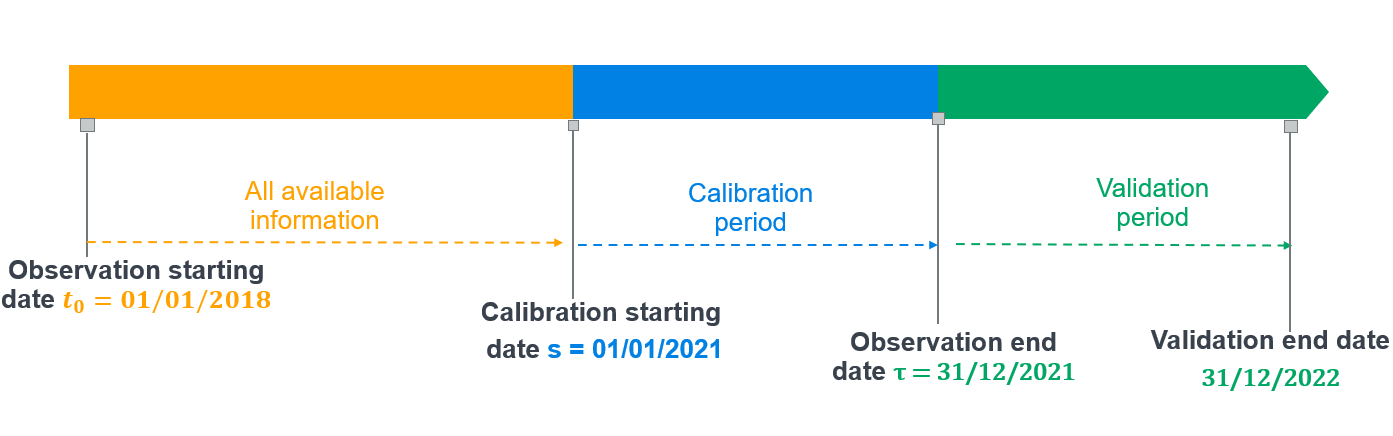}
    \caption{Illustration of the different periods in the calibration process in the Hackmageddon database}
    \label{calibration_periods}
\end{figure}

Table~\ref{nb_vuln} below shows the number of attacks and vulnerabilities in the used databases on the three timelines, using the same color code: 

\begin{table}[h]
\caption{Number of attacks and vulnerabilities in the Hackmageddon, KEV and NVD databases with a CVSS score above 5}\label{nb_vuln}
\begin{tabular}{@{}p{2cm}p{2cm}p{2cm}p{2cm}p{2cm}@{}}
\toprule
\textbf{Year} & \textbf{Nb. of attacks} & \textbf{Nb. of Hackmageddon vulnerabilities} & \textbf{Nb. of KEV vulnerabilities} & \textbf{Nb. of NVD vulnerabilities} \\
\midrule
\textcolor{orangeM}{\textbf{2018}} & \textcolor{orangeM}{\textbf{1310}}  & \textcolor{orangeM}{\textbf{31}} & \textcolor{orangeM}{\textbf{63}} & \textcolor{orangeM}{\textbf{8402}} \\
\textcolor{orangeM}{\textbf{2019}} & \textcolor{orangeM}{\textbf{1788}} & \textcolor{orangeM}{\textbf{62}}&\textcolor{orangeM}{\textbf{102}} & \textcolor{orangeM}{\textbf{11932}} \\
\textcolor{orangeM}{\textbf{2020}} & \textcolor{orangeM}{\textbf{2321}} & \textcolor{orangeM}{\textbf{63}} & \textcolor{orangeM}{\textbf{125}}&  \textcolor{orangeM}{\textbf{15979}} \\
\textcolor{blueM}{\textbf{2021}} & \textcolor{blueM}{\textbf{2528}} &\textcolor{blueM}{\textbf{128}} & \textcolor{blueM}{\textbf{175}} &  \textcolor{blueM}{\textbf{17829}}  \\
\textcolor{greenM}{\textbf{2022}} & \textcolor{greenM}{\textbf{2649}}  &\textcolor{greenM}{\textbf{91}} & \textcolor{greenM}{\textbf{113}}&  \textcolor{greenM}{\textbf{22288}} \\
\bottomrule
\end{tabular}
\end{table}

The size of the vulnerability database increases by refining the data sources. The impact of increasing the external database is illustrated in Section~\ref{results_calibration_section}.

\subsubsection{Calibration results} \label{results_calibration_section}

In this section, we present the obtained calibration results using the maximum likelihood method for two models: one is the standard linear Hawkes process without the external excitation and the other is a Hawkes process with an external excitation.  These results are illustrated in Table~\ref{results_calibration_table}. To simplify notations, the intensity of the standard Hawkes process is expressed as: $\lambda_t^{st} = \lambda_0 + \sum_{T_i < t} m e^{-\delta (t-T_i)}$. And recall that the intensity of the Hawkes process with external excitation is expressed as $\lambda_t = \lambda_0 + \sum_{T_i < t} m e^{-\delta (t-T_i)} + \sum_{\overline{T}_k < t} \overline{m} e^{-\delta (t-\overline{T}_k)}$. The results for the Hawkes process with the external excitation are presented for the three vulnerability databases, along with the 95\% confidence intervals for the estimated parameters.

We observe that the likelihood and MSEs values improve as the external excitation database is considered and expanded. The highest values for these indicators is achieved with a model without external excitation, while the smallest values are obtained with the NVD database, which represents the largest external excitation database, as illustrated in Table~\ref{nb_vuln}. These results emphasize the importance of considering a stochastic external excitation in the model.

Another interesting result is that $\lVert \phi \rVert$ (the norm of the self-excitation-internal kernel) decreases as the external database expands. It is almost halved between the model without external excitation and the model with external excitation when considering vulnerabilities from NVD, which is the most extensive database. By refining the external database with a more extensive collection of vulnerabilities, as illustrated in Table~\ref{nb_vuln}, what was initially attributed to self-excitation is, in fact, due to the arrival of external events. More generally, as we refine the contribution of external events, the endogenousness reflected by the $\lVert \phi \rVert$ decreases. 

Also, the indicator $\frac{\rho \hspace{1mm} \lVert \overline{\phi} \rVert + \lambda_0}{1 - \lVert \phi \rVert}$ remains stable regardless of the external database.

\clearpage

\begin{sidewaystable}[h]
\caption{Calibration results for a standard linear Hawkes process and the Hawkes process with external excitation with vulnerabilities from the Hackmageddon, KEV, and NVD databases with \(t_0 = 01/01/2018\), \(s = 01/01/2021\), and \(\tau = 31/12/2021\)}\label{results_calibration_table}
\begin{tabular*}{\textheight}{@{\extracolsep\fill}p{2cm}p{2cm}cccccc}
\toprule
\textbf{Model} & \textbf{Vuln. database} & \(\boldsymbol{\lambda_0}\) & \(\boldsymbol{\rho}\) & \(\boldsymbol{\overline{m}}\) & \(\boldsymbol{m}\) & \(\boldsymbol{\delta}\) & \(\boldsymbol{\lVert \phi \rVert}\) \\
\midrule
No external events & - & 2.7031 & - & - & 0.9182 & 1.5047 & 0.61 \\
& 95\% C.I & [2.4863,2.9199] & - & - & [0.8608, 0.9756] & [1.1723, 1.8371] & - \\
With external events & Hackmageddon & 2.7081 & 0.3636 & 0.5941 & 0.8891 & 1.5080 & 0.58 \\
& 95\% C.I & [2.4873,2.9289] & [0.3180, 0.4092] & [0.3484, 0.8398] & [0.6909, 1.0873] & [1.1649, 1.8511] & - \\
With external events & KEV & 2.6964 & 0.5057 & 0.9774 & 0.8529 & 1.5061 & 0.56 \\
& 95\% C.I & [2.4229, 2.9699] & [0.4527, 0.5587] & [0.4388, 1.2282] & [0.6734, 1.1048] & [1.1921, 1.8239] & - \\
With external events & NVD & 2.4195 & 48.849 & 0.077413 & 0.67139 & 1.8697 & 0.36 \\
& 95\% C.I & [2.1573,2.6817] & [48.2987,49.1993] & [0.01211,0.1427] & [0.4985,0.8442] & [1.3998,2.3396] & - \\
\cmidrule{1-8}
\textbf{Model} & \textbf{Vuln. database} & \(\boldsymbol{\lVert \overline{\phi} \rVert}\) & \(\boldsymbol{\rho \hspace{1mm} \lVert \overline{\phi} \rVert}\) & \(\boldsymbol{\frac{\rho \hspace{1mm} \lVert \overline{\phi} \rVert + \lambda_0}{1 - \lVert \phi \rVert}} \) & \(\boldsymbol{-ln(\mathcal{L})}\) & \(\frac{\textbf{MSE}^{\textbf{Ext}}}{(N_{\tau} - N_s)}\) & \(\frac{\textbf{MSE}^{\textbf{Int}}}{(N_{\tau}-N_s)}\) \\
\cmidrule{1-8}
No external events & - & - & - & 6.9349 & 6932.46 & 1.34 \% & 1.04 \% \\
With external events & Hackmageddon & 0.3960 & 0.1440 & 6.9002 & 6563.28 & 1.21 \% & 0.95 \% \\
With external events & KEV & 0.6516 & 0.3295 & 6.9828 & 6218.54 & 1.18 \% & 0.93 \% \\
With external events & NVD & 0.0416 & 2.0327 & 6.9595 & \textbf{5416.83} & \textbf{1.05} \% & \textbf{0.84} \% \\
\botrule
\end{tabular*}
\end{sidewaystable}

\clearpage

\subsection{Validation of the calibration}
\subsubsection{Validation tests}

We use a Kolmogorov-Smirnov test to evaluate the  goodness of fit of the used model. The test relies on the following {generic result which follows from Theorem 4.1 of Garcia and Kurtz \cite{garcia2008spatial}. }

\begin{proposition}
    Let $(T_k)_k$ be the jump times of the counting process $(N_t)_t$ with the intensity $\lambda_t$.
    Then $ (\tau_k = \int_{0}^{T_k} \lambda_t \mathrm{d}t, \quad k \ge 1 )$ are the jump times of a homogeneous Poisson process of intensity 1. 
\end{proposition}

If the underlying process is a Hawkes process with intensity $\lambda_t$, then the times $\theta_k = \tau_k - \tau_{k-1}$, $k\ge 1$ are independent and have an exponential distribution function with parameter 1. The test compares the empirical distribution function obtained from the observed times and the exponential distribution function with parameter 1. 

The null hypothesis considered is the adequacy with a  confidence level of 5\%. The cases where the null hypothesis is not rejected are highlighted in bold in Table~\ref{ks_results}: 

\begin{table}[h]
\caption{Kolmogorov-Smirnov goodness-of-fit results}\label{ks_results}
\begin{tabular}{@{}cccc@{}}
\toprule
$\textbf{Vulnerabilities database}$ & $\textbf{Hackmageddon}$ & $\textbf{KEV}$ & $\textbf{NVD}$ \\
\midrule
p-value & 0.04753 & \textbf{0.05065} & \textbf{0.1724} \\
\botrule
\end{tabular}
\end{table}

The adequacy hypothesis is rejected when vulnerabilities are extracted from the Hackmageddon database, and is not rejected with vulnerabilities extracted from the KEV and NVD databases. The test fits better when vulnerabilities are extracted from the largest database, here it is the NVD database. {This observation aligns with the findings in Table~\ref{results_calibration_table}, where the best results are obtained for the NVD database.} This highlights the importance of expanding the vulnerability database.

\subsubsection{Distribution of the number of attacks predicted in one year}

To further validate the model, we forecast the distribution of the number of attacks that could occur over a year and compare it with the number of attacks that occurred in 2022. The simulation is done using 10,000 trajectories generated using the thinning algorithm. The observed value in 2022 is then compared with the 95th and 5th percentiles for each vulnerability database. The choice of a one-year horizon is significant due to regulatory requirements in insurance, as it aligns with frameworks like Solvency II's internal models. 

We recall in Table~\ref{sim_params} the calibrated parameters obtained for each vulnerability database:

\begin{table}[h]
\caption{Simulation parameters obtained from calibration on different vulnerabilities databases with \( t_0 = 01/01/2018 \), \( s = 01/01/2021 \) and \( \tau = 31/12/2021 \)}\label{sim_params}
\begin{tabular}{@{}cccccc@{}}
\toprule
\textbf{Vulnerabilities database} & $\boldsymbol{\lambda_0}$ & $\boldsymbol{\rho}$ & $\boldsymbol{\overline{m}}$ & $\boldsymbol{m}$ & $\boldsymbol{\delta}$ \\
\midrule
Hackmageddon & 2.7081 & 0.3636 & 0.5941 & 0.8891 & 1.5080 \\
KEV & 2.6964 & 0.5057 & 0.9774 & 0.8529 & 1.5061 \\
NVD & 2.4195 & 48.849 & 0.077413 & 0.67139 & 1.8697 \\
\botrule
\end{tabular}
\end{table}

The distributions in Figure~\ref{distribitions_nb_attacks_valid} seem to capture the dynamics of cyber attacks in the Hackmageddon database. The distribution of the number of attacks with vulnerabilities from the NVD database has the smallest variance. By incorporating richer external information, the variance of the distribution obtained decreases as the parameters are better estimated. Here, the variance obtained from NVD is smaller than the one obtained from KEV, which is smaller than the one obtained from restricting the analysis to vulnerabilities contained in the Hackmageddon database. This decrease in variance has significant implications in insurance reserve calculations, for example, when considering only the Hackmageddon database for attacks and vulnerabilities would lead to higher reserve calculations compared to considering the NVD database for vulnerabilities. However, in the context of insurance premium calculations, the impact is minor, as the mean number of attacks remains consistent regardless of the external information database.

\begin{figure}[H]
\centering
\includegraphics[scale=0.5]{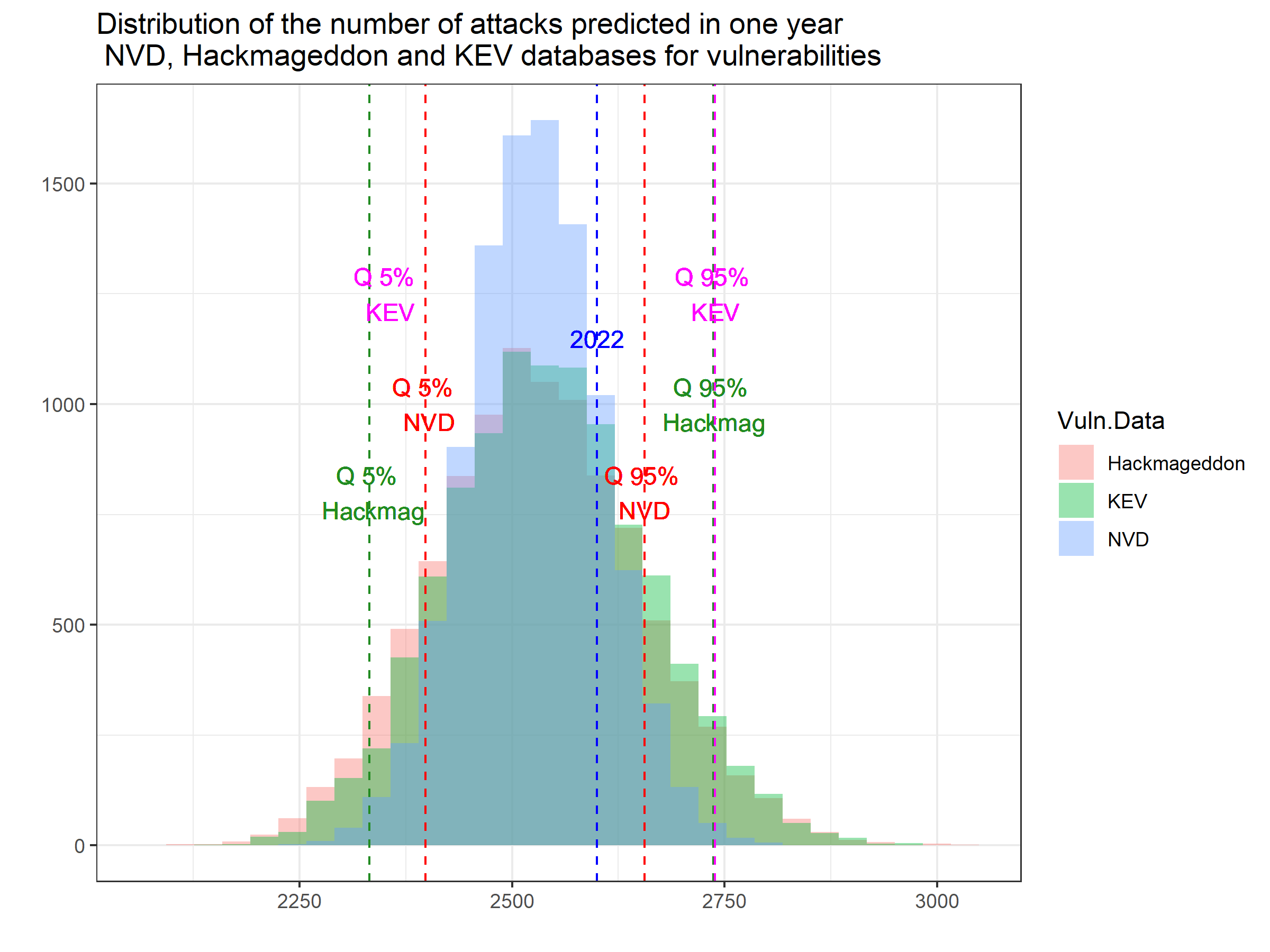}
\caption{Distribution of the number of attacks predicted in one year}
\label{distribitions_nb_attacks_valid}
\end{figure}

\subsection{Relative contribution of internal and external intensities in the global one} \label{intensity_decomposition}

In this section, the fractions of intensity attributed to external, internal, and baseline components are plotted for attacks extracted from the Hackmageddon database and for vulnerabilities respectively from Hackmageddon in Figure~\ref{ratio_intensities_Hackmag}, KEV in Figure~\ref{ratio_intensities_KEV} and NVD in Figure~\ref{intensity_frac_nvd}. We recall that: $ \lambda_t = \lambda_0  + \sum_{\overline{T_k} < t}\overline{m} e^{-\delta (t-\overline{T_k})} + \sum_{T_i < t}m e^{-\delta (t-T_i)} $ and we denote $\lambda_t^{\text{ext}} = \sum_{\overline{T_k} < t}\overline{m} e^{-\delta (t-\overline{T_k})} $ and $\lambda_t^{\text{int}} =  \sum_{T_i < t}m e^{-\delta (t-T_i)}$.

This visualization provides a clear understanding of the relative contributions of these factors to the overall intensity in the model. This breakdown of the intensity of the attacks process helps us determine what is driving the intensity of the Hawkes process and  where the observed attacks originate from. Meaning that: 
\begin{itemize}
    \item If the overall intensity is primarily attributed to $\lambda_t^{\text{int}}$, it indicates that the system in which attacks spread is endogenous, implying that one attack triggers another due to their contagious nature.
    \item If the baseline intensity $\lambda_0$ dominates, it suggests that the intensity is high because there is a high rate of attacks happening spontaneously, not necessarily triggered by previous attacks. 
    \item When $\lambda_t^{\text{ext}}$ is the dominant factor, it indicates a significant contribution from the exogenous events, here it is vulnerabilities. These events are increasing the intensity of attacks, potentially setting off a chain reaction of self-excitation.
\end{itemize}

This decomposition also allows for selecting the appropriate response strategy by activating the appropriate measures to mitigate the number of attacks, depending on whether the threat is endogenous or exogenous, as detailed in Section~\ref{section_forecasting}.

In Figure~\ref{ratio_intensities_Hackmag}, we focus on vulnerabilities present in the Hackmageddon database. Both baseline and internal components are prevalent, taking turns in dominance. However, the orange curve is more dominant for longer periods. Notably, we observe a spike in blue (representing the exogenous part) that triggers a series of attacks. This can be interpreted as follows: the initiation of the endogenous system results from the arrival of external events (from the deterministic baseline intensity $\lambda_0$ and from the stochastic $\lambda_t^{\text{ext}}$). Subsequently, we witness internal contagion. This interpretation aligns with the fact that the vulnerabilities considered here lead automatically to an attack.

\begin{figure}[h]
\centering
\includegraphics[scale=0.5]{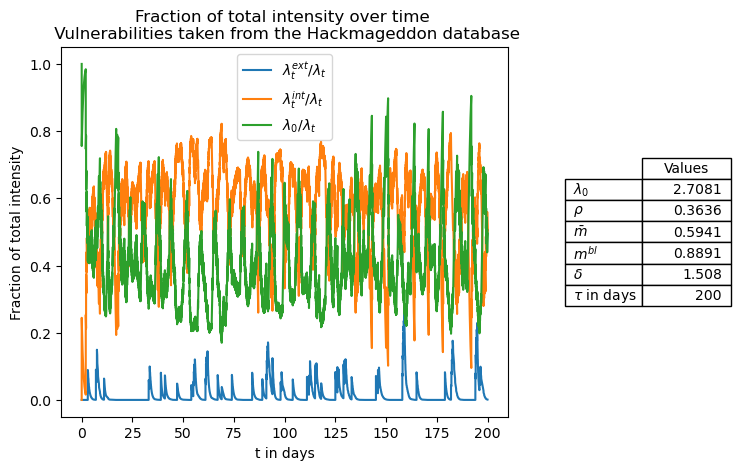}
\caption{Fraction of total intensity over time with vulnerabilities taken from the Hackmageddon database}
\label{ratio_intensities_Hackmag}
\end{figure}

\begin{figure}[h]
\centering
\includegraphics[scale=0.5]{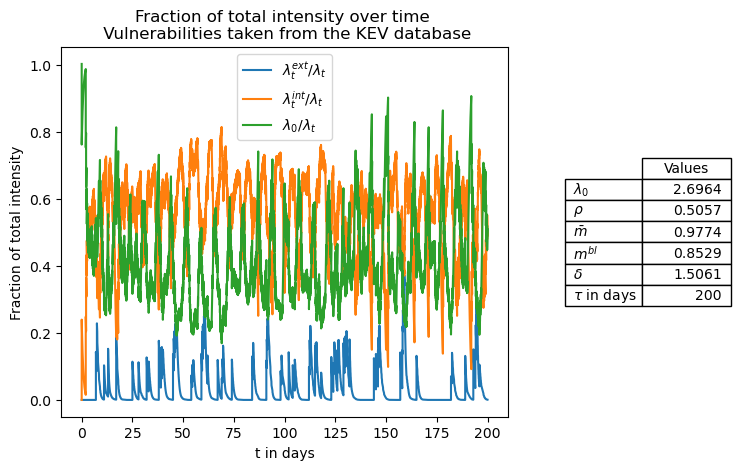}
\caption{Fraction of total intensity over time with vulnerabilities taken from the KEV database}
\label{ratio_intensities_KEV}
\end{figure}

In Figure~\ref{ratio_intensities_KEV}, the focus shifts to vulnerabilities from the KEV database. Recall that all vulnerabilities in this database are known to have triggered an attack. The same remarks previously made for Figure~\ref{ratio_intensities_Hackmag} apply here as well. We can see that the orange curve dominates more than the green and blue ones. The exogenous part is slightly more significant, but the system remains endogenous. While in Figure~\ref{ratio_intensities_Hackmag}, we could observe successive peaks in orange not necessarily associated with peaks in blue (for example, between day 13 and day 25), here we notice a stronger correlation. The same interpretation previously mentioned for Hackmageddon remains valid for KEV: the initiation of the endogenous system originates from the arrival of external events. 

Figure~\ref{intensity_frac_nvd} focuses on vulnerabilities extracted from the NVD. Here, the exogenous component is more pronounced, meaning that a significant portion of the excitation comes from the arrival of vulnerabilities. We observe fewer phases where the orange curve dominates. The following interpretation can be made: several vulnerabilities arrive and increase the intensity without necessarily triggering an attack, let alone contagious events. In other words, the intensity rises primarily due to external factors. This aligns with the fact that NVD vulnerabilities do not always cause an attack.  

\begin{figure}[H]
    \centering
    \includegraphics[scale=0.5]{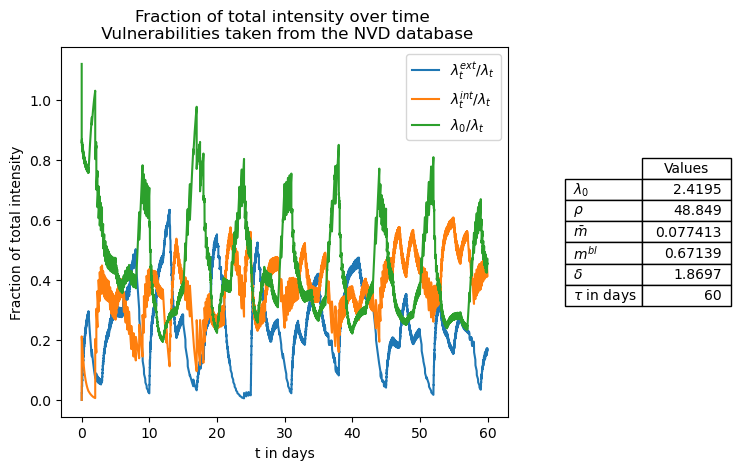}
 \caption{Fraction of total intensity over time with vulnerabilities taken from the NVD database}
 \label{intensity_frac_nvd}
\end{figure}

\subsection{Discussion on calibration choices } \label{justif_1P} 

In this section, we detail our choice of calibrating the first phase and not calibrating both phases simultaneously. 

Unlike the model used in \cite{dassiostwo}, which is designed to model the spread of Covid-19 and involves centralized reaction measures administered by states, the landscape of cyber risk is more heterogeneous. Notably, not all cyber attacks have the potential for widespread propagation. Generally, when attacks are isolated and not contagious, companies implement their own security measures, and the response time depends on how long it takes to find a fix. 
Estimating the effective response time $\ell$ is not meaningful when considering a portfolio exposed to isolated and non-contagious cyber attacks. 

However, when dealing with highly contagious attacks that have the potential to trigger a pandemic, such as when critical vulnerabilities are exploited to widespread contagion, the two-phase model could be calibrated for each critical vulnerability since the time to address them can vary and depends on the time required to find and deploy a patch for each. The Hackmageddon database for example connects certain attacks to specific vulnerabilities. Figure~\ref{number_attacks_CVE} shows that CVE 2021-44228 known as log4j vulnerability is the most exploited vulnerability in this database and results in 30 attacks. The distribution of these thirty attacks is as follows: 

\newpage

\begin{table}[h]
\caption{Number of attacks caused by the log4j vulnerability for each target}\label{attack_counts_log4j}
\begin{tabular}{@{}lc@{}}
\toprule
\textbf{Target} & \textbf{Number of attacks} \\
\midrule
Multiple Organizations & 14 \\
Vulnerable VMware Horizon deployments & 4 \\
Vulnerable Apache Log4j servers & 1 \\
Belgium's ministry of defense & 1 \\
Corporate networks in the Middle East & 1 \\
ONUS & 1 \\
Organizations in the U.S., Australia, Canada & 1 \\
Self-hosted Minecraft servers & 1 \\
Ubiquiti network appliances & 1 \\
Undisclosed Federal Civilian Executive Branch & 1 \\
Unnamed engineering company with energy and mining operations & 1 \\
VMware vCenter Server instances & 1 \\
Vulnerable ZyXEL devices & 1 \\
Windows and Linux devices & 1 \\
\botrule
\end{tabular}
\end{table}

In order to calibrate a second phase using the log4j data for example, additional information is required. Here, 14 victims in the category 'Multiple Organizations' are impacted by this vulnerability from 12/12/2021 to 06/23/2022. However, we do not have details on whether these same organizations were reinfected after deploying a patch or not.

To summarize, we apply the parameters identified in the first phase to derive the right response measures in the second phase in the context of attacks accumulation.  We assume that an insurer with a known response capacity is insuring a portfolio similar to the Hackmageddon database. Our objective is to determine the appropriate response parameters to ensure that the insurer's daily response capacity is not exceeded.

\section{Forecasting cyber pandemic scenarios} \label{section_forecasting}

The aim of this section is to explore how an insurer, confronted with a limited daily capacity to assist policyholders, can {take actions and incentivize them in order to overcome a saturation of its response capacity  during a cyber pandemic scenario.} We make the simplifying assumption that the portfolio under study corresponds to that of the Hackmageddon database and that each new attack requires assistance from the insurer. This assumption can, of course, be adjusted based on real portfolios and subscribed coverage.

The two-phase model, detailed in Section~\ref{section_two_phase_model_description}, is used here: 

\begin{equation}
\lambda_t = \begin{cases}
\lambda_0  + \sum_{\overline{T_k} < t}\overline{m} e^{-\delta (t-\overline{T_k})} + \sum_{T_i < t}m^{bl} e^{-\delta (t-T_i)} & \text{if } t < \ell \\ 
\\ 
\alpha_0 \lambda_{0} + \alpha_1 (\lambda_{\ell^{-}} - \lambda_0) & \text{if } t = \ell \\ 
\\
\alpha_0 \lambda_0  + \alpha_1 (\lambda_{\ell^{-}} - \lambda_0) e^{-\delta(t-\ell)}   +  \sum_{\ell<T_i < t}m^{al} e^{-\delta (t-T_i)} & \text{if } t > \ell 
\end{cases}
\end{equation}

The objective is to implement a response quantified by the following reaction parameters:
$$
(\alpha_0, \alpha_1, m^{al})
$$
This is achieved by applying the  parameters calibrated in the first phase, as detailed in Table~\ref{results_calibration_table}, using vulnerabilities from the NVD database. The second phase of process is then initiated at the deterministic time $\ell$ chosen as the first time the insurer's daily assistance limit is exceeded. This value is numerically set according to this.

Adjusting the $\alpha_0$ parameter involves reducing the baseline intensity, thereby decreasing the spontaneous rate at which cyber attacks occur. This could for example reflect increased vigilance among employees, such as proactive password changes. 
Modifying the $\alpha_1$ parameter impacts the branching ratio of events before $\ell$. The new branching ratio of these events is $\alpha_1 \lVert \phi \rVert = \alpha_1 \frac{m^{bl}}{\delta}$.

A choice of $\alpha_1<1$ could be interpreted as adjusting the patching speed of prior cyber attacks occurring before entering the reaction phase. A smaller $\alpha_1$ value indicates for example quicker patching, making previous incidents less contagious and having a diminished contribution to the intensity of the attack process. Finally, changing the $m^{al}$ parameter sets the level of contagion of attacks arriving after the reaction phase. The smaller this parameter, the less contagious future events will be. Such a decrease can be attributed to an increased awareness due to previous attacks and the implementation of strategies to prevent future incidents.
This change of $m^{al}$ also affects the branching ratio of events after $\ell$ to $\frac{m^{al}}{m^{bl}} \lVert \phi \rVert $.

The analysis starts in Section~\ref{section_reaction_measures} with a basic simulation in which distinct response parameters are individually adjusted to observe their impact on the number of attacks. {We compute the impact both dynamically on a specific trajectory and also statistically by representing the empirical distribution of the number of attacks and compare it to the cumulative reaction capacity over the total duration of the pandemic.} Next in Section~\ref{optim_parameters_response}, {and since the cumulative approach is less realistic,} we search for a set of response parameters that prevent the insurer from exceeding {its}  daily assistance capacity in average rather than its cumulative reaction capacity.  

\subsection{The insurer's reaction measures} \label{section_reaction_measures}

We assume that the assistance capacity $C$ for the insurer is of 5 policyholders per day. This value is provided for illustrative purposes only and can be adjusted based on the insurer's actual assistance capacity and the characteristics of the portfolio. This assumption allows us to numerically set the parameter $\ell$ of the activation of the reaction measures as the first time where the count of new attacks exceeds the insurer's cumulative assistance capacity. Using the calibrated parameters from the Hackmageddon database and the NVD database in Table~\ref{results_calibration_table}, $\ell = 3$ in this section. Over a span of 10 days - the duration of the pandemic in this illustrative example - this amounts to a total capacity of 50 cases. Three response strategies are considered in this example which are linked to the response parameters: $\alpha_0$, $\alpha_1$ or $m^{al}$. In Section~\ref{optim_parameters_response}, the combination of the three strategies is discussed. 

We plot the effect of these reaction parameters on the number of attacks: Figure~\ref{alpha_0_params_reaction} for $\alpha_0$ parameter, Figure~\ref{alpha_1_params_reaction} for  $\alpha_1$ parameter and Figure~\ref{m_al_params_reaction} for $m^{al}$ parameter. In each Figure's (a), the bar chart depicts the number of attacks over 10 days for different reaction measures on one scenario, distinguished by color: green for the one-phase (1P trajectory), and blue, red and yellow for the two-phase (2P trajectory) corresponding to different values of the tested reaction parameter. The reaction time $\ell$ is represented by a violet dotted line, while the 10-day cumulative maximal reaction capacity is indicated by a red dotted line. The histograms in each Figure's (b) represent the distribution of the predicted number of attacks from 10 000 simulations, highlighting the insurer's maximum cumulative assistance capacity with a blue dotted line.

In Figure~\ref{alpha_0_params_reaction}, three values of $\alpha_0$ are tested: (0,0.5,1) {and $\alpha_1$ is set to $\alpha_1 = 1$}. In the case where $\alpha_0 = 0$, as shown in blue in Figure~\ref{alpha_0_params_reaction}, the cumulated assistance capacity is not attained. However, in the case where $\alpha_0 = 0.5$, this maximum capacity might be exceeded in some {cases on the right of the distribution}. Moreover, when $\alpha_0 = 1$ or where no intervention takes place, this maximum capacity is more often exceeded.

\begin{figure}[H]
    \begin{subfigure}{0.45\textwidth}
     \centering
\includegraphics[scale=0.3]{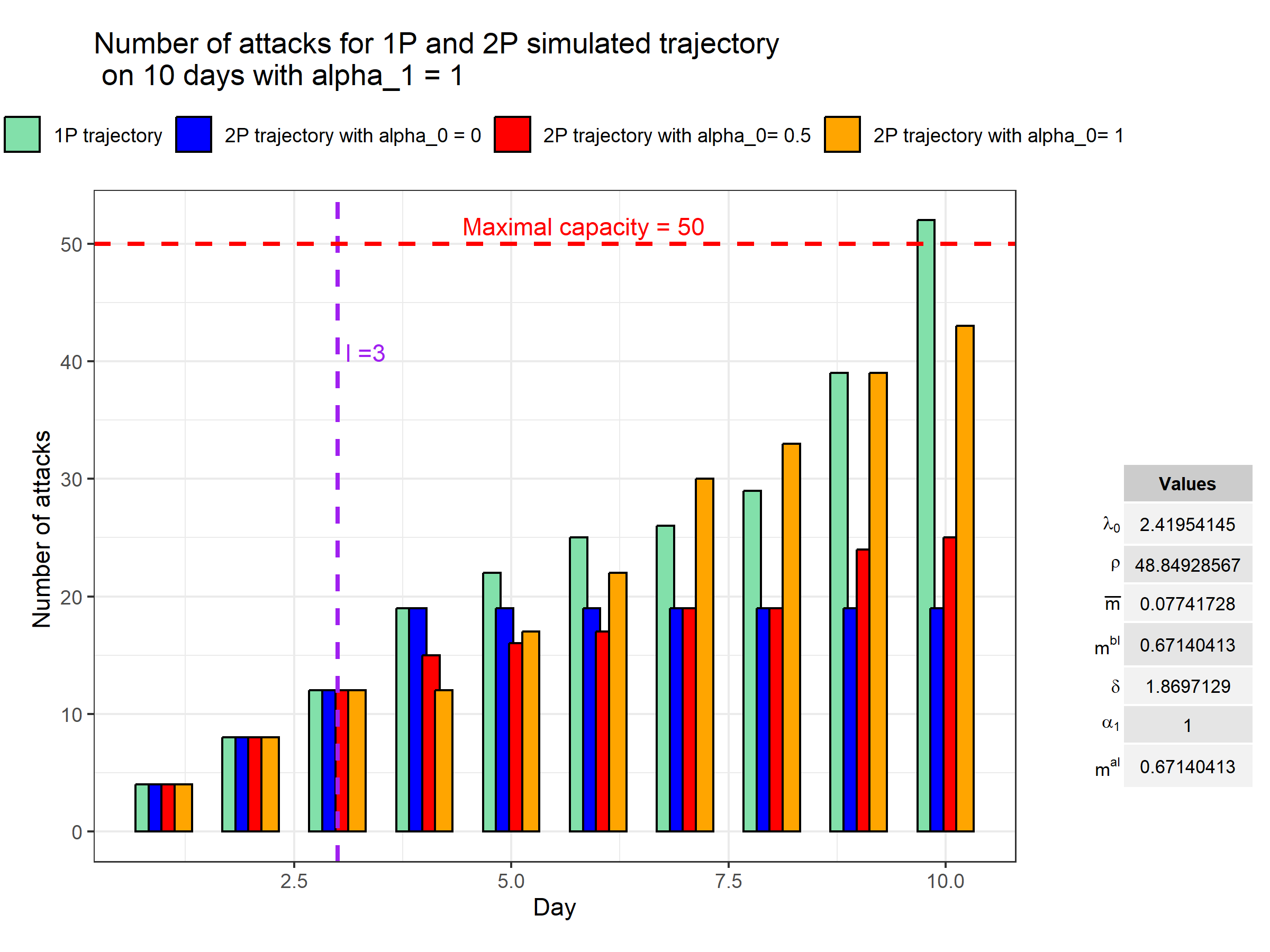}
\caption{Number of attacks predicted in 10 days for one trajectory}
    \end{subfigure}
    \hfill
    \begin{subfigure}{0.45\textwidth}
    \centering
\includegraphics[scale=0.3]{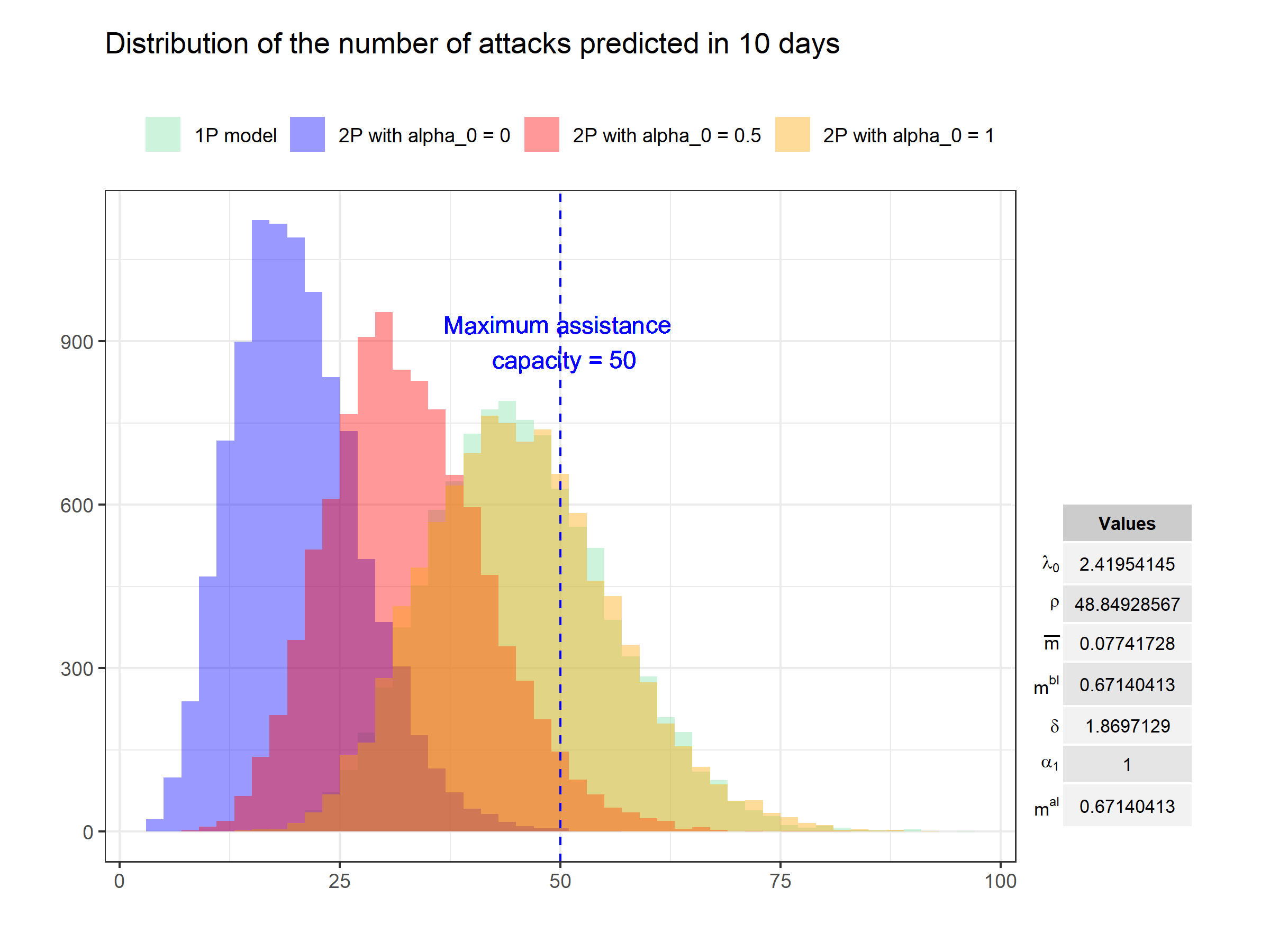}
\caption{Distribution of the number of attacks predicted in 10 days, simulation done with 10 000 trajectories}
    \end{subfigure}
    \caption{Effect of reaction measures on $\alpha_0$ parameter}
    \label{alpha_0_params_reaction}
\end{figure}

\vspace*{-0.5cm}
Figures \ref{alpha_1_params_reaction} and \ref{m_al_params_reaction} illustrate the effect of $\alpha_1$ and $m^{al}$ parameters while setting $\alpha_0 = 0.5$. {In Figure~\ref{alpha_1_params_reaction}, three values of $\alpha_1$ are tested:(0,0.5,1). }Adjusting the $\alpha_1$ parameter has little effect on the distribution of the {cumulative} number of attacks. However, by considering values such as $\alpha_0 = \alpha_1 = 0.5$, the insurer's maximum response capacity is still maintained to in most scenarios.

\begin{figure}[H]
    \begin{subfigure}{0.45\textwidth}
     \centering
\includegraphics[scale=0.3]{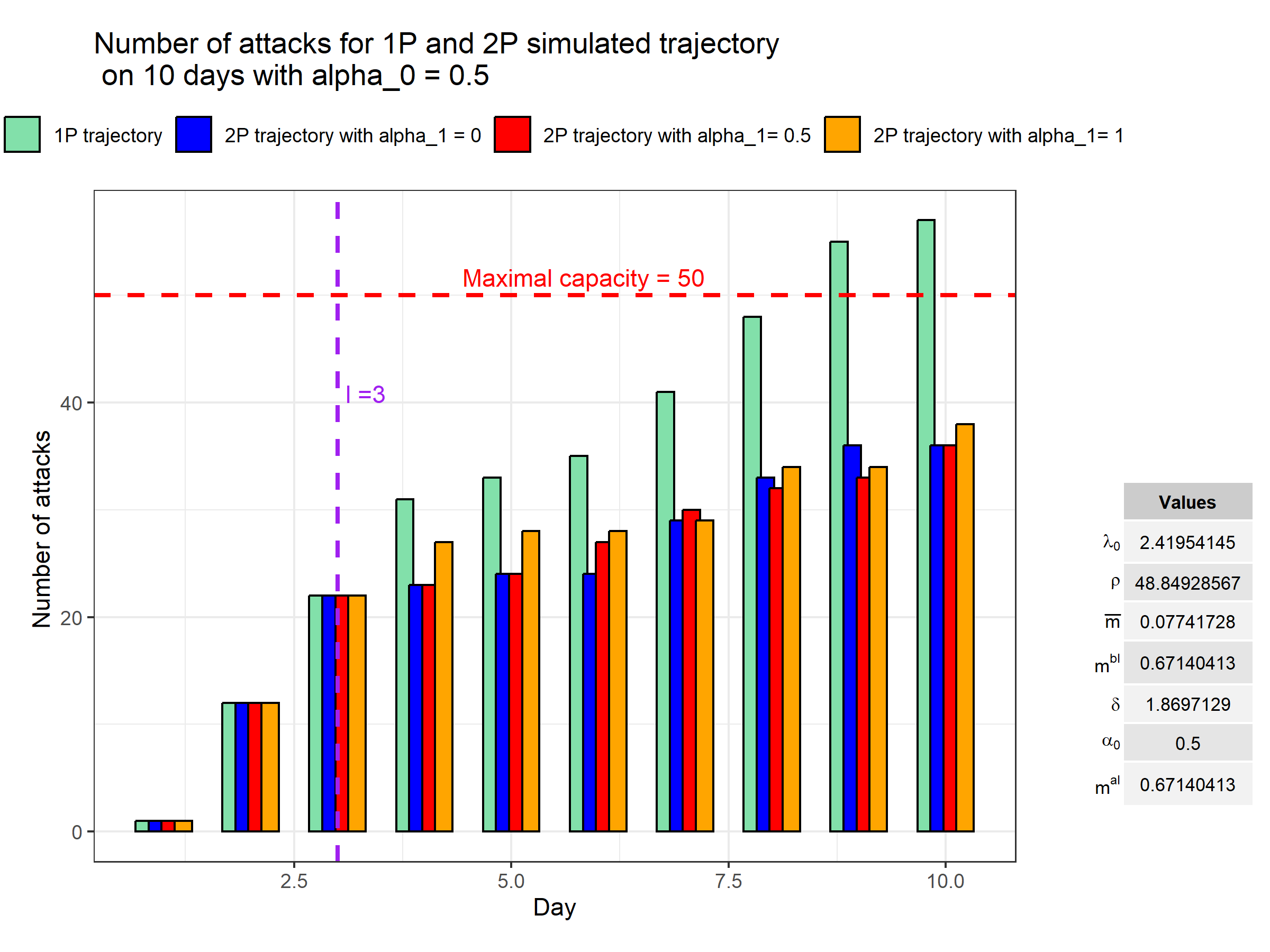}
\caption{Number of attacks predicted in 10 days for one trajectory}
    \end{subfigure}
    \hfill
    \begin{subfigure}{0.45\textwidth}
    \centering
\includegraphics[scale=0.3]{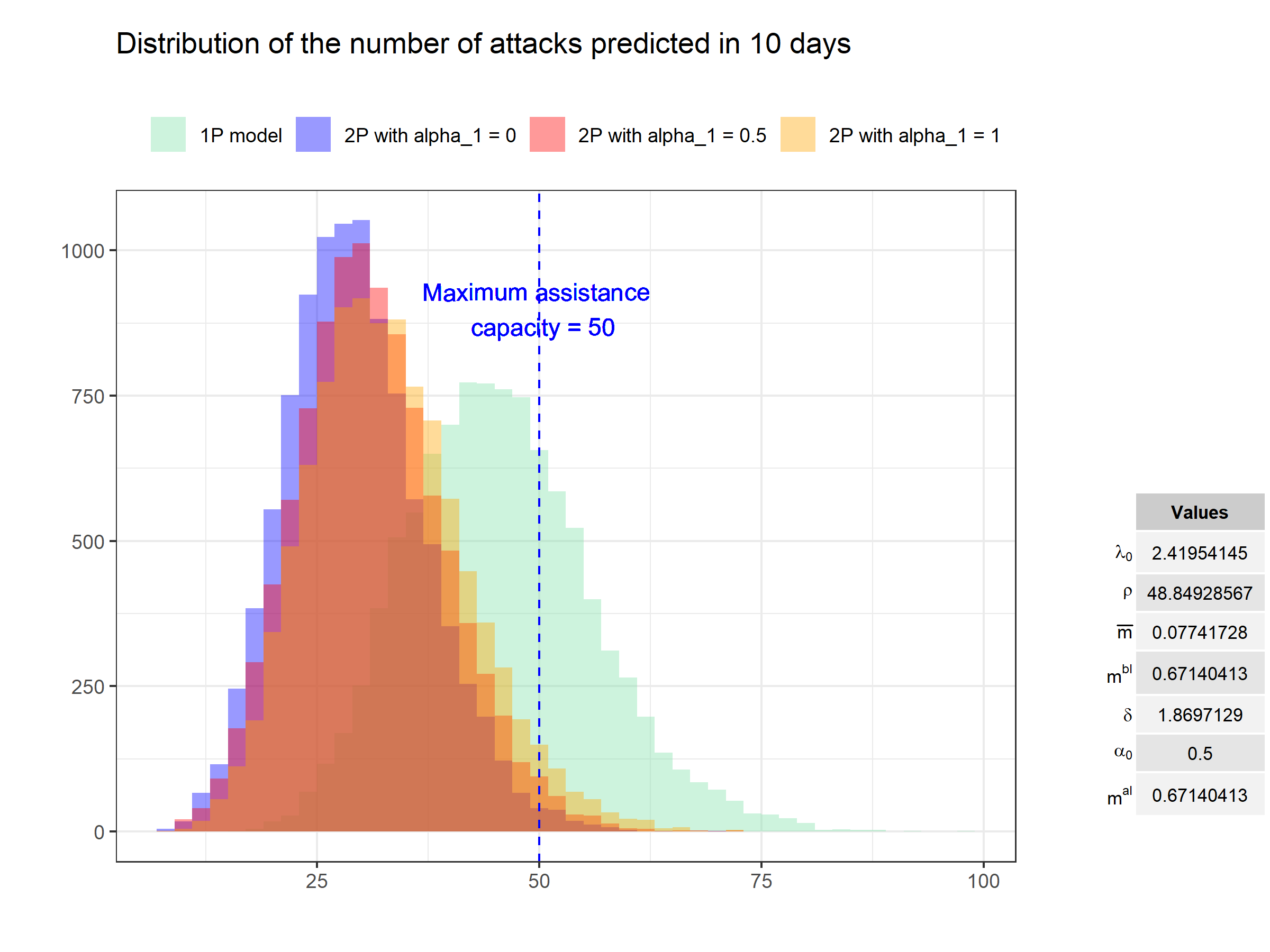}
\caption{Distribution of the number of attacks predicted in 10 days, simulation done with 10 000 trajectories}
    \end{subfigure}
    \caption{Effect of reaction measures on $\alpha_1$ parameter}
    \label{alpha_1_params_reaction}
\end{figure}

In the following, we illustrate the limited impact of adjusting the $m^{al}$ parameter on the distribution of the number of attacks in this specific example, since $m^{bl}$ is already quite small. Three configurations are tested, in blue $m^{al} = m^{bl}$, in red $m^{al} = \frac{m^{bl}}{2}$ and in yellow $m^{al} = \frac{m^{bl}}{4}$ in addition to the one-phase case in green.

\begin{figure}[H]
    \begin{subfigure}{0.45\textwidth}
     \centering
\includegraphics[scale=0.3]{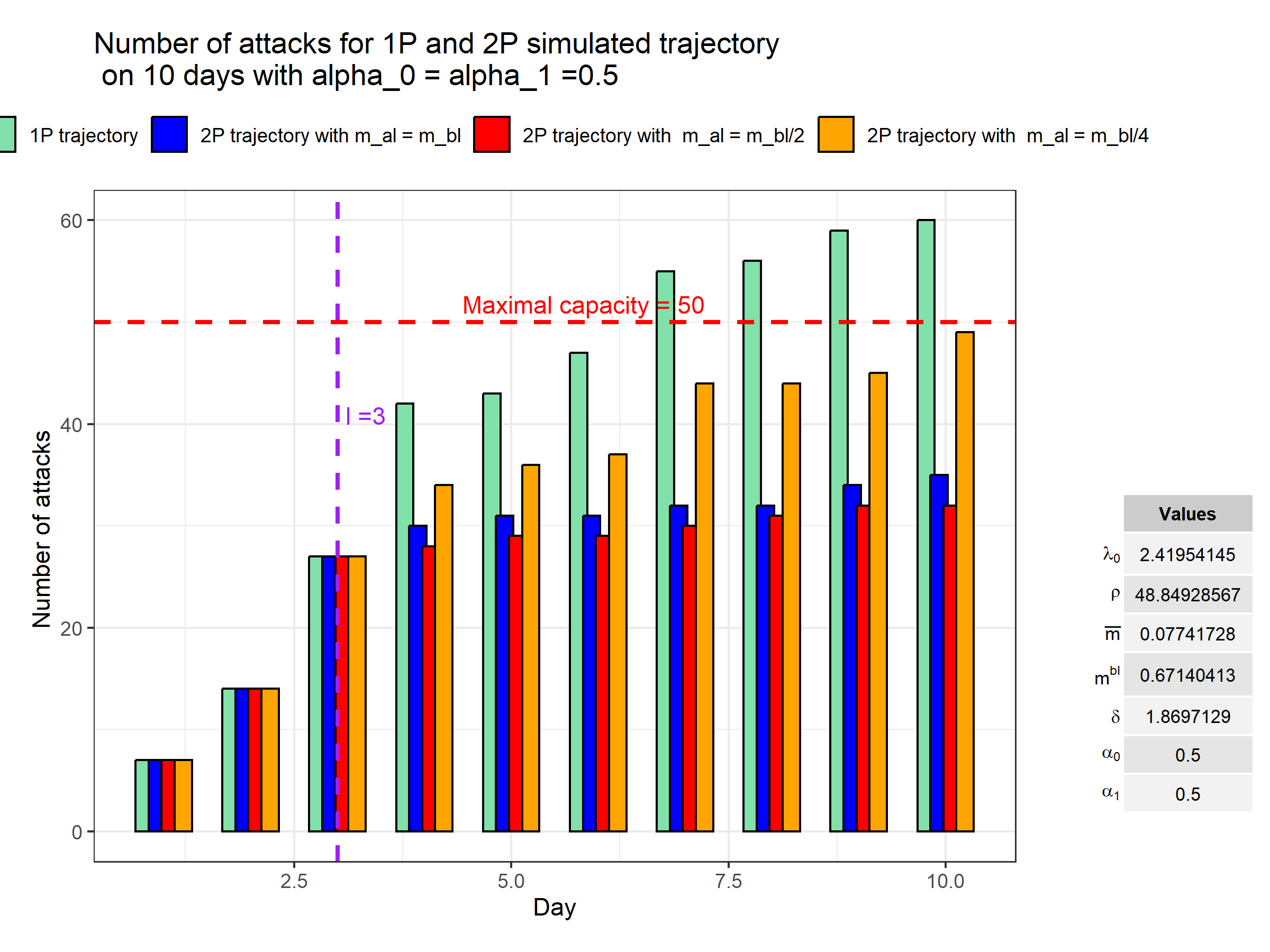}
\caption{Number of attacks predicted in 10 days for one trajectory}
    \end{subfigure}
    \hfill
    \begin{subfigure}{0.45\textwidth}
    \centering
\includegraphics[scale=0.3]{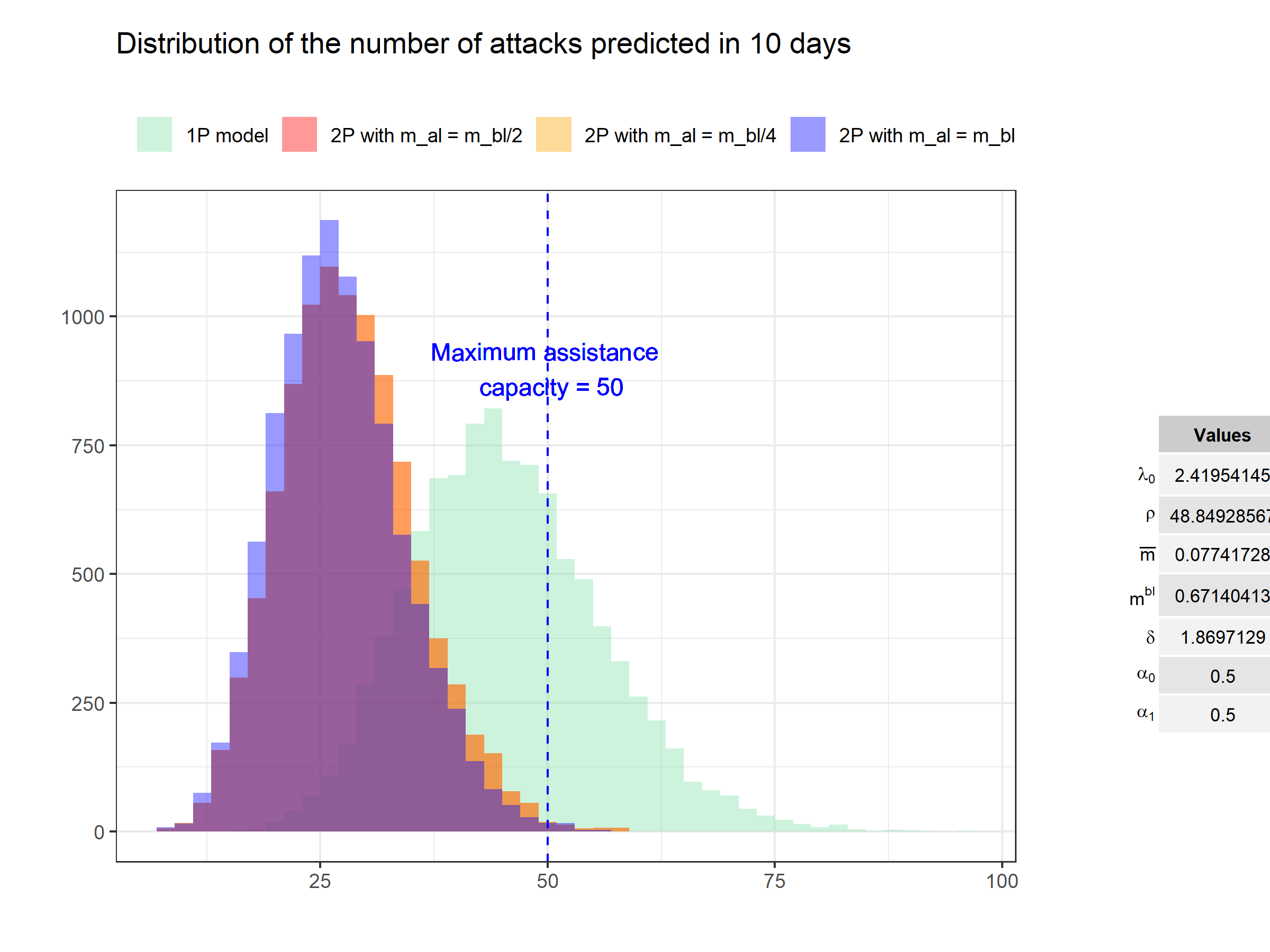}
\caption{Distribution of the number of attacks predicted in 10 days, simulation done with 10 000 trajectories}
    \end{subfigure}
    \caption{Effect of reaction measures on $m^{al}$ parameter}
    \label{m_al_params_reaction}
\end{figure}

By adjusting the parameters one by one, a parameter set to prevent the insurer from being overwhelmed could be $(\alpha_0 = 0.5, \alpha_1 = 0.5, m^{al} = m^{bl})$, the 99.5th percentile in this configuration is 49. It is important to note that adjusting these parameters do not cost the same. For example, setting $\alpha_0<1$
enhances preventive measures by training employees for better digital hygiene and increased awareness. This may come at a lesser cost than setting $\alpha_1<1$ and $m^{al}<m^{bl}$. In the latter configuration,  patching strategies would need to be implemented, and there might be a necessity to modify the network structure and IT infrastructure in order to disconnect certain computer links. Such alterations could lead to business interruptions in order to reduce the contagiousness of the event which may be more expensive. 

In what follows, our aim is to {choose a } set of parameters $(\alpha_0, \alpha_1, m^{al})$ while ensuring that the insurer's maximum capacity is not exceeded. We prioritize {selecting} the values of $\alpha_0$ and $\alpha_1$ while aiming to keep the event's contagiousness unchanged - meaning maintaining a small difference between $m^{al}$ and $m^{bl}$. 

\subsection{Response parameters selection} \label{optim_parameters_response} 

The aim of this section is to find an optimal set of reaction parameters $(\alpha_0, \alpha_1, m^{al})$ that ensures that the overall daily assistance capacity $C$ of 5 policyholders per day is not exceeded on average {during the response phase}. To achieve this, we use the closed-form formulas developed in Section~\ref{closed_formulas_expectation}. 

The insurer triggers the reaction phase at time $\ell = 3$ days. {As $(\mathbb{E}[N_{\ell}] -  C.\ell)^{+}$  of policyholders could not be assisted in the first phase, the daily assistance capacity in the second phase } is diminished with $\frac{\mathbb{E}[N_{\ell}] -  C.\ell}{\tau - \ell}$. We also recall that $\tau = 10$, which is the total duration of the pandemic. The insurer would want to {choose} $(\alpha_0, \alpha_1, m^{al})$  where  $0 < \alpha_0 \leq 1$, $0 < \alpha_1 \leq 1$ and $m^{al} \le m^{bl}$ such that 
$\left [
 (C - \frac{\mathbb{E}[N_{\ell}] - C.\ell }{\tau - \ell }) - \mathbb{E}[N_{t+1} - N_{t} ]\right ]$ remains  positive for all $t \ge \ell$.

 Since activating the different reaction parameters incurs different strategies as mentioned in the introduction of Section~\ref{section_forecasting}, the initial approach is to keep $m^{al} = m^{bl}$ and to search for the values of ($\alpha_0$,$\alpha_1$) in a grid within [0,1]. If no solution is found, then the selection procedure incorporates the parameter $m^{al}$ and searches for its value within ]0,$m^{bl}$[ using the same grid technique.

Figure~\ref{alpha_0_1_selection_area} illustrates the parameters sets that enable staying within the daily response capacity. The grey area represents the parameter region where the insurer's reaction capacity is exceeded. The insurer would need to choose a response set based on its constraints. For instance, one option is to select $\alpha_0$ and $\alpha_1$ from the efficient frontier represented in red in Figure~\ref{alpha_0_1_selection_area}. This choice could imply increasing employee vigilance to reduce spontaneous attacks while keeping $\alpha_1$ at 1 to delay immediate patching past attacks.

In this example, {$(\alpha_0,\alpha_1)$ are found without having to activate $m^{al}$ selection.  }

\begin{figure}[h]
    \centering
    \includegraphics[scale=0.45]{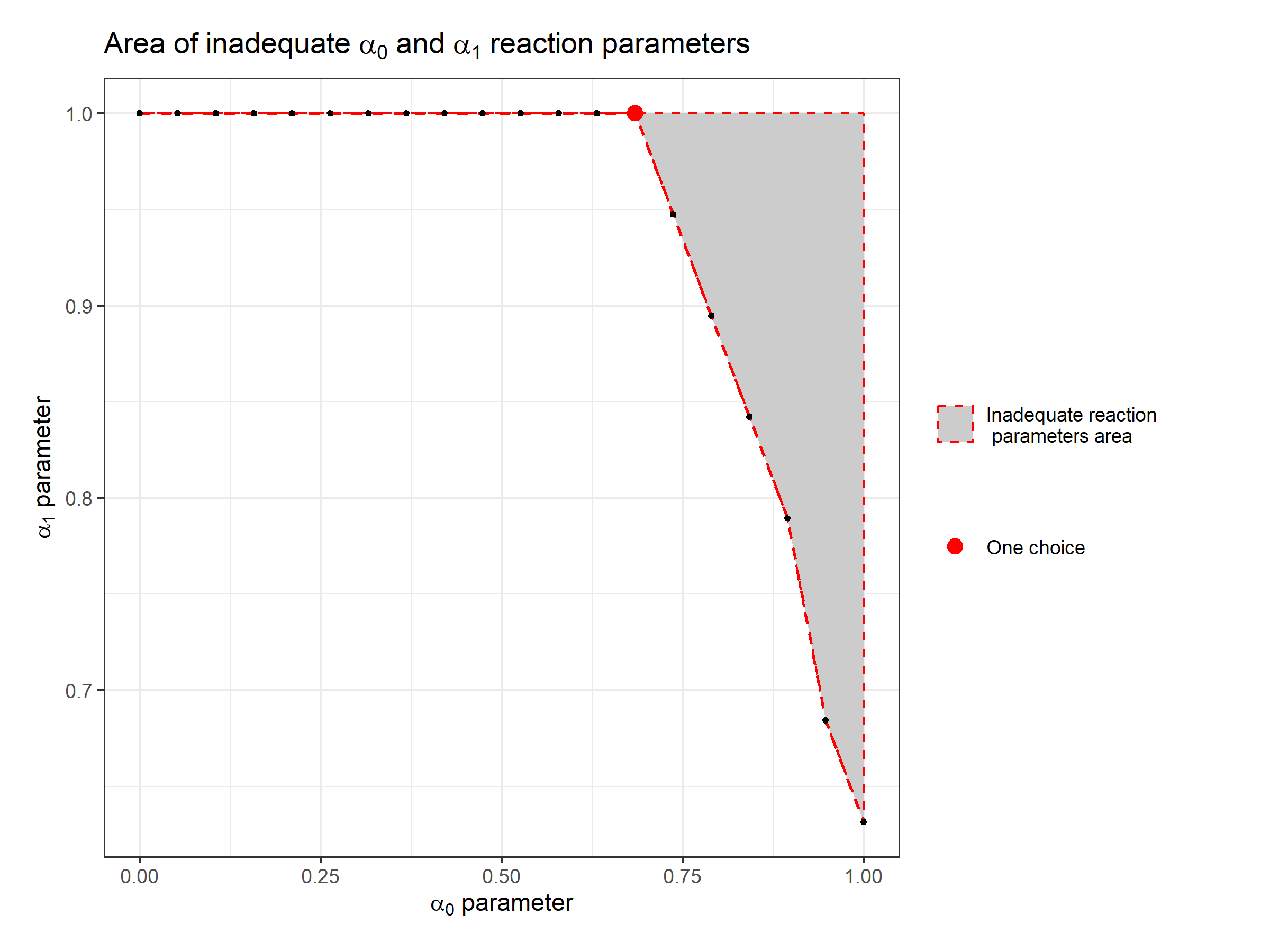}
    \caption{Area of inedequate $\alpha_0$ and $\alpha_1$ reaction parameters}
    \label{alpha_0_1_selection_area}
\end{figure}

 Figure~\ref{expected_nb_attacks_param_selction} displays the expected daily number of new attacks: the 1P case is represented in green, where no reaction phase is triggered, while the 2P case is illustrated in blue, using the identified $(\alpha_0,\alpha_1)$ denoted by the red point in Figure~\ref{alpha_0_1_selection_area}. This corresponds to the values $(\alpha_0 = 0.66, \alpha_1 = 1)$. {In dotted lines in red, the diminished daily assistance capacity $\frac{\mathbb{E}[N_{\ell}] - 
 C.\ell}{\tau - \ell} = 4.696$}.

\begin{figure}[h]
\centering
\includegraphics[scale=0.45]{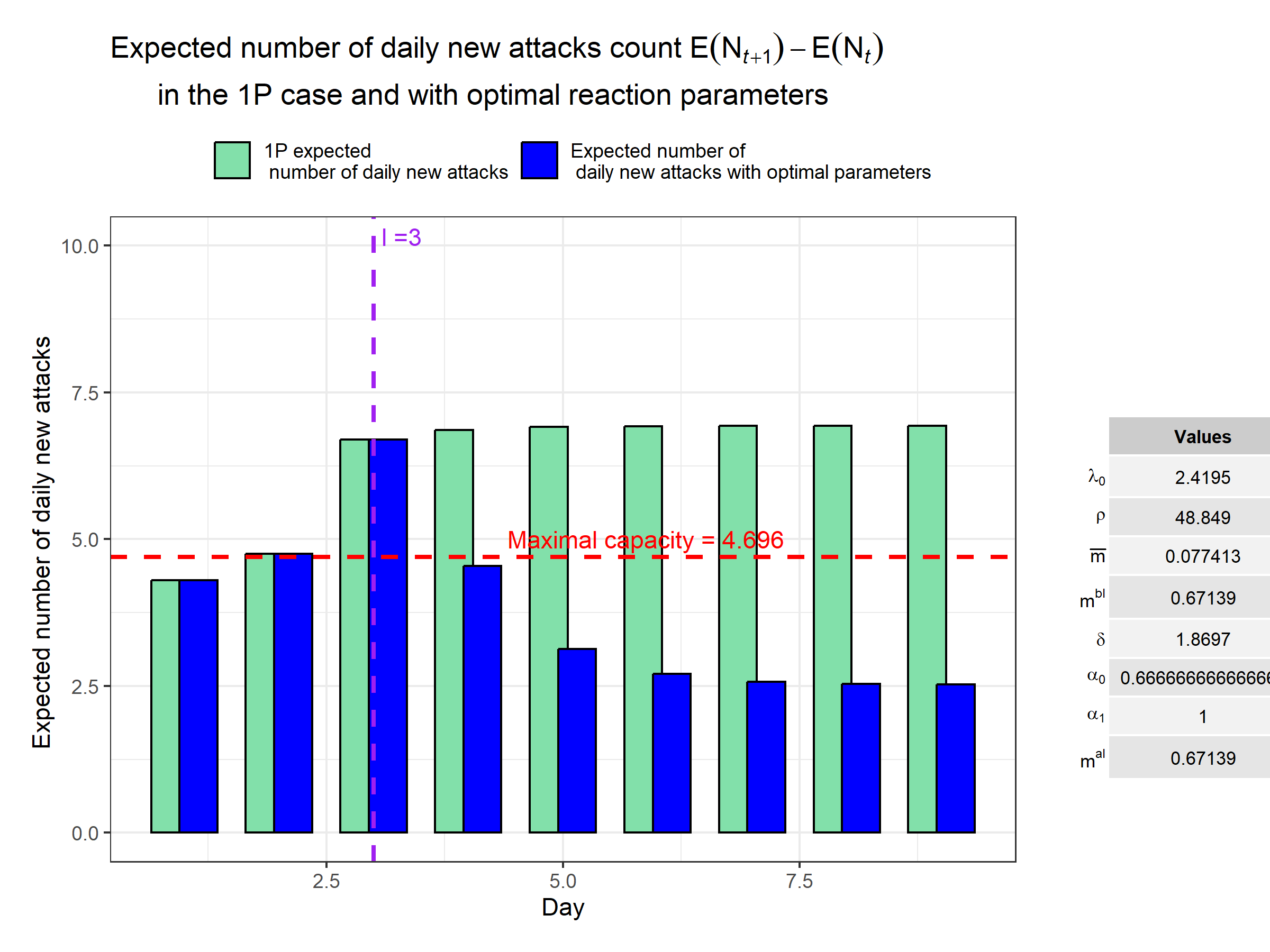}
\caption{Expected number of daily new attacks with optimal reaction parameters and in the 1P case}
\label{expected_nb_attacks_param_selction}
\end{figure}

The identified set of parameters ensures that the insurer's highest daily assistance capacity remains within limits in average throughout the reaction phase. 

The analysis carried out in this section is based on parameters obtained through the calibration using vulnerabilities from the NVD database. This particular scenario {is one where policyholders are exposed to multiple vulnerabilities and not all of them are exploited for launching attacks. Naturally, more extreme scenarios could be explored,} such as when the organization becomes a target, for example. In this case, vulnerabilities tend to be more exploited and the contagiousness of the attack is intensified. These extreme scenarios would allow us to assess how the reaction parameters perform under the most challenging circumstances.

\section{Conclusion}

This paper proposes a self-exciting model with external events to predict the arrivals of cyber attacks. This is achieved through Hawkes processes with external excitation. The latter capture the contagion of cyber events and the impact of cyber vulnerabilities disclosure on the dynamics of the cyber attacks process.
For this class of models, we have developed closed-form fomulas for the expectation of the 2P Hawkes process with external excitation based on the exponential kernels. Our proofs draw from population theory, providing a general framework that could be expanded to address other kernels. We first illustrate, through a simulated example, the crucial importance of incorporating external events into our model, as neglecting to do so could lead to a misidentification of the system's regime and results in an overestimation of its endogeneity. The analysis on real data is then conducted on the Hackmageddon database, KEV and NVD databases. We show that this degree of endogeneity can be halved by considering the appropriate external excitation found in the NVD database.

The proposed model has two phases. During the first phase, we have computer vulnerabilities that increase the intensity of the cyber attacks process, which can potentially lead to a clustering of cyber attacks. The second phase, on the other hand, is only activated if reactive measures need to be taken. Unlike in the case of a biological epidemy, the second phase here is used for a customized reaction on the level of an insurance portofolio. Using the Hackmageddon database, we considered a fictional insurance company with a limited known reaction capacity and investigated the best mix of strategies to manage claims without being overwhelmed. 

Moving forward, an interesting path to explore would be dynamic risk monitoring. The aim would be to find the best set of reaction parameters not just on average, but to fit on a trajectory that deviates from the expected one. Such dynamic risk management strategies would allow monitoring the peak in assistance requests during a cyber-pandemic scenario. Similarly, the time  $\ell$  of  reactive measures activation could be a random time: for example, $\ell$ could be chosen as the time at which the intensity hits a fixed  threshold or when the number of attacks surpasses a set limit that can be chosen as the insurer's maximum capacity for simultaneous daily assistance. In addition, the reaction parameters $\alpha_0$, $\alpha_1$, $m^{al}$ could be chosen dynamically using a stochastic optimization problem.

\begin{appendices}

\section{Computation of the expectation of the two-phase Hawkes process with external excitation} \label{annex_a}

The aim of this section is to detail the expectation of the two-phase Hawkes process with external excitation. The dynamics of $\mathbb{E}\left[\lambda_{t} \bigm| \mathcal{F}_s \right]$ is given in Proposition \ref{prop1}, then the solutions are provided in Proposition \ref{prop2}  and by the martingale property, Proposition \ref{prop3} is then deduced. 

\begin{proposition}
\label{prop1}
Let $0\le s<t$ and 
$ m^{bl} := \mathrm{E}[Y^{bl}_1]$, $ m^{al} := \mathrm{E}[Y^{al}_1]$ and $ \overline{m} := \mathrm{E}[\overline{Y}_1]$. 
$\mathbb{E}\left[\lambda_{t} \bigm| \mathcal{F}_s \right]$ satisfies the following dynamics: 
\begin{equation}
\mathbb{E}\left[\lambda_{t} \bigm| \mathcal{F}_s \right]=
\begin{cases}
\begin{aligned}
\lambda_s & + \delta \lambda_0 (t-s) - \left(\delta-m^{bl} \right) \int_{s}^{t} \mathbb{E}\left[\lambda_{u} \bigm| \mathcal{F}_s \right] \mathrm{~d}u \\ & + \rho \overline{m} (t-s) 
\end{aligned}
& \text{if } 0<s<t<\ell \\
\lambda_s + \alpha_0 \delta \lambda_0 (t-s) - \left(\delta-m^{al} \right) \int_{s}^{t} \mathbb{E}\left[\lambda_{u} \bigm| \mathcal{F}_s \right] \mathrm{~d}u & \text{if } 0<\ell \le s<t 
\end{cases}
\label{ODES_lambda_t}
\end{equation}
\end{proposition} 

\begin{proof}
We give here the proof to find the two ODEs {given by (\ref{ODES_lambda_t})} for $s<t<\ell$ and for $\ell \le s<t$. Regarding the case of $s<\ell<t$, we provide the expression of $\mathbb{E}\left[ \lambda_t \bigm| \mathcal{F}_s\right]$  in equation (\ref{eq_conditional_expectation_case_3}) by using the tower property $\mathbb{E}\left[ \lambda_t \bigm| \mathcal{F}_s\right]= \mathbb{E}\left[\mathbb{E}\left[ \lambda_t \bigm| \mathcal{F}_{\ell}\right] \bigm|  \mathcal{F}_{s} \right] .$ We start by writing the intensity as: 

\begin{equation*}
\lambda_t = \begin{cases}  
\lambda_0 + \sum_{\overline{T_k} < t} \overline{\phi}\left(t-\overline{T_k}, \overline{Y_k}\right)+\sum_{T_i < t} \phi\left(t-T_i, Y^{bl}_i\right) & \text{if } 0<t <\ell \\ 
\alpha_0 \lambda_{0} + \alpha_1 (\lambda_{\ell^{-}} - \lambda_0) & \text{if } t = \ell \\ 
\alpha_0 \lambda_0  + \alpha_1 (\lambda_{\ell^{-}} - \lambda_0) e^{-\delta(t-\ell)}   +  \sum_{\ell<T_i < t} \phi\left(t-T_i, Y^{al}_i\right) & \text{if } t > \ell 
\end{cases}
\end{equation*}

$
\begin{aligned}
\text{We take here } \quad
 \overline{\phi}(a, x)=\phi(a, x)=x e^{-\delta a}.
\end{aligned}
$
This can be generalized to different kernels, such as the  kernel with delay for example as in \cite{Multi-Variate_HawkesBessy}. We introduce then the following measure-valued processes:
\begin{equation*}
\begin{aligned}
& \overline{Z_t}(\mathrm{d}a,\mathrm{d}x)=\sum_{\overline{T_k} \le t} \delta_{\left(t-\overline{T_k}, \overline{Y_k}\right)}\left(\mathrm{d}a, \mathrm{d}x\right) \\
& Z_t(\mathrm{d}a,\mathrm{d}x)= \begin{cases}
    \sum_{T_i \le t} \delta_{(t-T_i, Y^{bl}_i)}\left(\mathrm{d}a,\mathrm{d}x\right)  & \text{if } 0 < t < \ell \\
    \sum_{T_i \le \ell} \delta_{(t-T_i, Y^{bl}_i)}\left(\mathrm{d}a,\mathrm{d}x\right) + \sum_{\ell<T_i \le t} \delta_{(t-T_i, Y^{al}_i)}\left(\mathrm{d}a,\mathrm{d}x\right) & \text{if } t \ge \ell
\end{cases}
\end{aligned}
\end{equation*}

and the following notations $\langle Z_t, f \rangle=\int_{\mathbb{R}^{+}} \int_{\mathbb{R}^{+}} f\left(a,x\right) Z_t\left(\mathrm{~d} a, \mathrm{~d} x\right)$ and $ \langle \overline{Z_t}, f \rangle=\int_{\mathbb{R}^{+}} \int_{\mathbb{R}^{+}} f\left(a,x\right) \overline{Z_t}\left(\mathrm{~d}  a, \mathrm{~d} x\right) $. 
 For example, the Hawkes process is $N_t = \langle Z_t, \mathbf{1}\rangle $, whereas the number of external shocks is $\overline{N_t}= \langle \overline{Z_t}, \mathbf{1}\rangle$. Also, the intensity $\lambda_t$ of the Hawkes process $N_t$ can be written as: 
\begin{equation*}
\lambda_{t} =
\begin{cases}
\lambda_0 + \langle \overline{Z_{t^{-}}}, \overline{\phi} \rangle + \langle Z_{t^{-}}, \phi\rangle  & \text{if } 0<t <\ell \\
\alpha_0 \lambda_{0} + \alpha_1 (\lambda_{\ell^{-}} - \lambda_0) & \text{if } t = \ell \\ 
\alpha_0 \lambda_0 + \alpha_1 e^{-\delta (t-\ell)} \langle \overline{Z_{\ell}}, \overline{\phi} \rangle + (\alpha_1 - 1) e^{-\delta (t-\ell)} \langle Z_{\ell}, \phi \rangle + \langle Z_{t^{-}}, \phi \rangle & \text{if } 0<\ell<t .
\end{cases} 
\end{equation*}
Let $0\le s<t$, we introduce $ \overline{v}_t = \mathbb{E}[\langle \overline{Z_t}, \overline{\phi}  \rangle \bigm| \mathcal{F}_s ]$ and
$
v_t = \mathbb{E}[\langle Z_t, \phi  \rangle \bigm| \mathcal{F}_s ].
$ {The goal here is to determine the ODE satisfied by $\mathbb{E}\left[\lambda_{t} \bigm| \mathcal{F}_s \right]$ by investigating the ODEs satisfied by $v_t$ and $\overline{v}_t$.}
Let $N$ and $\overline{N}$ be two random point measures where: \\ 
$\overline{N}(\mathrm{d} t, \mathrm{d} x) = \sum_{k \geq 1} \delta_{\left(\overline{T_k}, \overline{Y}_k\right)}(\mathrm{d} t, \mathrm{d} x)$ and $N(\mathrm{d} t, \mathrm{d} x) = \sum_{n \geq 1} \delta_{\left(T_n, Y_n\right)}(\mathrm{d} t, \mathrm{d} x)$.

Following (5) in \cite{boumezoued2016population}, and as $\phi$ is the exponential kernel with decay parameter $\delta$:
\begin{eqnarray*}
\mathrm{d}\langle Z_t, \phi\rangle&=&\int_{x \in \mathbb{R^{+}}} \left[ \phi(0, x) N(\mathrm{d}t, \mathrm{d}x) \right] +\langle Z_t, \frac{\partial \phi}{\partial a} \rangle \mathrm{d}t\\
&=&\int_{x \in \mathbb{R^{+}}} \left[ \phi(0, x) N(\mathrm{d}t, \mathrm{d}x) \right] -\delta \langle Z_t, \phi \rangle \mathrm{d}t 
\end{eqnarray*}
Then, using $\int_{s}^{t} \mathrm{d}\left \langle Z_{\tau}, \phi \right \rangle  = 
\left \langle Z_t, \phi \right \rangle - \left \langle Z_s, \phi \right \rangle$ and $\phi(0,x) = x$:
\begin{equation}\label{eqv1}
\mathbb{E}\left[\left \langle Z_t, \phi \right \rangle \bigm| \mathcal{F}_s\right] = \left \langle Z_s, \phi \right \rangle + \mathbb{E}\left[ \int_{s}^t \left[ \int_{x \in \mathbb{R^{+}}} x N(\mathrm{d}u, \mathrm{d}x) -\delta \langle Z_u, \phi\rangle \mathrm{d}u \right] \bigm| \mathcal{F}_s \right]  
\end{equation}
In the same way:
\begin{equation}\label{eqv2}
\mathbb{E}\left[ \langle \overline{Z_t}, \overline{\phi}  \rangle \bigm| \mathcal{F}_s\right]  =  \langle \overline{Z_s}, \overline{\phi}  \rangle + \mathbb{E}\left[ \int_{s}^t \left[ \int_{x \in \mathbb{R^{+}}} x \overline{N}(\mathrm{d}u, \mathrm{d}x) -\delta \langle \overline{Z_u}, \overline{\phi}\rangle \mathrm{d}u \right] \bigm| \mathcal{F}_s \right] 
\end{equation}
Since the compensating measure of $\overline{N}(\mathrm{d} t, \mathrm{~d} x)$ is $\rho \overline{F}(\mathrm{d}x)\mathrm{d}t$ and that of  $N(\mathrm{d} t, \mathrm{~d} x)$ is  $\lambda_t G_t(\mathrm{d}x)\mathrm{d}t$, where  $G_t := G^{bl} \mathrm{1}_{t < \ell} + G^{al} \mathrm{1}_{t>\ell}$, 
the two processes  $X_t = \int_{0}^t\int_{x \in \mathbb{R^{+}}} x \left[ N(\mathrm{d}u, \mathrm{d}x) - \lambda_u G_u(\mathrm{d}x) \mathrm{d}u \right] $ and $\overline{X_t} = \int_{0}^t\int_{x \in \mathbb{R^{+}}} x \left[ \overline{N}(\mathrm{d}u, \mathrm{d}x) - \rho \overline{F}(\mathrm{d}x) \mathrm{d}u \right] $ are two 
martingales. Therefore: 
\begin{equation*}
 \begin{aligned}
    \mathbb{E} \left[ \int_{s}^t\int_{x \in \mathbb{R^{+}}} x \left[ N(\mathrm{d}u, \mathrm{d}x) \right] \bigm| \mathcal{F}_s  \right] & =  \mathbb{E}\left[ \int_{s}^t \int_{x \in \mathbb{R^{+}}}  x G_u(\mathrm{d}x) \lambda_u \mathrm{d}u \bigm| \mathcal{F}_s \right] \\ 
     \mathbb{E} \left[ \int_{s}^t\int_{x \in \mathbb{R^{+}}} x \left[ \overline{N}(\mathrm{d}u, \mathrm{d}x) \right] \bigm| \mathcal{F}_s  \right] & = \overline{m} \rho (t-s) \\ 
\end{aligned}   
\end{equation*}

Then \eqref{eqv1} and \eqref{eqv2} can be rewritten respectively

\begin{equation*}
\begin{aligned}
    v_t    & =  \langle Z_s, \phi\rangle +  \int_s^{t} \int_{x \in \mathbb{R^{+}}} x G_u(\mathrm{d}x) \mathbb{E} \left[  \lambda_u \bigm| \mathcal{F}_s \right]  \mathrm{d}u - \delta \int_s^t v_\tau \mathrm{d}\tau \\ 
    \overline{v}_t  & =  \langle \overline{Z}_s, \overline{\phi}\rangle +  \rho \overline{m} (t-s) - \delta \int_s^t \overline{v}_u \mathrm{d}u
\end{aligned}
\end{equation*}

and 
\begin{equation*}
    v_t = \begin{cases}
         \langle Z_s, \phi\rangle + m^{bl} \int_s^t \mathbb{E} \left[  \lambda_u \bigm| \mathcal{F}_s \right] \mathrm{d}u - \delta \int_s^t v_\tau \mathrm{d}\tau & \text {if }  0<s<t <\ell \\ 
         \langle Z_s, \phi\rangle + m^{al} \int_s^t \mathbb{E} \left[  \lambda_u \bigm| \mathcal{F}_s \right] \mathrm{d}u - \delta \int_s^t v_\tau \mathrm{d}\tau & \text {if }  0<\ell \le s<t  
    \end{cases}
\end{equation*}

Putting this together leads to:
{
\small
\begin{equation*}
    \mathbb{E}\left[\lambda_{t} \bigm| \mathcal{F}_s \right] =
    \begin{cases}
    \begin{aligned}
          \lambda_0 &  + \langle Z_s, \phi\rangle + m^{bl} \int_s^{t} \mathbb{E} \left[  \lambda_u \bigm| \mathcal{F}_s \right] \mathrm{d}u \\ 
         & - \delta \int_s^t v_\tau \mathrm{d}\tau + \langle \overline{Z}_s, \overline{\phi}\rangle +  \rho \overline{m} (t-s) - \delta \int_s^t \overline{v}_u \mathrm{d}u &  \text{if } 0<s<t <\ell
       \\   
    \alpha_0 \lambda_0 & + \langle Z_s, \phi\rangle + m^{al} \int_s^{t} \mathbb{E} \left[  \lambda_u \bigm| \mathcal{F}_s \right] \mathrm{d}u \\ 
    &- \delta \int_s^t v_\tau \mathrm{d}\tau +  \alpha_1 e^{-\delta (t-\ell)}  \overline{v}(\ell^{-}) \\ 
         & + (\alpha_1-1) e^{-\delta (t-\ell)} v(\ell^{-}) &  \text{if } 0<\ell \le s<t 
    \end{aligned}
\end{cases} 
\end{equation*}
}

By  rearranging the different terms we have the following: 

{\small
\begin{equation*}
\mathbb{E}\left[\lambda_{t} \bigm| \mathcal{F}_s \right]=
\begin{cases}
\lambda_s + \delta \lambda_0 (t-s) - \left(\delta-m^{bl} \right) \int_{s}^{t} \mathbb{E}\left[\lambda_{u} \bigm| \mathcal{F}_s \right] \mathrm{~d}u + \rho \overline{m} (t-s)  & \text{if } 0<s<t<\ell \\
\lambda_s + \alpha_0 \delta \lambda_0 (t-s) - \left(\delta-m^{al} \right) \int_{s}^{t} \mathbb{E}\left[\lambda_{u} \bigm| \mathcal{F}_s \right] \mathrm{~d}u & \text{if } 0<\ell \le s<t 
\end{cases}
\end{equation*}
}

\end{proof}

\begin{proposition}
\label{prop2}
    The conditional expectation of  $\lambda_t$ given $\mathcal{F}_s $ for $0<s<t<\ell$ is:
\begin{equation*}
\mathbb{E}\left[\lambda_{t} \bigm| \mathcal{F}_s \right] = 
\begin{cases}
\lambda_s + (\delta \lambda_0  + \rho \overline{m}) (t-s) & \text {if }  \delta=m^{bl}
\\ 
\frac{\rho \overline{m} + \delta \lambda_0}{\delta - m^{bl}} + (\lambda_s - \frac{\rho \overline{m} + \delta \lambda_0}{\delta - m^{bl}} )e^{- (\delta - m^{bl})(t- s)} & \text {if } \delta \neq m^{bl} 
\end{cases} 
\end{equation*}

The conditional expectation of $\lambda_t$ given $\mathcal{F}_s $ for $\ell<s<t$ is:
\begin{equation}
\mathbb{E}\left[\lambda_{t} \bigm| \mathcal{F}_s \right] = 
\begin{cases}
\lambda_s + \alpha_0 \delta \lambda_0 (t-s) & \text {if }  \delta=m^{al}
\\ 
\frac{\alpha_0 \delta \lambda_0}{\delta - m^{al}} + (\lambda_s - \frac{\alpha_0 \delta \lambda_0}{\delta - m^{al}} )e^{- (\delta - m^{al})(t- s)} & \text {if } \delta \neq m^{al} 
\end{cases} 
\label{eq_conditional_expectation_case_2}
\end{equation}

The conditional expectation of $\lambda_t$ given $\mathcal{F}_s $ for $s<\ell<t$ is:
{
\small
\begin{equation}
\mathbb{E}\left[\lambda_{t} \bigm| \mathcal{F}_s \right] = 
\begin{cases}
\alpha_0 \delta \lambda_0 (t-\ell) +\lambda_0 (\alpha_0 - \alpha_1) + \alpha_1 \mathbb{E}\left[\lambda_{\ell^{-}} \bigm| \mathcal{F}_s \right] & \text {if }  \delta = m^{al} 
\\ 
\frac{\alpha_0 \delta \lambda_0}{\delta - m^{al}} + ( (\alpha_0 - \alpha_1) \lambda_0 +\alpha_1 \mathbb{E}\left[\lambda_{\ell^{-}} \bigm| \mathcal{F}_s \right] - \frac{\alpha_0 \delta \lambda_0}{\delta - m^{al}} )e^{- (\delta - m^{al})(t- \ell)} & \text {if } \delta \ne  m^{al} \\ 
\end{cases} 
\label{eq_conditional_expectation_case_3}
\end{equation}
with: 
\begin{equation*}
\mathbb{E}\left[\lambda_{\ell^{-}} \bigm| \mathcal{F}_s \right] = 
\begin{cases}
\lambda_s + (\delta \lambda_0  + \rho \overline{m}) (\ell-s) & \text {if }  \delta=m^{bl}
\\ 
\frac{\rho \overline{m} + \delta \lambda_0}{\delta - m^{bl}} + (\lambda_s - \frac{\rho \overline{m} + \delta \lambda_0}{\delta - m^{bl}} )e^{- (\delta - m^{bl})(\ell- s)} & \text {if } \delta \neq m^{bl} 
\end{cases} 
\end{equation*}
}
\end{proposition}

\begin{proof}
    By solving the two first order ODEs in equation (\ref{ODES_lambda_t}), we find (\ref{eq_conditional_expectation_case_2}) and (\ref{eq_conditional_expectation_case_3}). Then Relation (\ref{eq_conditional_expectation_case_3}) follows using $\mathbb{E}\left[ \lambda_t \bigm| \mathcal{F}_s\right]= \mathbb{E}\left[\mathbb{E}\left[ \lambda_t \bigm| \mathcal{F}_{\ell}\right] \bigm|  \mathcal{F}_{s} \right] $. 
\end{proof}

\begin{proposition}
\label{prop_annewe}
    The conditional expectation of  $N_t$ given $\mathcal{F}_s $ for $0<s<t<\ell$ is:
    \begin{equation*}
    \mathbb{E}\left[N_{t} \bigm| \mathcal{F}_s \right] = 
\begin{cases}
N_{s}+\lambda_{s}(t-s)+\frac{1}{2} (\rho \overline{m} + \delta \lambda_0) (t-s)^{2} & \text {if }  \delta=m^{bl}
\\ 
N_{s}+\frac{(\rho \overline{m} + \delta \lambda_0) }{\delta - m^{bl}} (t-s)+\left(\lambda_{s}-\frac{\rho \overline{m} + \delta \lambda_0 }{\delta - m^{bl}}\right) \frac{1-e^{-(\delta - m^{bl})(t-s)}}{\delta - m^{bl}} & \text {if } \delta \neq m^{bl} 
\end{cases} 
    \end{equation*}
    The conditional expectation of the process $N_t$ given $\mathcal{F}_s $ for $0<\ell<s<t$  is:
\begin{equation*}
\mathbb{E}\left[N_{t} \bigm| \mathcal{F}_s \right] = 
\begin{cases}
N_{s}+\lambda_{s}(t-s)+\frac{1}{2} \alpha_0 \delta \lambda_0 (t-s)^{2} & \text {if }  \delta=m^{bl}
\\ 
N_{s}+\frac{\alpha_0 \delta \lambda_0 }{\delta - m^{bl}} (t-s)+\left(\lambda_{s}-\frac{\alpha_0 \delta \lambda_0 }{\delta - m^{bl}}\right) \frac{1-e^{-(\delta - m^{bl})(t-s)}}{\delta - m^{bl}} & \text {if } \delta \neq m^{bl} 
\end{cases} 
\end{equation*}

The conditional expectation of $N_t$ given $\mathcal{F}_s $ for $0<s<\ell<t$ is:
{
\scriptsize
\begin{equation*}
\mathbb{E}\left[N_{t} \bigm| \mathcal{F}_s \right] = 
\begin{cases}
\begin{aligned}
\mathbb{E}\left[N_{\ell} \bigm| \mathcal{F}_s \right] & + \frac{\alpha_0 \delta \lambda_0}{2} (t-\ell) ^2 + \lambda_0 (\alpha_0 - \alpha_1) (t-\ell) \\ & + \alpha_1 \mathbb{E}\left[\lambda_{\ell^{-}}  \bigm| \mathcal{F}_s \right] (t-\ell) & \text {if }  \delta = m^{al} 
\\ 
\mathbb{E}\left[N_{\ell} \bigm| \mathcal{F}_s \right]  & + 
\frac{\alpha_0 \delta \lambda_0}{\delta - m^{al}} (t-\ell)  \\ & + \left( (\alpha_0 - \alpha_1) \lambda_0 + \alpha_1 \mathbb{E}\left[\lambda_{\ell^{-}} \bigm| \mathcal{F}_s \right] - \frac{\alpha_0 \delta \lambda_0}{\delta - m^{al}}\right) \\ & \frac{1}{(\delta - m^{al})}(1 - e^{- (\delta - m^{al})(t- \ell)}) & \text {if } \delta \ne  m^{al} \\ 
\end{aligned}
\end{cases} 
\end{equation*}
}
with  
\begin{equation*}
\mathbb{E}\left[\lambda_{\ell^{-}} \bigm| \mathcal{F}_s \right] = 
\begin{cases}
\lambda_s + (\delta \lambda_0  + \rho \overline{m}) (\ell-s) & \text {if }  \delta=m^{bl}
\\ 
\frac{\rho \overline{m} + \delta \lambda_0}{\delta - m^{bl}} + (\lambda_s - \frac{\rho \overline{m} + \delta \lambda_0}{\delta - m^{bl}} )e^{- (\delta - m^{bl})(\ell- s)} & \text {if } \delta \neq m^{bl} 
\end{cases} 
\end{equation*}
and
\color{black}
\begin{equation*}
\mathbb{E}\left[N_{\ell} \bigm| \mathcal{F}_s \right] = 
\begin{cases}
N_{s}+\lambda_{s}(\ell-s)+\frac{1}{2} (\rho \overline{m} + \delta \lambda_0) (\ell-s)^{2} & \text {if }  \delta=m^{bl}
\\ 
N_{s}+\frac{(\rho \overline{m} + \delta \lambda_0) }{\delta - m^{bl}} (\ell-s)+\left(\lambda_{s}-\frac{\rho \overline{m} + \delta \lambda_0 }{\delta - m^{bl}}\right) \frac{1-e^{-(\delta - m^{bl})(\ell-s)}}{\delta - m^{bl}} & \text {if } \delta \neq m^{bl} 
\end{cases} 
\end{equation*}    
\end{proposition}

\begin{proof}
 Using  the martingale property of  the compensated process, we have \\ for $0<s \le \ell<t$: 
    \begin{eqnarray*}
        \mathbb{E}\left[N_t \bigm| \mathcal{F}_s \right]  &=&
N_s + \int_s^{t}\mathbb{E}\left[ \lambda_u \bigm| \mathcal{F}_s\right]\mathrm{~d} u  \\
  &=& N_s + \int_s^{\ell}\mathbb{E}\left[ \lambda_u \bigm| \mathcal{F}_s \right]\mathrm{~d} u + \int_{\ell}^{t}\mathbb{E}\left[ \lambda_u \bigm| \mathcal{F}_s \right] \mathrm{~d} u  \\
    & =& \mathbb{E}\left[N_{\ell} \bigm| \mathcal{F}_s \right] + \int_{\ell}^{t}\mathbb{E}\left[ \lambda_u \bigm| \mathcal{F}_s \right]\mathrm{~d} u.
      \end{eqnarray*}    
\end{proof}

\section{MSE Calibration strategies} \label{annex_mse_calibration_strategies}

An attempt to improve MSE results was made through the following strategies: 

\begin{enumerate}
    \item The first estimation procedure attempts to simultaneously estimate all 5 parameters $(\lambda_0, \rho, \overline{m}, m, \delta)$. However, the results shown in Table~\ref{five_sim_calibration} indicate that the estimates are not accurate. 
    
    \item To address the accuracy issue, a second approach is tested. This approach involves isolating the estimation of the $\rho$ parameter by leveraging external events and subsequently incorporating its value into the minimization problem. The focus of estimation then shifts to the remaining four parameters, namely $(\lambda_0, \overline{m}, m, \delta)$. This unfortunately did not improve the accuracy of the estimates, as shown in Table~\ref{4_param_given_rho}.
    
    \item The third approach in Table~\ref{4_param_given_rho_m_bar} resembles the second one, with the distinction of incorporating not only $\rho$ but also $\rho \overline{m}$ into the minimization problem. This redefines the optimization targets as $(\lambda_0, \rho, m, \delta)$, while the value of $\overline{m}$ is derived by dividing $\rho \overline{m}$ by $\rho$. This allows a refined estimation of the $\lambda_0$ parameter.
\end{enumerate}

\subsection{Simultaneous estimation of the 5 parameters} \label{sim_example_results}

{In Table~\ref{five_sim_calibration}, the results of the first approach, which aims to perform a simultaneous estimation of the five parameters, are displayed. Among these results, $\lVert \phi \rVert$ and $\frac{\rho \lVert \overline{\phi} \rVert + \lambda_0}{1 - \lVert \phi \rVert}$ are relevant. This table also includes the opposite of the log-likelihood values, the MSE outcomes, and the average computational time taken for 1000 calibration runs. The computation of the 95 \% confidence intervals for the parameters estimated using the likelihood method are detailed in Appendix \ref{annex_ic_likelihood}.
}

\clearpage
\begin{sidewaystable}[h]
\caption{Calibration results on a simulated example: simultaneous estimation of the 5 parameters}
\label{five_sim_calibration}
\begin{tabular*}{\textheight}{@{\extracolsep\fill}p{2cm}p{1cm}p{1cm}p{1cm}p{1cm}p{1cm}p{1cm}p{1cm}p{1cm}p{1cm}p{1cm}p{1cm}}
\toprule
& $\boldsymbol{\lambda_0}$ & $\boldsymbol{\rho}$ & $\boldsymbol{\overline{m}}$ & $\boldsymbol{m}$ & $\boldsymbol{\delta}$ & ${\lVert \phi \rVert}$ & $\boldsymbol{\frac{\rho \hspace{1mm} \lVert \overline{\phi} \rVert + \lambda_0}{1 - \lVert \phi \rVert}}$ & $\boldsymbol{-ln(\mathcal{L})}$ & $\frac{\textbf{MSE}^{\textbf{Ext}}}{(N_{\tau}-N_s)}$ & $\frac{\textbf{MSE}^{\textbf{Int}}}{(N_{\tau} - N_s)}$ & \textbf{Av. Exec. time in (s) on 1000 runs} \\
\midrule
\textbf{True values} & 0.6 & 0.2 & 0.8 & 0.5 & 1.5 & 0.33 & 1.06 & 783.55 &  0.726\% & 0.653 \% & - \\ 
\textbf{Likelihood method} & 0.59 & 0.19 & 0.82 & 0.51 & 1.44 & 0.36 & 1.09 & \bf{784.22} & 0.771\% & 0.756 \% & 22.14 \\ 
\textbf{95 \% C.I} & [0.49,0.72] & [0.17,0.23] & [0.47,1.15] & [0.33,0.68] & [1.02,2.04] & - & - & - & - & - & - \\ 
$\textbf{MSE}^{\textbf{Int}}$ & 0.022 & 0.39 & 1.49e-05 & 0.263 & 0.27 & 0.97 & 1.065 & 2947.28 & 0.823\% & 0.729\% & 18.73 \\ 
$\textbf{MSE}^{\textbf{Ext}}$ & 0.143 & 0.188 & 0.663 & 0.262 & 0.439 & 0.59 & 1.062 & 1837.38 & \bf{0.7251 \%} & \bf{0.6607\%} & 19.05 \\
\botrule
\end{tabular*}
\end{sidewaystable}

\clearpage

The following observations can be made: 
\begin{itemize}
    \item The likelihood estimation method provides parameter estimates close to the simulated ones. On the other hand, $\text{MSE}^{\text{Int}}$ and $\text{MSE}^{\text{Ext}}$ methods produce estimates far from the true values, except for the $\rho$ value for the $\text{MSE}^{\text{Ext}}$, which is expected. An attempt to address the accuracy issue in the MSE method has been made by repeating the procedure multiple times. Each iteration began with the outcome of the previous optimization as its starting point. Unfortunately, this did not improve the accuracy. 
    \item Both MSE methods have an identifiability issue, which means that it is challenging to uniquely determine the exact values of the parameters, as multiple sets of parameter values result in similar MSE values.
    \item The likelihood method achieves a closer $\lVert \phi \rVert$  to the true value compared to the MSE methods. The $\text{MSE}^{\text{Int}}$ result approaches the critical regime as 0.97 is close to 1. 
    \item Both the $\text{MSE}^{\text{Int}}$ and $\text{MSE}^{\text{Ext}}$ methods provide a value of $\frac{\rho \hspace{1mm} \lVert \overline{\phi} \rVert + \lambda_0}{1 - \lVert \phi \rVert}$ that is close to the real one compared with that obtained with the likelihood method. This is due to the fact that the MSE method tries to match the expectation of the process while the likelihood method takes into account the whole distribution, not only the expectation.  
    \item Both MSE methods allow for a faster calibration compared to the likelihood maximization method. In the case where the whole distribution is not available, these two methods can estimate the order of magnitude of the expectation of the process and track its evolution over time.
    \item The values highlighted in bold represent the lowest estimated $-ln(\mathcal{L})$, $\text{MSE}^{\text{Ext}},$ and $\text{MSE}^{\text{Int}}$. As expected, the maximum likelihood estimation method has the smallest $-ln(\mathcal{L})$. The $\text{MSE}^{\text{Ext}}$ estimation method on the other hands results in smaller mean squares errors compared with the $\text{MSE}^{\text{Int}}$ method. 
\end{itemize}

\subsection{Estimation with injected value of $\rho$}

As the arrival times of external events can be observed, the parameter $\rho$ can be estimated separately and injected in the optimization problem. The optimization problem is then on: 
$$ 
(\lambda_0,\overline{m},m,\delta).
$$ 
The results are given in Table~\ref{4_param_given_rho}. The estimated value of $\rho$ is 0.19. Injecting $\rho$ does not enhance the performance of the two MSE estimation methods. {The four parameters are still misestimated}. The likelihood method is more accurate and the value of $\frac{\rho \hspace{1mm} \lVert \overline{\phi} \rVert + \lambda_0}{1 - \lVert \phi \rVert}$ remains stable and close to the actual value. The comments on the identifiability issue, the regime misestimation, the execution time and the bold metric values remain unchanged in this approach.

\clearpage

\begin{sidewaystable}[h]
\centering
\caption{Calibration results on a simulated example: simultaneous estimation of 4 parameters with given value of $\rho$}
\label{4_param_given_rho}
\begin{tabular*}{\textheight}{@{\extracolsep\fill}p{2cm}p{1cm}p{1cm}p{1cm}p{1cm}p{1cm}p{1cm}p{1cm}p{1cm}p{1cm}p{1cm}p{1cm}}
\toprule
& $\boldsymbol{\lambda_0}$ & $\boldsymbol{\overline{m}}$ & $\boldsymbol{m}$ & $\boldsymbol{\delta}$ & ${\lVert \phi \rVert}$ & $\boldsymbol{\frac{\rho \hspace{1mm} \lVert \overline{\phi} \rVert + \lambda_0}{1 - \lVert \phi \rVert}}$ & $\boldsymbol{-ln(\mathcal{L})}$ & $\frac{\textbf{MSE}^{\textbf{Ext}}}{N_{\tau}}$ & $\frac{\textbf{MSE}^{\textbf{Int}}}{N_{\tau}}$ & \textbf{Av. Exec. time in (s) on 1000 runs} \\
\midrule
\textbf{True values} & 0.6 & 0.8 & 0.5 & 1.5 & 0.33 & 1.06 & 783.55 & 0.726\% & 0.653 \% & - \\ 
\textbf{Likelihood method} & 0.62 & 0.89 & 0.53 & 1.68 & 0.31 & 1.04 & \bf{783.6} & 0.745\% & 0.737\% & 20.08 \\ 
\textbf{95 \% C.I} & [0.52,0.73] & [0.55,1.22] & [0.35,0.71] & [1.08,2.24] & [0.15,0.65] & [0.67,2.77] & - & - & - & - \\ 
$\textbf{MSE}^{\textbf{Int}}$ & 0.1 & 0.005 & 0.08 & 0.09 & 0.9 & 1.06 & 832.06 & 0.791\% & 0.718\% & 17.47 \\ 
$\textbf{MSE}^{\textbf{Ext}}$ & 0.25 & 0.08 & 0.04 & 0.05 & 0.66 & 1.06 & 828.443 & 0.725\% & 0.713\% & 18.03 \\
\botrule
\end{tabular*}
\end{sidewaystable}
\clearpage

\subsection{Estimation with injected value of $\rho \overline{m}$}

To improve the accuracy of the parameter estimates in the previous two strategies, the third approach consists of injecting the value of $\rho \overline{m}$. The optimization problem is on:
$$ 
(\lambda_0,\rho,m,\delta).
$$ 
The results are displayed in Table~\ref{4_param_given_rho_m_bar}. The value of $\overline{m}$ is deduced by dividing $\rho \overline{m}$ by the estimated value of $\rho$. Here, the injected value $\rho \overline{m} = 0.16 $. The key result of this approach is the improvement of the estimation of the $\lambda_0$ parameter and the $\lVert \phi \rVert$ norm using the $\text{MSE}^{\text{Ext}}$ method. Since the latter estimates the $\rho$ parameter accurately, injecting $\rho \overline{m}$ results in a $\lambda_0$ and $\lVert \phi \rVert$ values close to the actual ones. The other comments on $\frac{\rho \hspace{1mm} \lVert \overline{\phi} \rVert + \lambda_0}{1 - \lVert \phi \rVert}$,the identifiability issue, the regime misestimation in the $\text{MSE}^{\text{Int}}$ method, the execution time and the bold metric values remain again unchanged in this approach.

Injecting the value $\rho \overline{m}$ has improved the accuracy of the baseline intensity estimate and the regime, but it remains inadequate for estimating $\delta$ and $m$.

\clearpage
\begin{sidewaystable}[h]
\centering
\caption{Calibration results in a simulated example: simultaneous estimation of 4 parameters with given value of $\rho \overline{m}$}
\label{4_param_given_rho_m_bar}
\begin{tabular*}{\textheight}{@{\extracolsep\fill}p{2cm}p{1cm}p{1cm}p{1cm}p{1cm}p{1cm}p{1cm}p{1cm}p{1cm}p{1cm}p{1cm}p{1cm}}
\toprule
& $\boldsymbol{\lambda_0}$ & $\boldsymbol{\rho}$ & $\boldsymbol{m}$ & $\boldsymbol{\delta}$ & ${\lVert \phi \rVert}$ & $\boldsymbol{\frac{\rho \hspace{1mm} \lVert \overline{\phi} \rVert + \lambda_0}{1 - \lVert \phi \rVert}}$ & $\boldsymbol{-ln(\mathcal{L})}$ & $\frac{\textbf{MSE}^{\textbf{Ext}}}{N_{\tau}}$ & $\frac{\textbf{MSE}^{\textbf{Int}}}{N_{\tau}}$ & \textbf{Av. Exec. time in (s) on 1000 runs} \\
\midrule
\textbf{True values} & 0.6 & 0.2 & 0.5 & 1.5 & 0.33 & 1.06 & 783.55 & 0.726\% & 0.653 \% & - \\ 
\textbf{Likelihood method} & 0.63 & 0.21 & 0.61 & 1.84 & 0.33 & 1.02 & \bf{788.3} & 0.791\% & 0.718\% & 21.38 \\ 
\textbf{95 \% C.I} & [0.53,0.73] & [0.17,0.25] & [0.43,0.78] & [1.24,2.43] & [0.15,0.65] & [0.66,2.7] & - & - & - & - \\ 
$\textbf{MSE}^{\textbf{Int}}$ & 1.22E-16 & 2.18 & 1.25 & 1.4 & 0.89 & 1.06 & 1123.2 & 0.734\% & 0.694 \% & 18.53 \\ 
$\textbf{MSE}^{\textbf{Ext}}$ & 0.54 & 0.19 & 1.55 & 3.57 & 0.43 & 1.06 & 802.41 & \bf{0.738\%} & \bf{0.628 \%} & 19.01 \\
\botrule
\end{tabular*}
\end{sidewaystable}

\clearpage

\section{Choice of the  time-step  $\Delta$ in the MSE calibration method} \label{annex_step_choice_mse}

In this appendix, we explore the choice of the time-step $\Delta$ in the MSE calibration method. As presented in Section~\ref{sim_example_results}, the $MSE^{\text{ext}}$ calibration method is robust for estimating $\frac{\rho \hspace{1mm} \lVert \overline{\phi} \rVert + \lambda_0}{1 - \lVert \phi \rVert}$ quantity. Here, we display the obtained results of this ergodicity ratio from calibrations performed over various time-steps, ranging from 1 to 10. The dotted line represents the true value, and the closer the bars are to the dotted lines, the better the estimation is.

Figure~\ref{frac_mse_ext_timestep} below suggests that $\Delta$ should be selected as smaller than 4 days since the accuracy of this estimation deteriorates as $\Delta$ increases. In this context, we have opted for a compromise between 1, 2 and 3 and decided on $\Delta = 2$ days. This choice aligns with the capabilities of our high-performing server, which can handle a $\Delta$ of 2 days efficiently, whereas choosing 1 day would be more computationally demanding. Recall that the server is an Intel Weon 4310 CPU server operating at a 2.1 GHZ frequency and is equipped with 24 processors and 100 GB of memory. In comparison, my professional laptop has 4 processors and 16 GB of memory. 

\begin{figure}[H]
\centering
 \includegraphics[scale=0.45]{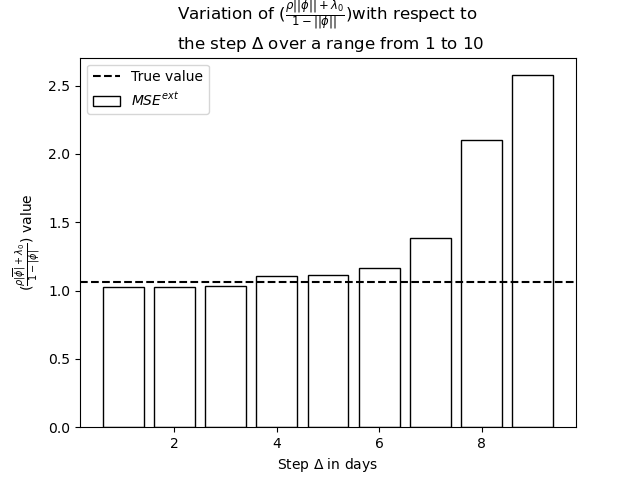}
\caption{Variation of $\frac{\rho \hspace{1mm} \lVert \overline{\phi} \rVert + \lambda_0}{1 - \lVert \phi \rVert}$ by step $\Delta$ }
\label{frac_mse_ext_timestep}
\end{figure}

\section{Computation of confidence interval in the likelihood method} \label{annex_ic_likelihood}

Recall as in (\ref{likelihood_equ2}) that: 

\begin{equation*}
     \mathcal{L}= \exp \left(-\int_s^\tau \lambda_u \mathrm{~d} u\right) \prod_{n=(N_{s}+1)}^{N_{\tau}}\hspace{-4mm}(\lambda_{t_n}) \hspace{4mm}\rho^{(\overline{N}_{\tau} - \overline{N}_{s}) } \hspace{1mm} \exp \left(-\rho\left(\overline{t}_{\overline{N}_{\tau}}-\overline{t}_{\overline{N}_s}\right)\right).
\end{equation*}

Then by taking the logarithm and using the antiderivative of an exponential function, we have the following formula:
\begin{align*}
    \ln (\mathcal{L}) & = -\int_s^\tau \lambda_u \mathrm{~d} u + \sum_{n = (N_s+1)}^{N_{\tau}} \ln(\lambda_{t_n}) + \ln(\rho) (\overline{N}_{\tau} - \overline{N}_{s} ) - \rho (\overline{t}_{\overline{N}_{\tau}} - \overline{t}_{\overline{N}_{s}})
    \\
    &  = -\lambda_0 (\tau - s) - \sum_{s<t_i<\tau} \frac{m}{\delta} (1 - e^{-\delta (\tau - t_i)}) - \sum_{s<\overline{t}_k<t} \frac{\overline{m}}{\delta} (1 - e^{-\delta (\tau - \overline{t}_k)}) \\ & + \sum_{n = (N_s+1)}^{N_{\tau}} \ln(\lambda_{t_n}) + \ln(\rho) (\overline{N}_{\tau} - \overline{N}_{s} ) - \rho (\overline{t}_{\overline{N}_{\tau}} - \overline{t}_{\overline{N}_{s}}).
\end{align*}
Since $\lambda_{t_n} = \lambda_0 + \sum_{t_i<t_n}me^{-\delta (t_n - t_i)} + \sum_{\overline{t}_k<t_n}\overline{m}e^{-\delta (t_n - \overline{t}_k)}  $ and by replacing $\lambda_{t_n}$ by its value, we have:
\begin{align*}
   \ln (\mathcal{L}) & =  -\lambda_0 (\tau - s) - \frac{m}{\delta} (N_{\tau} - N_s) + \frac{m}{\delta}\sum_{s<t_i<t} e^{-\delta (\tau - t_i)} -  \frac{\overline{m}}{\delta} (\overline{N}_{\tau} - \overline{N}_s) 
   \\ 
   +&\frac{\overline{m}}{\delta} \sum_{s<\overline{t}_k<t} e^{-\delta (\tau - \overline{t}_k)} 
   \\ 
   + & \sum_{n = (N_s+1)}^{N_{\tau}} \left[ \ln(\lambda_0) + \ln(m) + \ln(\sum_{t_i<t_n}e^{-\delta (t_n - t_i)}) + \ln(\overline{m}) + \ln( \sum_{\overline{t}_k<t_n}e^{-\delta (t_n - \overline{t}_k)}) \right ]
   \\ 
   + & \ln(\rho) (\overline{N}_{\tau} - \overline{N}_{s} ) - \rho (\overline{t}_{\overline{N}_{\tau}} - \overline{t}_{\overline{N}_{s}}).
\end{align*}
Thus, by rearranging the different terms: 
\begin{align*}
   \ln (\mathcal{L})  = &  -\lambda_0 (\tau - s) - \frac{m}{\delta} (N_{\tau} - N_s) + \frac{m}{\delta}\sum_{s<t_i<t} e^{-\delta (\tau - t_i)} -  \frac{\overline{m}}{\delta} (\overline{N}_{\tau} - \overline{N}_s)  \\ +
    & \frac{\overline{m}}{\delta} \sum_{s<\overline{t}_k<t} e^{-\delta (\tau - \overline{t}_k)} 
    + (N_{\tau} - N_s) \left\{\ln(\lambda_0) + \ln(m) + \ln(\overline{m}) \right\} 
    \\ 
    +  &  
   \sum_{n = (N_s+1)}^{N_{\tau}} \left[ \ln(\sum_{t_i<t_n}e^{-\delta (t_n - t_i)})  + \ln( \sum_{\overline{t}_k<t_n}e^{-\delta (t_n - \overline{t}_k)}) \right]  \\
   + & \ln(\rho) (\overline{N}_{\tau} - \overline{N}_{s} ) - \rho (\overline{t}_{\overline{N}_{\tau}} - \overline{t}_{\overline{N}_{s}}).
\end{align*}
For the sake of notation conciseness, let us denote $f = \ln (\mathcal{L})$. The Hessian matrix, denoted as $H$ is defined as:

\begin{equation*}
 H =   \left(\begin{array}{lllll}
\frac{\partial^2 f}{\partial \lambda_0^2} & \frac{\partial^2 f}{\partial \lambda_0 \partial \rho} &  \frac{\partial^2 f}{\partial \lambda_0 \partial \overline{m}  } & \frac{\partial^2 f}{\partial \lambda_0 \partial m} & \frac{\partial^2 f}{\partial \lambda_0 \partial \delta}  \\
\frac{\partial^2 f}{\partial \rho \partial \lambda_0} &  \frac{\partial^2 f}{\partial \rho^2} & \frac{\partial^2 f}{\partial \rho \partial \overline{m}} &\frac{\partial^2 f}{\partial \rho \partial m} & \frac{\partial^2 f}{\partial \rho \partial \delta} \\ 
\frac{\partial^2 f}{\partial \overline{m} \partial \lambda_0} & \frac{\partial^2 f}{\partial \overline{m} \partial \rho} & \frac{\partial^2 f}{\partial \overline{m}^2} & \frac{\partial^2 f}{\partial \overline{m} \partial m} & \frac{\partial^2 f}{\partial \overline{m} \partial \delta} \\ 
\frac{\partial^2 f}{\partial m \partial \lambda_0} & \frac{\partial^2 f}{\partial m \partial \rho} & \frac{\partial^2 f}{\partial m \partial \overline{m}} & \frac{\partial^2 f}{\partial m^2} & \frac{\partial^2 f}{\partial m \partial \delta} \\ 
\frac{\partial^2 f}{\partial \delta \partial \lambda_0} & \frac{\partial^2 f}{\partial \delta\partial \rho} & \frac{\partial^2 f}{\partial \delta\partial \overline{m}} & \frac{\partial^2 f}{\partial \delta \partial m} & \frac{\partial^2 f}{\partial \delta^2} 
\end{array}\right).
\end{equation*}

By computing the different terms: 

\begin{equation*}
{\sixpt 
 \hspace{-1.7cm} H =   \left(\begin{array}{lllll}
\frac{-(N_{\tau} - N_s)}{\lambda_0^2}  & 0 & 0 & 0 &0  \\ 
0 &\frac{-(N_{\tau} - N_s)}{\rho^2}  & 0&0 &0 \\ 
 0& 0& \frac{-(N_{\tau} - N_s)}{\overline{m}^2} & 0 & \frac{(\overline{N}_{\tau} - \overline{N}_{s} ) - (1+ \delta^2) \sum_{s<\overline{t}_k<\tau}e^{-\delta (\tau - \overline{t}_k)} }{\delta^2} \\ 
 0&0 &0 & \frac{-(N_{\tau} - N_s)}{m^2} & \frac{(N_{\tau} - N_{s} ) - (1+ \delta^2) \sum_{s<t_i<\tau}e^{-\delta (\tau - t_i)} }{\delta^2} \\ 
 0 & 0 & \frac{(\overline{N}_{\tau} - \overline{N}_{s} ) - (1+ \delta^2) \sum_{s<\overline{t}_k<\tau}e^{-\delta (\tau - \overline{t}_k)} }{\delta^2}  & \frac{(N_{\tau} - N_{s} ) - (1+ \delta^2) \sum_{s<t_i<\tau}e^{-\delta (\tau - t_i)} }{\delta^2}  & \frac{\partial^2 f}{\partial \delta^2}
\end{array}\right)
}
\end{equation*}
with: 
{\small
\begin{align*}
\frac{\partial^2 f}{\partial \delta^2} & = \frac{-2m (N_{\tau} - N_{s})}{\delta^3} + \frac{2m (N_{\tau} - N_{s})}{\delta^3} \sum_{s<t_i<\tau} e^{-\delta (\tau - t_i)}\\& + \frac{m}{\delta} \sum_{s<t_i<\tau}e^{-\delta (\tau - t_i)} + m \delta \sum_{s<t_i<\tau}e^{-\delta (\tau - t_i)}  \\&  - \frac{2 \overline{m} (\overline{N}_{\tau} - \overline{N}_{s})}{\delta^3} + \frac{2\overline{m} (\overline{N}_{\tau} - \overline{N}_{s})}{\delta^3} \sum_{s<\overline{t}_k<\tau} e^{-\delta (\tau - \overline{t}_k)} \\& + \frac{\overline{m}}{\delta} \sum_{s<\overline{t}_k<\tau}e^{-\delta (\tau - \overline{t}_i)} + \overline{m} \delta \sum_{s<\overline{t}_k<\tau}e^{-\delta (\tau - \overline{t}_i)}-  2 (N_{\tau} - N_s).  
\end{align*}
}

The variance-covariance matrix of the parameters is then obtained by calculating the opposite of the inverse of the Hessian matrix $H$ evaluated at the maximum likelihood estimates. For a given parameter, the square root of its diagonal element in the variance-covariance matrix gives its standard error, which is used to calculate confidence intervals. This is achieved using Numpy. The 95 \% confidence interval are then calculated using a normality assumption. 

\end{appendices}

\bibliography{sn-bibliography}

\begin{thebibliography}{27}
\providecommand{\natexlab}[1]{#1}
\providecommand{\url}[1]{{#1}}
\providecommand{\urlprefix}{URL }
\providecommand{\doi}[1]{\url{https://doi.org/#1}}
\providecommand{\eprint}[2][]{\url{#2}}
 \bibcommenthead

\bibitem[{Awiszus et~al(2023)Awiszus, Knispel, Penner, Svindland, Vo{\ss}, and
  Weber}]{awiszus2023modeling}
Awiszus K, Knispel T, Penner I, et~al (2023) Modeling and pricing cyber
  insurance: idiosyncratic, systematic, and systemic risks. European Actuarial
  Journal pp 1--53

\bibitem[{Baldwin et~al(2017)Baldwin, Gheyas, Ioannidis, Pym, and
  Williams}]{baldwin2017contagion}
Baldwin A, Gheyas I, Ioannidis C, et~al (2017) Contagion in cyber security
  attacks. Journal of the Operational Research Society 68(7):780--791

\bibitem[{Bessy-Roland et~al(2020)Bessy-Roland, Boumezoued, and
  Hillairet}]{Multi-Variate_HawkesBessy}
Bessy-Roland Y, Boumezoued A, Hillairet C (2020) Multivariate {H}awkes process
  for cyber insurance. Annals of Actuarial Science, 15(1), 14-39
  \urlprefix\url{https://hal.archives-ouvertes.fr/hal-02546343/documenta}

\bibitem[{Biener et~al(2015)Biener, Eling, and Wirfs}]{biener2015insurability}
Biener C, Eling M, Wirfs JH (2015) Insurability of cyber risk: An empirical
  analysis. The Geneva Papers on Risk and Insurance-Issues and Practice
  40(1):131--158

\bibitem[{Boumezoued(2016)}]{boumezoued2016population}
Boumezoued A (2016) Population viewpoint on {H}awkes processes. Advances in
  Applied Probability 48(2):463--480

\bibitem[{Boyd et~al(2023)Boyd, Chang, Mandt, and Smyth}]{boyd2023inference}
Boyd A, Chang Y, Mandt S, et~al (2023) Inference for mark-censored temporal
  point processes. Uncertainty in Artificial Intelligence pp 226--236

\bibitem[{Brachetta et~al(2022)Brachetta, Callegaro, Ceci, and
  Sgarra}]{brachetta2022optimal}
Brachetta M, Callegaro G, Ceci C, et~al (2022) Optimal reinsurance via {BSDEs}
  in a partially observable contagion model with jump clusters. arXiv preprint
  arXiv:220705489 To appear in Finance and Statistics:
  \url{https://wwwspringercom/journal/780/updates/19991928 }

\bibitem[{Chen et~al(2021)Chen, Dassios, Kuan, Lim, Qu, Surya, and
  Zhao}]{dassiostwo}
Chen Z, Dassios A, Kuan V, et~al (2021) A two-phase dynamic contagion model for
  {COVID-19}. Results in Physics 26:104264

\bibitem[{CISA(2023 :
  https://www.cisa.gov/known-exploited-vulnerabilities-catalog)}]{KEV}
CISA (2023 : https://www.cisa.gov/known-exploited-vulnerabilities-catalog) Kev.
  \urlprefix\url{https://www.cisa.gov/known-exploited-vulnerabilities-catalog}

\bibitem[{Dassios and Zhao(2011)}]{dassios2011dynamic}
Dassios A, Zhao H (2011) A dynamic contagion process. Advances in applied
  probability 43(3):814--846

\bibitem[{Eling and Schnell(2020)}]{eling2020capital}
Eling M, Schnell W (2020) Capital requirements for cyber risk and cyber risk
  insurance: {A}n analysis of {S}olvency {II}, the {US} risk-based capital
  standards, and the {S}wiss {S}olvency {T}est. North American Actuarial
  Journal 24(3):370--392

\bibitem[{Fahrenwaldt et~al(2018)Fahrenwaldt, Weber, and
  Weske}]{fahrenwaldt2018pricing}
Fahrenwaldt MA, Weber S, Weske K (2018) Pricing of cyber insurance contracts in
  a network model. ASTIN Bulletin: The Journal of the IAA 48(3):1175--1218

\bibitem[{Farkas et~al(2021)Farkas, Lopez, and Thomas}]{farkas2021cyber}
Farkas S, Lopez O, Thomas M (2021) Cyber claim analysis using {G}eneralized
  {P}areto regression trees with applications to insurance. Insurance:
  Mathematics and Economics 98:92--105

\bibitem[{Garcia and Kurtz(2008)}]{garcia2008spatial}
Garcia NL, Kurtz TG (2008) Spatial point processes and the projection method.
  In and Out of Equilibrium 2 pp 271--298

\bibitem[{Hardiman and Bouchaud(2014)}]{hardiman2014}
Hardiman SJ, Bouchaud JP (2014) Branching-ratio approximation for the
  self-exciting {H}awkes process. Physical Review E 90(6):062807

\bibitem[{Hillairet and Lopez(2021)}]{hillairet2021propagation}
Hillairet C, Lopez O (2021) Propagation of cyber incidents in an insurance
  portfolio: counting processes combined with compartmental epidemiological
  models. Scandinavian Actuarial Journal pp 1--24

\bibitem[{NVD(2023 : https://nvd.nist.gov/general)}]{NVD}
NVD (2023 : https://nvd.nist.gov/general) Nvd.
  \urlprefix\url{https://nvd.nist.gov/general}

\bibitem[{Ogata(1981)}]{ogata1981lewis}
Ogata Y (1981) On {L}ewis' simulation method for point processes. IEEE
  transactions on information theory 27(1):23--31

\bibitem[{Passeri(2023 : https://www.hackmageddon.com/)}]{hackmageddon}
Passeri P (2023 : https://www.hackmageddon.com/) Hackmageddon.
  \urlprefix\url{https://www.hackmageddon.com/}

\bibitem[{Peng et~al(2017)Peng, Xu, Xu, and Hu}]{peng2017modeling}
Peng C, Xu M, Xu S, et~al (2017) Modeling and predicting extreme cyber attack
  rates via marked point processes. Journal of Applied Statistics
  44(14):2534--2563

\bibitem[{Rambaldi et~al(2015)Rambaldi, Pennesi, and
  Lillo}]{rambaldi2015modeling}
Rambaldi M, Pennesi P, Lillo F (2015) Modeling foreign exchange market activity
  around macroeconomic news: {H}awkes-process approach. Physical Review E
  91(1):012819

\bibitem[{Rambaldi et~al(2018)Rambaldi, Filimonov, and
  Lillo}]{rambaldi2018detection}
Rambaldi M, Filimonov V, Lillo F (2018) Detection of intensity bursts using
  {H}awkes processes: An application to high-frequency financial data. Physical
  Review E 97(3):032318

\bibitem[{Shelton et~al(2018)Shelton, Qin, and Shetty}]{shelton2018hawkes}
Shelton C, Qin Z, Shetty C (2018) {H}awkes process inference with missing data.
  Proceedings of the AAAI Conference on Artificial Intelligence 32(1)

\bibitem[{Stabile and Torrisi(2010)}]{stabile2010risk}
Stabile G, Torrisi GL (2010) Risk processes with non-stationary {H}awkes claims
  arrivals. Methodology and Computing in Applied Probability 12:415--429

\bibitem[{Zeller and Scherer(2022)}]{zeller2022comprehensive}
Zeller G, Scherer M (2022) A comprehensive model for cyber risk based on marked
  point processes and its application to insurance. European Actuarial Journal
  12(1):33--85

\bibitem[{Zeller and Scherer(2023)}]{zeller2023accumulation}
Zeller G, Scherer MA (2023) Is accumulation risk in cyber systematically
  underestimated? Available at SSRN 4353098

\bibitem[{Zhang et~al(2015)Zhang, Ou, and Caragea}]{zhang2015predicting}
Zhang S, Ou X, Caragea D (2015) Predicting cyber risks through national
  vulnerability database. Information Security Journal: A Global Perspective
  24(4-6):194--206

\end{thebibliography}

\end{document}